\documentclass{amsart}

\usepackage{ae,aecompl}
\usepackage{esint}
\usepackage{graphicx}
\usepackage[all,cmtip]{xy}
\usepackage{amsmath, amscd}
\usepackage{booktabs}
\usepackage{amsthm}
\usepackage{amssymb}
\usepackage{amsfonts}
\usepackage{qsymbols}
\usepackage{latexsym}
\usepackage{mathrsfs}
\usepackage{cite}
\usepackage{color}
\usepackage{url}
\usepackage{enumerate}
\usepackage{inputenc,todonotes,esint}
\usepackage{verbatim}
\usepackage[draft=false, colorlinks=true]{hyperref}
\usepackage[new]{old-arrows}

\allowdisplaybreaks

\newcommand\ang[1]{\langle #1 \rangle}

\newcommand {\bmo}{\mathrm{bmo}}
\newcommand {\BMO}{\mathrm{BMO}}

\newcommand {\C}{{\mathbb C}}

\newcommand {\ud}{\mathrm{d}}

\newcommand {\veps}{\varepsilon}

\newcommand {\F}{\mathcal{F}}

\newcommand {\HT}{\mathcal{H}}
\newcommand {\Hp}{\mathcal{H}^{p}_{FIO}(\Rn)}

\newcommand {\Hps}{\mathcal{H}^{s,p}_{FIO}(\Rn)}

\newcommand {\ind}{\mathbf{1}}

\newcommand {\J}{\mathcal{J}}
\newcommand {\ka}{\kappa}

\newcommand {\la}{\lambda}
\newcommand {\rb}{\rangle}
\newcommand {\lb}{\langle}
\newcommand {\La}{\mathcal{L}}
\newcommand {\loc}{\mathrm{loc}}
\newcommand {\Lqp}{L^{q}_{\w}L^{p}_{x}L^{2}_{\sigma}(\Spp)}

\newcommand {\N}{{\mathbb N}}

\newcommand {\ph}{\varphi}

\newcommand {\R}{\mathbb R}
\newcommand {\Rtwo}{\mathbb{R}^{2}}
\newcommand {\RR}{\mathbb R}

\newcommand {\Rn}{\mathbb{R}^{n}}
\newcommand {\Rnone}{\mathbb{R}^{n+1}}

\newcommand {\supp}{\mathrm{supp}}

\newcommand {\Sp}{S^{*}\Rn}
\newcommand {\Spp}{S^{*}_{+}\Rn}

\newcommand {\Sw}{\mathcal{S}}

\newcommand {\w}{\omega}

\newcommand {\wh}{\widehat}
\newcommand {\wt}{\widetilde}

\newcommand {\Z}{\mathbb Z}

\newcommand {\vanish}[1]{\relax}

\newcommand{\fun}{\mathcal{L}_{W,s}}
\newcommand{\funpqs}{\fun^{q,p}(\Rn)}

\newcommand{\funpqsone}{\fun^{q,p}(\R^{n+1})}
\newcommand{\funpqzero}{\mathcal{L}_{W,0}^{q,p}(\R^{n})}

	\newcommand{\x}{\times}

\DeclareFontFamily{U}{mathx}{\hyphenchar\font45}
\DeclareFontShape{U}{mathx}{m}{n}{
      <5> <6> <7> <8> <9> <10>
      <10.95> <12> <14.4> <17.28> <20.74> <24.88>
      mathx10
      }{}
\DeclareSymbolFont{mathx}{U}{mathx}{m}{n}
\DeclareFontSubstitution{U}{mathx}{m}{n}
\DeclareMathAccent{\widecheck}{0}{mathx}{"71}

\newcommand\ldg{\mathrm{ldg}}

\DeclareMathOperator{\Real}{Re}
\DeclareMathOperator{\Imag}{Im}

\newtheorem{theorem}{Theorem}[section]
\newtheorem{lemma}[theorem]{Lemma}
\newtheorem{proposition}[theorem]{Proposition}
\newtheorem{corollary}[theorem]{Corollary}

\theoremstyle{definition}
\newtheorem{definition}[theorem]{Definition}
\newtheorem{remark}[theorem]{Remark}

\numberwithin{equation}{section}
\protected\def\ignorethis#1\endignorethis{}
\let\endignorethis\relax

\title{Function spaces for decoupling}

\author{Andrew Hassell}
\address{Mathematical Sciences Institute\\
and France-Australia Mathematical Sciences and Interactions\\
ANU-CNRS International Research Laboratory\\
Australian National University \\
Ngunnawal and Ngambri Country \\
Canberra ACT 0200, Australia
}
\email{Andrew.Hassell@anu.edu.au}

\author{Pierre Portal}
\address{Mathematical Sciences Institute\\
and France-Australia Mathematical Sciences and Interactions\\
ANU-CNRS International Research Laboratory\\
Australian National University \\
Ngunnawal and Ngambri Country \\
Canberra ACT 0200, Australia}
\email{Pierre.Portal@anu.edu.au}

\author{Jan Rozendaal}
\address{Institute of Mathematics, Polish Academy of Sciences\\
Ul.~\'{S}niadeckich 8\\
00-656 Warsaw\\
Poland}
\email{jrozendaal@impan.pl}

\author{Po--Lam Yung}
\address{Mathematical Sciences Institute\\
and France-Australia Mathematical Sciences and Interactions\\
ANU-CNRS International Research Laboratory\\
Australian National University \\
Ngunnawal and Ngambri Country \\
Canberra ACT 0200, Australia}
\email{PoLam.Yung@anu.edu.au, plyung@math.cuhk.edu.hk}



\thanks{This research was funded in part by the National Science Center, Poland, grant 2021/43/D/ST1/00667. The third author is partially supported by NCN grant UMO-2023/49/B/ST1/01961. Yung was partially supported by Future Fellowship FT200100399 from the Australian Research Council.}

\begin{document}

\begin{abstract}
We introduce new function spaces $\funpqs$ that yield a natural reformulation of the $\ell^{q}L^{p}$ decoupling inequalities for the sphere and the light cone. These spaces are invariant under the Euclidean half-wave propagators, but not under all Fourier integral operators unless $p=q$, in which case they coincide with the Hardy spaces for Fourier integral operators. We use these spaces to obtain improvements of the classical fractional integration theorem and local smoothing estimates.
\end{abstract}
	
\maketitle

\section{Introduction}\label{sec:intro}

This article aims to provide a bridge between the theory of decoupling inequalities in Fourier analysis on the one side, and that of invariant spaces for wave propagators and more general Fourier integral operators on the other. 

\subsection{Setting}\label{subsec:setting}

In recent years, the theory of decoupling inequalities has developed into a highly active part of harmonic analysis. Starting with seminal work of Bourgain and Demeter \cite{Bourgain-Demeter15}, which in turn has its origins in ideas of Wolff \cite{Wolff00}, decoupling inequalities have found applications both in number theory and to partial differential equations, leading for example to a proof of the main conjecture in Vinogradov's mean value theorem \cite{BoDeGu16,GuLiKuYuZo21}, the resolution of Carleson's problem on almost everywhere convergence for the Schr\"{o}dinger equation \cite{DuGuLi17,Du-Zhang19,Hickman23}, and significant progress towards the local smoothing conjecture for the Euclidean wave equation \cite{Bourgain-Demeter15,GuWaZh20,BeHiSo21}. 

In a separate development, a scale of invariant spaces $(\Hp)_{1\leq p\leq \infty}$ for wave propagators and more general Fourier integral operators (FIOs) was introduced in \cite{HaPoRo20}, extending a pioneering construction due to Smith \cite{Smith98a} for $p=1$. These spaces satisfy the following Sobolev embeddings into the $L^{p}$ scale, for $1<p<\infty$:
\begin{equation}\label{eq:Sobolevintro}
W^{s(p),p}(\Rn)\subseteq \Hp\subseteq W^{-s(p),p}(\Rn),
\end{equation}
with the appropriate modifications involving the local Hardy space $\HT^{1}(\Rn)$ and $\bmo(\Rn)$ for $p=1$ and $p=\infty$, respectively. Here and throughout, we write
\begin{equation}\label{eq:sp}
s(p):=\frac{n-1}{2}\Big|\frac{1}{2}-\frac{1}{p}\Big|
\end{equation}
for $1\leq p\leq \infty$. These embeddings allow one to recover the optimal $L^{p}(\Rn)$ regularity of Fourier integral operators due to Seeger, Sogge and Stein \cite{SeSoSt91}, and the accompanying sharp fixed-time $L^{p}(\Rn)$ regularity for wave equations. 
Moreover, the Hardy spaces for Fourier integral operators have been applied to wave equations with rough coefficients \cite{Frey-Portal20,Hassell-Rozendaal23} and to nonlinear wave equations \cite{Rozendaal-Schippa23,LiRoSoYa24} in ways that do not seem possible when working on $L^{p}(\Rn)$ directly.

Recently, it was realized that invariant spaces for wave propagators are also connected to decoupling inequalities. First, in \cite{Rozendaal22b} it was shown that, loosely speaking, $\Hp$ is the largest space of initial data for which one can obtain local smoothing estimates for the wave equation when relying on the $\ell^{p}$ decoupling inequality for the light cone. This observation yields estimates that improve upon those in the local smoothing conjecture for $p\geq \frac{2(n+1)}{n-1}$. In \cite{Rozendaal-Schippa23,LiRoSoYa24}, this connection was extended to other decoupling inequalities, thereby providing improved local smoothing estimates for both constant-coefficient and variable-coefficient wave equations.

However, this development only concerns applications of decoupling inequalities to wave equations, and it shows that the Hardy spaces for Fourier integral operators are natural spaces of initial data in that regard. In the present article, we 
develop the connection between decoupling inequalities and invariant spaces for wave propagators further, by introducing function spaces that capture the decoupling inequalities for the sphere and the light cone in full generality. In particular, these are not just suitable spaces of initial data for the Euclidean wave equation, but also for its solutions. 
Moreover, by working with general $\ell^{q}$ decoupling inequalities instead of merely considering the case where $q=p$, we obtain improved fractional integration theorems, as well as new local smoothing estimates and well-posedness results for nonlinear wave equations.

\subsection{Main results}\label{subsec:mainresults}

We introduce a collection of function spaces $\funpqs$, for $p,q\in[1,\infty)$ and $s\in\R$. Here $p$ is a spatial integrability parameter, $s$ a smoothness parameter, $q$ measures integrability with respect to angular localizations, and $W$ is a wave packet transform that implicitly connects the spatial and angular parameters.  More precisely, cf.~Definition \ref{def:spaces}, for $1<p<\infty$ the space $\funpqs$ consists of all $f\in\Sw'(\Rn)$ such that
\begin{equation}\label{eq:normintro}
\|f\|_{\funpqs}=\|\rho(D)f\|_{L^{p}(\Rn)}+\Big(\int_{S^{n-1}}\|\ph_{\w}(D)f\|_{W^{s,p}(\Rn)}^{q}\ud \w\Big)^{1/q}<\infty.
\end{equation}
Here $\rho(D)$ is a relatively unimportant low-frequency cut-off, and the $\ph_{\w}(D)$, for $\w$ in the unit sphere $S^{n-1}$, are Fourier multipliers that localize to a paraboloid in the direction of $\w$ (see Section \ref{subsec:packets}). 
For $p=1$ one replaces the Sobolev space $W^{s,p}(\Rn)$ in \eqref{eq:normintro} by the Hardy--Sobolev space $\HT^{s,1}(\Rn) =  (1-\Delta)^{-s/2}\HT^{1}(\Rn)$.

\subsubsection{Boundedness of Fourier integral operators}

By combining parabolic frequency localizations with the implicit Littlewood--Paley decomposition of $L^{p}(\Rn)$, \eqref{eq:normintro} involves a dyadic-parabolic decomposition which goes back to  \cite{Fefferman73b} and which plays a key role in the proof of the optimal $L^{p}$ regularity of FIOs in \cite{SeSoSt91}. On a more intrinsic level, angular localization is one of the fundamental tenets of microlocal analysis, and it allows one to deal with the phenomenon of propagation of singularities that is inherent to hyperbolic equations and FIOs. Accordingly, $\funpqzero$ coincides with $\Hp$ when $p=q$, and as such it is invariant under FIOs of order zero. Our first main result shows that this invariance extends to general $p,q\in[1,\infty)$ for some, but not all, FIOs.

\begin{theorem}\label{thm:invarianceintro}
Let $p,q\in[1,\infty)$ and $s\in\R$. Then $e^{it\sqrt{-\Delta}}:\funpqs\to\funpqs$ is bounded for all $t\in\R$.

On the other hand, there exists a compactly supported Fourier integral operator $T$ such that $T:\funpqzero\to\funpqzero$ is not bounded if $p\neq q$.
\end{theorem}

The first part of Theorem \ref{thm:invarianceintro} is contained in Theorem \ref{thm:FIObdd}, and the second part is proved in Section \ref{subsec:unboundedop}. 

One can choose the operator $T$ in Theorem \ref{thm:invarianceintro} 
to be a simple change of coordinates, independent of $p$ and $q$, although the same holds if $T=e^{it\sqrt{L}}$ for $t\neq 0$ and $L$ a suitable variable-coefficient second order differential operator, cf.~Remark \ref{rem:wavenotbdd}. In fact, the unboundedness of FIOs on $\funpqs$ for $p\neq q$ appears to be rather generic in the case of variable-coefficient propagation of singularities.

\subsubsection{Sobolev embeddings}

Although the spaces $\funpqs$ do not have the same invariance properties as $\Hp$ if $p\neq q$, several key results persist, such as an extension of \eqref{eq:Sobolevintro}.  

\begin{theorem}\label{thm:Sobolevintro}
Let $p,q,r\in(1,\infty)$ and $s\in\R$. Then the following statements hold.
\begin{enumerate}
\item\label{it:Sobolevintro1} If $p\leq q\leq 2$ or $2\leq q\leq p$, then
\[
W^{s(p)+s,p}(\Rn)\subseteq \mathcal{L}_{W,s} ^{q,p}(\Rn)\subseteq W^{-s(p)+s,p}(\Rn);
\]
\item\label{it:Sobolevintro2} If $p\leq r$, then
$
 \mathcal{L}_{W,s+\frac{n+1}{2}(\frac{1}{p}-\frac{1}{r})} ^{q,p}(\Rn)\subseteq \mathcal{L}_{W,s} ^{q,r}(\Rn).
$
\end{enumerate}
\end{theorem}

Theorem \ref{thm:Sobolev2} contains \eqref{it:Sobolevintro1}, while \eqref{it:Sobolevintro2} can be found in Theorem \ref{thm:Sobolev1}. The statements also hold for $p=1$, upon replacing $W^{s,p}(\Rn)$ by the Sobolev space $\HT^{s,1}(\Rn)$ over the local Hardy space $\HT^{1}(\Rn)$. 

Due to Theorem \ref{thm:invarianceintro}, $\Hp$ seems to be a more natural space than $\funpqs$ for the analysis of variable-coefficient wave equations, at least if $p\neq q$. On the other hand, Theorem \ref{thm:Sobolevintro} indicates that the case $q=2$ is of particular interest for Sobolev embeddings. Indeed, by combining \eqref{it:Sobolevintro1} and \eqref{it:Sobolevintro2}, one obtains the following continuous inclusions for $p\leq 2\leq r$:
\[
W^{n(\frac{1}{p}-\frac{1}{r}),p}(\Rn)\subseteq \mathcal{L}_{W,n(\frac{1}{p}-\frac{1}{r})-s(p)} ^{2,p}(\Rn) \subseteq \mathcal{L}_{W,\frac{n-1}{2}(\frac{1}{p}-\frac{1}{r})-s(p)} ^{2,r}(\Rn)\subseteq L^{r}(\Rn).
\]
Given that the exponents in \eqref{it:Sobolevintro1} are sharp, these embeddings constitute a strict improvement of the classical fractional integration theorem $W^{n(\frac{1}{p}-\frac{1}{r}),p}(\Rn)\subseteq L^{r}(\Rn)$ if $(p,r)\neq (2,2)$, with an improvement of up to $2s(p)$ derivatives for suitable functions in 
$\mathcal{L}_{W,n(\frac{1}{p}-\frac{1}{r})-s(p)} ^{2,p}(\Rn)$ (see also Remark \ref{rem:fracintimproved}).

For $p,r\leq 2$ or $p,q\geq 2$, Theorem \ref{thm:Sobolevintro} complements the classical fractional integration theorem, in the sense that neither statement implies the other. Nonetheless, in this case we obtain strict improvements of classical mapping properties for FIOs of negative order, cf.~Corollary \ref{cor:fracFIO} and Remark \ref{rem:fracFIO}. 


\subsubsection{Decoupling}

We have discussed some of the basic properties of the spaces $\funpqs$, and how they relate to those of $\Hp$. On the other hand, it is not immediately clear how these spaces relate to the theory of Fourier decoupling. For example, decoupling inequalities are typically formulated using a discrete decomposition of functions that have Fourier support in a fixed compact set, whereas \eqref{eq:normintro} is only of interest at high frequencies. However, after rescaling and observing that the continuous decomposition in \eqref{eq:normintro} is equivalent to a discrete one on dyadic frequency annuli, one can reinterpret the $\ell^{q}$ decoupling inequalities for the sphere and the cone as an improvement over the Sobolev embeddings from Theorem \ref{thm:Sobolevintro} \eqref{it:Sobolevintro1}, for functions with highly localized frequency support. 

More precisely, for each $R\geq 2$, let $V_{R}\subseteq S^{n-1}$ be a maximal collection of unit vectors satisfying $|\nu-\nu'|\geq R^{-1/2}$ for all $\nu,\nu'\in V_{R}$, and let $(\chi_{\nu})_{\nu\in V_{R}}$ be an associated partition of unity of functions homogeneous of degree $0$ (see Section \ref{subsec:discrete} for details). Set
\begin{equation}\label{eq:alphapintro}
\alpha(p) := \begin{cases}
s(p)-\frac{1}{p} &\quad \text{for }\frac{2(n+1)}{n-1}\leq p <\infty,\\
0 &\quad \text{for }2<p \leq \frac{2(n+1)}{n-1}.
\end{cases}
\end{equation}
The following theorem connects the spaces $\funpqs$ to decoupling theory. 

\begin{theorem}\label{thm:decoupleintro}
Let $p,q\in[1,\infty)$, $s\in\R$ and $\veps>0$. Then there exists a $C>0$ such that the following statements hold for all $f\in \funpqs$ and $R\geq 2$.
\begin{enumerate}
\item\label{it:decoupleintro1} If $\supp(\wh{f}\,)\subseteq\{\xi\in\Rn\mid R/2\leq |\xi|\leq 2R\}$, then
\[
\frac{1}{C}\|f\|_{\funpqs} \leq  \Big(\sum_{\nu\in V_{R}}\|\chi_{\nu}(D)f\|_{\funpqs}^{q}\Big)^{1/q}\leq C\|f\|_{\funpqs}.
\]
\item\label{it:decoupleintro2} If $\supp(\wh{f}\,)\subseteq \{\xi\in\Rn\mid R-1\leq |\xi|\leq R+1\}$ and $q\geq 2$, then $f\in W^{s-\alpha(p)-\veps,p}(\Rn)$ and
\begin{equation}\label{eq:decoupleintro}
\|f\|_{W^{s-\alpha(p)-\veps,p}(\Rn)}\leq C\|f\|_{\funpqs}.
\end{equation}
\end{enumerate}
\end{theorem}

Part \eqref{it:decoupleintro1} is contained in Proposition \ref{prop:discrete}, and part \eqref{it:decoupleintro2} in Corollary \ref{cor:decouplesphere}.

The analogue of \eqref{it:decoupleintro1} for $L^{p}(\Rn)$ does not hold, and in fact it is the failure of such a two-sided inequality which motivates decoupling theory. By contrast, Theorem \ref{thm:decoupleintro} shows that the $\funpqs$ norm is invariant under decoupling into the dyadic-parabolic pieces that arise when intersecting a dyadic annulus with the cones containing the support of the $\chi_{\nu}$.

Part \eqref{it:decoupleintro2} of Theorem \ref{thm:decoupleintro} is equivalent to the $\ell^{q}L^{p}$ decoupling inequality for the sphere (see Remark \ref{rem:decouplingcompare}). Given that $\alpha(p)<s(p)$ for all $2<p<\infty$, \eqref{eq:decoupleintro} constitutes an improvement over the second embedding in Theorem \ref{thm:Sobolevintro} \eqref{it:Sobolevintro1}, which is sharp for general $f\in \funpqs$ with frequency support in a dyadic annulus, assuming a more restrictive condition on the frequency support of $f$.

Corollary \ref{cor:decouplecone} contains a similar reformulation of the $\ell^{q}L^{p}$ decoupling inequality for the light cone. In this case, one obtains improvements of the Sobolev embeddings for $\funpqsone$ for functions that have frequency support contained in the intersection of a dyadic annulus and a unit neighborhood of the light cone.

\subsubsection{Regularity for wave equations}


As was done in \cite{Rozendaal22b,Rozendaal-Schippa23} for related function spaces, we obtain local smoothing estimates for the Euclidean wave equation using $\funpqs$ as our space of initial data. In this manner, when combined with the fractional integration result for $\fun^{2,p}(\Rn)$ in Theorem \ref{thm:Sobolevintro}, one simultaneously obtains improved local smoothing estimates, as well as improved Strichartz estimates for suitable initial data (see Theorem \ref{thm:localsmoothmain} and Remark \ref{rem:Strichartz}). Moreover, as in \cite{Rozendaal-Schippa23,LiRoSoYa24}, we apply these local smoothing estimates to nonlinear wave equations with initial data outside of $L^{2}$-based Sobolev spaces, in Section \ref{subsec:nonlinear}.

\subsection{Techniques}\label{subsec:techniques}

Apart from the specific results and the connection to decoupling theory, a significant difference between the present article and earlier contributions in this direction concerns the techniques that we use. 

Namely, in \cite{Smith98a,HaPoRo20} the Hardy spaces for FIOs were defined using conical square functions over the cosphere bundle $\Sp=\Rn\times S^{n-1}$, which allows one to incorporate the theory of tent spaces to prove various fundamental properties of the spaces. Such a conical square function characterization over the cosphere bundle is implicitly contained in \eqref{eq:normintro} when $p=q$, due to Fubini's theorem, but when $p\neq q$ this argument breaks down. More generally, the fact that FIOs are typically not bounded on $\funpqs$  indicates that one cannot expect to apply the same techniques to $\funpqs$ when $p\neq q$. 

On the other hand, \cite{Rozendaal-Schippa23} introduced Besov-type spaces adapted to the half-wave group using a similar norm as in \eqref{eq:normintro}, for general $p$ and $q$, albeit with $W^{s,p}(\Rn)$ replaced by a Besov space. The latter does not make a difference on dyadic frequency annuli, so for our applications to decoupling theory one could also use the spaces from \cite{Rozendaal-Schippa23}. On the other hand, working with Besov spaces does not allow one to recover the optimal fixed-time $L^{p}$ regularity for wave equations, and it does not yield improvements of the classical fractional integration theorem. Moreover, on a technical level, for Besov spaces it typically suffices to obtain estimates on dyadic frequency annuli, as opposed to having to deal with all frequency scales simultaneously through a square function, as is necessary in our setting. 

So, instead of relying on techniques from earlier work in this direction, we make connections to other areas to prove our main results. For example, we incorporate the theory of parabolic Hardy spaces from \cite{Calderon-Torchinsky75,Calderon-Torchinsky77}, by observing that it is convenient to replace $W^{s,p}(\Rn)$ in \eqref{eq:normintro} by $(1-\Delta)^{-s/2}H^{p}_{\w}(\Rn)$, where $H^{p}_{\w}(\Rn)$ is a parabolic Hardy space associated to a family of dilations that is anisotropic in the direction of $\w\in S^{n-1}$. For $1<p<\infty$ one has $L^{p}(\Rn)=H^{p}_{\w}(\Rn)$, but in Proposition \ref{prop:equivpar} we show that 
\[
\|\ph_{\w}(D)f\|_{\HT^{1}(\Rn)}\eqsim \|\ph_{\w}(D)f\|_{H^{1}_{\w}(\Rn)}
\]
as well, due to the fact that $\ph_{\w}(D)$ localizes to a paraboloid in the direction of $\w$. One can then use anisotropic Calder\'{o}n--Zygmund theory to prove embeddings and invariance properties of $\funpqs$.

As in \cite{HaPoRo20}, we use a wave packet transform $W$ to lift functions on $\Rn$ to phase space $T^{*}\Rn$ minus the zero section, parametrized in spherical coordinates as $\Rn\times S^{n-1}\times(0,\infty)$. This allows us to embed $\funpqs$ into a larger, but simpler and more established, space of functions on phase space. We can then derive various properties of $\funpqs$ from those of the encompassing space. However, unlike in \cite{HaPoRo20}, we do not use a tent space norm on $T^{*}\Rn$, working instead with an $L^{q}(S^{n-1};L^{p}(\Rn;L^{2}(0,\infty)))$ norm that arises naturally from \eqref{eq:normintro}, through the Littlewood--Paley decomposition of $L^{p}(\Rn)$. This in turn means that we cannot rely on tools such as the atomic decomposition of tent spaces. Instead, to prove the fundamental Theorem \ref{thm:FIOconj}, we use both the boundedness of the vector-valued Hardy--Littlewood maximal function over $S^{n-1}$, as well as the boundedness of anisotropic maximal functions in each separate direction $\w\in S^{n-1}$. Theorem \ref{thm:FIOconj} then allows us to deduce properties of $\funpqs$ from those of $L^{q}(S^{n-1};L^{p}(\Rn;L^{2}(0,\infty)))$. In its current incarnation, this setup does restrict us to considering $p,q\in(1,\infty)$ in several of our results. Due to the fact that decoupling theory is typically only of interest for $2\leq p,q<\infty$, we choose not to focus on the endpoint cases here.

\subsection{Organization}\label{subsec:organization}

In Section \ref{sec:prelim} we collect background on anisotropic dilations on $\Rn$
, on parabolic Hardy spaces, and on Fourier integral operators. In Section \ref{sec:transforms} we then introduce the parabolic frequency localizations and the associated wave packet transforms. In Section \ref{sec:spaces} we define the spaces $\funpqs$, and we derive many of their basic properties. The fact that they are invariant under the Euclidean wave propagators, but not under general FIOs, is shown in Section \ref{sec:invariance}. Section \ref{sec:embeddings} then contains the Sobolev embeddings and fractional integration theorems for $\funpqs$, while Section \ref{sec:decouple} connects these spaces to decoupling inequalities. Finally, in Section \ref{sec:wave} we obtain local smoothing estimates using $\funpqs$, as well as well-posedness results for nonlinear wave equations.

\subsection{Notation and terminology}\label{subsec:notation}

The natural numbers are $\N=\{1,2,\ldots\}$, and $\Z_{+}:=\N\cup\{0\}$. We write $\R_{+}$ for $(0,\infty)$, endowed with the Haar measure $\ud\sigma/\sigma$. Moreover, $S^{n-1}$ is the unit sphere in $\Rn$, with the unit normalized surface measure $\ud\w$. Throughout this article we fix $n\in\N$ with $n\geq2$.

For $\xi\in\Rn$ we write $\lb\xi\rb=(1+|\xi|^{2})^{1/2}$, and $\hat{\xi}=\xi/|\xi|$ if $\xi\neq0$. We use multi-index notation, where $\partial_{\xi}=\nabla_{\xi}=(\partial_{\xi_{1}},\ldots,\partial_{\xi_{n}})$ and $\partial^{\alpha}_{\xi}=\partial^{\alpha_{1}}_{\xi_{1}}\ldots\partial^{\alpha_{n}}_{\xi_{n}}$
for $\xi=(\xi_{1},\ldots,\xi_{n})\in\Rn$ and $\alpha=(\alpha_{1},\ldots,\alpha_{n})\in\Z_{+}^{n}$. Moreover, $\partial_{x\eta}^{2}\Phi$ is the mixed Hessian of a function $\Phi$ of the variables $x$ and $\eta$. 

The Fourier transform of a tempered distribution $f\in\Sw'(\Rn)$ is denoted by $\F f$ or $\widehat{f}$, and its inverse Fourier tranform by $\F^{-1}f$
. If $f\in L^{1}(\Rn)$ and $\xi\in\Rn$, then $\F f(\xi)=\int_{\Rn}e^{-i x\cdot \xi}f(x)\ud x$. The 
distributional pairing between $f\in\Sw'(\Rn)$ and $g\in\Sw(\Rn)$ is denoted by $\lb f,g\rb_{\Rn}$, and we write $\ph(D)$ for the Fourier multiplier with symbol $\ph$. The volume of a measurable subset $B\subseteq \Omega$ of a measure space $\Omega$ will be denoted by $|B|$, and its indicator function by $\ind_{B}$.  
The space of bounded linear operators between Banach spaces $X$ and $Y$ is $\La(X,Y)$, and $\La(X):=\La(X,X)$. 
The notation $W^{s,p}(\Rn)$ is used for the Sobolev space \[
W^{s,p}(\Rn):=\lb D\rb^{-s}L^{p}(\Rn).
\]
We write $f(s)\lesssim g(s)$ to indicate that $f(s)\leq Cg(s)$ for all $s$ and a constant $C>0$ independent of $s$, and similarly for $f(s)\gtrsim g(s)$ and $g(s)\eqsim f(s)$.

A list of notation specific to this article follows. 

\begin{center}
\begin{tabular}{|p{2cm}|p{6cm}|p{3.6cm}|}
\hline
\textbf{Notation} & \textbf{Meaning} & \textbf{Link to definition} \\
\hline
$s(p)$ & Adjustment in differentiability index & Equation \eqref{eq:sp} \\
$\|\cdot\|_{\funpqs}$ & Function space norm with indices $(p, q, s)$ associated with the wave packet transform $W$& Equation \eqref{eq:normintro}, Definition \ref{def:spaces} \\
$A_{\w,\sigma}$ & Anisotropic scaling & Equation \eqref{eq:dilations} \\
$|\cdot|_\omega$ & Anisotropic norm & Section \ref{subsec:norms} \\
$\Rn_{\w}$ & $\Rn$ endowed with the norm $|\cdot|_{\w}$ &Section \ref{subsec:norms}\\
$B_{\tau}^{\w}(x)$ & Anisotropic ball & Equation \eqref{eq:anisball} \\
$M_\omega$ & Anisotropic maximal function & Equation \eqref{eq:anismaximal} \\
$\Psi_0$, $\Psi$ & Littlewood--Paley cutoff & Equations \eqref{eq:Psi0}, \eqref{eq:PsiRn} \\
$\Psi_\w$ & Anisotropic Littlewood--Paley cutoff & Equation \eqref{eq:Psiw} \\
$\HT^p(\Rn)$ & Local Hardy space & Definition \ref{def:localHardy} \\
$\HT^{s,p}(\Rn)$ & Sobolev space over local Hardy space & Definition \ref{def:localHardy} \\
$H^{p}_{\w}(\Rn)$ & Anisotropic Hardy space & Definition \ref{def:parHardy} \\
$\ph_{\w}$ & Parabolic frequency localization & Section \ref{subsec:packets} \\
$\psi_{\w,\sigma}$ & Wave packet & Equation \eqref{eq:psiwsigma} \\
$W$ & Wave packet transform & Equation \eqref{eq:W} \\
$V$ & Adjoint of wave packet transform & Equation \eqref{eq:V} \\
$\chi_\nu$ & Angular localizer & Section \ref{subsec:discrete} \\
$d(p,q)$ & Decoupling exponent & Equation \eqref{eq:dpq} \\
$\alpha(p)$ & Adjustment in differentiability index & Equations \eqref{eq:alphapintro}, \eqref{eq:alphap} \\
\hline
\end{tabular}
\end{center}

\section{Preliminaries}\label{sec:prelim}

In this section we first introduce a family of norms, associated with groups of dilations. We then define parabolic Hardy spaces associated with these dilations, and we collect some background on a specific class of Fourier integral operators.

\subsection{A family of norms}\label{subsec:norms}

In this subsection, we define a collection of norms on $\Rn$, and the associated metrics. We then derive some of their basic properties.

For $\w\in S^{n-1}$, $\sigma>0$ and $x\in\Rn$, set 
\begin{equation}\label{eq:dilations}
A_{\w,\sigma}(x):=\sigma^{2}(\w\cdot x)\w+\sigma\Pi_{\w}^{\perp}x,
\end{equation}
where $\Pi_{\w}^{\perp}:\Rn\to\Rn$ is the orthogonal projection onto the complement of the span of $\w$. Note that, with $P_{\w}(x):=x+(\w\cdot x)\w$, one has
\[
A_{\w,\sigma}=\sigma^{P_{\w}}=\exp(\log(\sigma)P_{\w}).
\]
This implies that $(A_{\w,\sigma})_{\sigma>0}$ is a group of transformations as in \cite{Calderon-Torchinsky75}, and we will rely on the theory developed there in what follows.

Let $|x|_{\w}$ be the unique $\sigma_{x}>0$ such that $|A_{\w,1/\sigma_{x}}(x)|=1$. Then $|\cdot|_{\w}$ is a norm on $\Rn$, in the sense of \cite[XIII.5]{Stein93} (see \cite[Section 1.4]{Calderon-Torchinsky75}), and $|A_{\w,\sigma}x|_{\w}=\sigma|x|_{\w}$ for all $\sigma>0$. 
Let $\Rn_{\w}$ 
be the metric measure space obtained by endowing $\Rn$ with the associated metric and the standard Lebesgue measure. For $\tau>0$, we write
\begin{equation}\label{eq:anisball}
B_{\tau}^{\w}(x):=\{y\in \Rn\mid |x-y|_{\w}<\tau\}
\end{equation}
for the ball in $\Rn_{\w}$ around $x\in\Rn$ with radius $\tau$, and
\[
B_{\tau}(x):=\{y\in \Rn\mid |x-y|<\tau\}
\]
for the corresponding Euclidean ball in $\Rn$. Note that $B_{\tau}^{\w}(x)$ is a convex set with respect to the Euclidean metric, given that $B_{\tau}^{\w}(0)$ is the inverse image of $B_{1}(0)$ under the linear map $A_{\w,1/\tau}$. For the same reason, $B^{\w}_{1}(x)=B_{1}(x)$. Moreover, using that $A_{\w,\tau}$ has determinant $\tau^{n+1}$, it follows that
\begin{equation}\label{eq:volumeball}
|B_{\tau}^{\w}(x)|=\tau^{n+1}|B_{1}(0)|.
\end{equation}
In particular, $\Rn_{\w}$ is a doubling metric measure space.

We will often work with the equivalent expression for $|\cdot|_{\w}$ provided by the following lemma.

\begin{lemma}\label{lem:equivnorm}
Let $\omega\in S^{n-1}$ and $x\in\Rn$. Then
\begin{equation}\label{eq:equivnorm}
|x|_{\w}\leq |\w\cdot x|^{1/2}+|\Pi_{\w}^{\perp}x|\leq 2|x|_{\w},
\end{equation}
and 
\begin{equation}\label{eq:equivnorm2}
|x|\leq |x|_{\w}\text{ if and only if }x\in \overline{B^{\w}_{1}(0)}=\overline{B_{1}(0)}.
\end{equation}
\end{lemma}
\begin{proof}
We may consider $x\neq0$. Set $\sigma:=|x|_{\w}$. Then, by definition and by \eqref{eq:dilations},
\[
1=|A_{\w,1/\sigma}(x)|=\sqrt{\sigma^{-4}|\w\cdot x|^{2}+\sigma^{-2}|\Pi_{\w}^{\perp}x|^{2}}
\]
and
\begin{equation}\label{eq:expressnorm}
|x|_{\w}=\sigma=\sqrt{\sigma^{-2}|\w\cdot x|^{2}+|\Pi_{\w}^{\perp}x|^{2}}.
\end{equation}
This implies \eqref{eq:equivnorm2}. Moreover, trivial estimates yield $|x|_{\w}\geq |x|_{\w}^{-1}|\w\cdot x|$ and $|x|_{\w}\geq |\Pi_{\w}^{\perp}x|$, which proves the second inequality in \eqref{eq:equivnorm}.

For the other inequality, it follows from \eqref{eq:expressnorm} that, for all $y\in\Rn$, one has $|y|_{\w}=|\Pi_{\w}^{\perp}y|$ if $\w\cdot y=0$. Moreover, $|y|_{\w}=|y|_{\w}^{-1}|\w\cdot y|$ and $|y|_{\w}=|\w\cdot y|^{1/2}$ if $\Pi_{\w}^{\perp}y=0$. Hence one may use that $|\cdot|_{\w}$ is a norm to write
\[
|x|_{\w}\leq |(\w\cdot x)\w|_{\w}+|\Pi_{\w}^{\perp}x|_{\w}=|\w\cdot x|^{1/2}+|\Pi_{\w}^{\perp}x|.\qedhere
\]
\end{proof}




It follows from Lemma \ref{lem:equivnorm} that the metric associated with $|\cdot|_{\w}$ is, at small distances, essentially the restriction to the subspace $\Rn\times \{\w\}\subseteq \Sp$ of the metric on the cosphere bundle from \cite{HaPoRo20,Smith98a}. Similarly, the square of the metric associated with $|\cdot|_{\w}$ essentially coincides with the quasi-distance from \cite{Frey-Portal20}. 


\begin{corollary}\label{cor:equivnorm}
Let $\ka_{1},\ka_{2}\geq1$, $\w\in S^{n-1}$ and $\xi\in\Rn$ be such that $|\xi|\geq\ka_{1}^{-1}$ and $|\hat{\xi}-\w|\leq \ka_{2} |\xi|^{-1/2}$. Then $|\xi|_{\w}\geq \ka_{1}^{-1}$ and 
\begin{equation}\label{eq:equivnorm3}
(4+\ka_{1}^{2})^{-1/2}|\xi|^{1/2}\leq |\xi|_{\w}\leq (1+\ka_{2})|\xi|^{1/2}.
\end{equation}
\end{corollary}
\begin{proof}
Since $\ka_{1}\geq 1$, the first inequality follows from \eqref{eq:equivnorm2}. On the other hand, $|\w\cdot\xi|^{1/2}\leq |\xi|^{1/2}$ and, by assumption,
\[
|\Pi_{\w}^{\perp}\xi|\leq \sqrt{|\w\cdot \xi-|\xi||^{2}+|\Pi_{\w}^{\perp}\xi|^{2}}=|\xi-|\xi|\w|\leq \ka_{2}|\xi|^{1/2}.
\]
Now \eqref{eq:equivnorm} yields the right-most inequality in \eqref{eq:equivnorm3}. For the remaining inequality, we can combine what we have already shown with \eqref{eq:equivnorm} to write
\[
|\xi|\leq |\w\cdot\xi|+|\Pi_{\w}^{\perp}\xi| \leq 1+|\w\cdot\xi|+|\Pi_{\w}^{\perp}\xi|^{2}\leq \ka_{1}^{2}|\xi|_{\w}^{2}+4|\xi|_{\w}^{2}.\qedhere
\] 
\end{proof}

Since $\Rn_{\w}$ is a doubling metric measure space, it is natural to consider the (centered) Hardy--Littlewood maximal function $M_{\w}$ on $\R^{n}_{\w}$, given by 
\begin{equation}\label{eq:anismaximal}
M_{\w}f(x):=\sup_{\tau>0}\fint_{B_{\tau}^{\w}(x)}|f(y)|\ud y
\end{equation}
for $x\in\Rn$ and $f\in L^{1}_{\loc}(\Rn)$. 
We record the following standard lemma concerning the boundedness of the vector-valued extension of $M_{\w}$ to $L^{p}(\Rn;L^{q}(\R_{+}))$, for $p,q\in(1,\infty)$. Recall that $\R_{+}$ denotes the measure space $(0,\infty)$, endowed with the Haar measure $\ud\sigma/\sigma$.

\begin{lemma}\label{lem:maximal}
Let $p,q\in(1,\infty)$. Then there exists a $C\geq 0$ such that 
\[
\Big(\int_{\Rn}\Big(\int_{0}^{\infty}|M_{\w}F(\cdot,\sigma)(x)|^{q}\frac{\ud\sigma}{\sigma}\Big)^{p/q}\ud x\Big)^{1/p}\leq C\|F\|_{L^{p}(\Rn;L^{q}(\R_{+}))}.
\]
for all $\w\in S^{n-1}$ and $F\in L^{p}(\Rn;L^{q}(\R_{+}))$.
\end{lemma}
\begin{proof}
For a fixed $\w\in S^{n-1}$, the statement is a special case of \cite[Sections II.1 and II.5.14]{Stein93} or 
\cite[Lemma 2.8]{Sato18}. But by rotation, the resulting constant is independent of $\w$.
\end{proof}


\begin{corollary}\label{cor:maximal}
Let $p,q\in(1,\infty)$ and $N>(n+1)/2$. Then there exists a $C\geq 0$ such that 
\[
\Big(\int_{\Rn}\Big(\int_{0}^{\infty}\Big(\int_{\Rn}\sigma^{-\frac{n+1}{2}}\frac{|F(x-y,\sigma)|}{(1+\sigma^{-1}|y|_{\w}^{2})^{N}}\ud y\Big)^{q}\frac{\ud\sigma}{\sigma}\Big)^{p/q}\ud x\Big)^{1/p}\leq C\|F\|_{L^{p}(\Rn;L^{q}(\R_{+}))}
\]
for all $\w\in S^{n-1}$ and $F\in L^{p}(\Rn;L^{q}(\R_{+}))$.
\end{corollary}
\begin{proof}
For all $x\in\Rn$ and $\sigma>0$ one has
\begin{align*}
\int_{\Rn}\sigma^{-\frac{n+1}{2}}\frac{|F(x-y,\sigma)|}{(1+\sigma^{-1}|y|_{\w}^{2})^{N}}\ud y&\leq M_{\w}F(\cdot,\sigma)(x)\int_{\Rn}\sigma^{-\frac{n+1}{2}}\frac{1}{(1+|A_{\w,1/\sqrt{\sigma}}y|_{\w}^{2})^{N}}\ud y\\
&\lesssim M_{\w}F(\cdot,\sigma)(x),
\end{align*}
with the first inequality being classical (see e.g.~\cite[Sections II.2.1 and II.5.14]{Stein93}). Then Lemma \ref{lem:maximal} concludes the proof.
\end{proof}

In the next section, we will combine Lemma \ref{lem:maximal} with the following 
pointwise variant of \cite[Lemma 2.3]{Sato18}. 

\begin{proposition}\label{prop:analytic}
Let $0<r<\infty$ and $\ka>0$. Then there exists a $C\geq 0$ such that, for all $\w\in S^{n-1}$, $\tau>0$, $x,y\in\Rn$ and $f\in L^{\infty}(\Rn)$ with
\[
\supp(\wh{f}\,)\subseteq\{\xi\in\Rn\mid |\xi|_{\w}\leq \ka\tau^{-1}\},
\]
one has 
\[
|f(x-y)|^{r}\leq C(1+\tau^{-1}|y|_{\w})^{n+1}M_{\w}(|f|^{r})(x).
\]
\end{proposition}
\begin{proof}
After replacing $f(x)$ by $\tau^{n+1}f(A_{\w,\tau}x)$ and using that $A_{\w,\tau}$ has determinant $\tau^{n+1}$, we may suppose that $\tau=1$. The proof is then completely analogous to that of the 
corresponding Euclidean inequality 
\begin{equation}\label{eq:isanalytic}
|g(x-y)|^{r}\lesssim (1+\rho^{-1}|y|)^{n}M(|g|^{r})(x),
\end{equation}
where $g\in L^{\infty}(\Rn)$ and $\rho>0$ are such that $\supp(\wh{g}\,)\subseteq\{\xi\in\Rn\mid |\xi|\leq \ka\rho^{-1}\}$ and  $M$ is the classical Hardy--Littlewood maximal function, given by
\begin{equation}\label{eq:ismaximal}
Mg(x):=\sup_{\tau>0}\fint_{B_{\tau}(x)}|g(z)|\ud z .
\end{equation}
One uses the mean value theorem and the fact that balls in $\R^{n}_{\w}$ are convex with respect to the Euclidean metric. For a proof of \eqref{eq:isanalytic} in the case where $\rho=1$, see e.g.~\cite[Theorem 1.3.1]{Triebel10}.
\end{proof}

%

\subsection{Hardy spaces}\label{subsec:Hardy}

In this subsection we collect some basics on certain parabolic Hardy spaces. We refer to \cite{Calderon-Torchinsky75,Calderon-Torchinsky77,Folland-Stein82} for the general theory of parabolic Hardy spaces.

Throughout, we fix a real-valued $\Psi_0\in C^{\infty}_{c}(\R)$ such that $\supp(\Psi_0)\subseteq [1/2,2]$ and
\begin{equation}\label{eq:Psi0}
\int_{0}^{\infty}\Psi_{0}(\sigma)^{2}\frac{\ud \sigma}{\sigma}=1.
\end{equation}
We then define $\Psi \in C^\infty_c(\Rn)$ by 
\begin{equation}\label{eq:PsiRn}
\Psi(\xi):= \Psi_0(|\xi|)\quad(\xi\in\Rn)
\end{equation}
and note that 
\begin{equation}\label{eq:Psi}
\int_{0}^{\infty}\Psi(\sigma\xi)^{2}\frac{\ud \sigma}{\sigma}=1
\end{equation}
if $\xi\neq0$.

Let $H^{p}(\Rn)$, for $1\leq p<\infty$, consist of those $f\in\Sw'(\Rn)$ such that 
\begin{equation}\label{eq:LittlePaley}
\|f\|_{H^{p}(\Rn)}:=\Big(\int_{\Rn}\Big(\int_{0}^{\infty}|\Psi(\sigma D)f(x)|^{2}\frac{\ud\sigma}{\sigma}\Big)^{p/2}\ud x\Big)^{1/p}<\infty.
\end{equation}
Then $H^{p}(\Rn)=L^{p}(\Rn)$ for $1<p<\infty$, while $H^{1}(\Rn)$ is the classical real Hardy space. We also write $H^{\infty}(\Rn):=\BMO(\Rn)=(H^{1}(\Rn))^{*}$, for convenience.

In this article, a bigger role will be played by the \emph{local} Hardy spaces $\HT^{p}(\Rn)$ and their Sobolev spaces. For $1<p<\infty$ one again has $\HT^{p}(\Rn)=L^{p}(\Rn)$, but we give a unified definition for all $1\leq p\leq\infty$, in part for the purpose of comparison to other function spaces that we will encounter in Definition \ref{def:spaces}. Throughout, let $\rho\in C^{\infty}_{c}(\Rn)$ be such that $\rho(\xi)=1$ for $|\xi|\leq 2$.

\begin{definition}\label{def:localHardy} 
Let $1\leq p\leq \infty$. Then $\HT^{p}(\Rn)$ consists of those $f\in\Sw'(\Rn)$ such that $\rho(D)f\in L^{p}(\Rn)$ and $(1-\rho)(D)f\in H^{p}(\Rn)$, endowed with the norm
\[
\|f\|_{\HT^{p}(\Rn)}:=\|\rho(D)f\|_{L^{p}(\Rn)}+\|(1-\rho)(D)f\|_{H^{p}(\Rn)}
\]
for $f\in\HT^{p}(\Rn)$. Moreover, we set
\[
\HT^{s,p}(\Rn):=\lb D\rb^{-s}\HT^{p}(\Rn)
\]
for $s\in\R$.
\end{definition}

As already noted, for all $1<p<\infty$ and $s\in\R$ one simply has 
\[
\HT^{s,p}(\Rn)=W^{s,p}(\Rn),
\]
while $\HT^{\infty}(\Rn)=(\HT^{1}(\Rn))^{*}$. Note that the difference between the standard Hardy space $H^{1}(\Rn)$ and the local Hardy space $\HT^1(\Rn)$ concerns the low frequencies. In particular, the low frequency localised elements of $\HT^{1}(\Rn)$ do not need to satisfy a cancellation condition.

\begin{remark}\label{rem:Hpindependent}
For use in Section \ref{sec:spaces}, it is convenient to note that the definition of $\HT^{p}(\Rn)$ is independent of the choice of low-frequency cutoff, up to norm equivalence. That is, the function $\rho$ in Definition \ref{def:localHardy} can be replaced by any $\rho'\in C^{\infty}_{c}(\Rn)$ satisfying $\rho'\equiv1$ in a neighborhood of zero.
\end{remark}

We define the parabolic Hardy spaces in a manner analogous to \eqref{eq:LittlePaley}, using the dilations from \eqref{eq:dilations} in place of standard dilations. For $\w\in S^{n-1}$ and $\xi\in\Rn$, set 
\begin{equation}\label{eq:Psiw}
\Psi_{\w}(\xi):=\Psi_{0}(|\xi|_{\w}).
\end{equation}
Then $\Psi_{\w}\in C^{\infty}_{c}(\Rn)$, due to the fact that $\xi\mapsto |\xi|_{\w}$ is smooth on $\Rn\setminus\{0\}$ (see \cite[Lemma 1.5]{Calderon-Torchinsky75}). Moreover, $\supp(\Psi_{\w})\subseteq\{\xi\in\Rn\mid |\xi|_{\w}\in[1/2,2]\}$, and
\begin{equation}\label{eq:Psiwint}
\int_{0}^{\infty}\Psi_{\w}(A_{\w,\sigma}\xi)^{2}\frac{\ud \sigma}{\sigma}=1
\end{equation}
if $\xi\neq0$. Indeed, since $|A_{\omega, \sigma} \xi|_\omega = \sigma |\xi|_\omega$, \eqref{eq:Psiwint} is a consequence of \eqref{eq:Psi0}. 

We will write $\Psi_{\w}(A_{\w,\sigma}D)$ for the Fourier multiplier with symbol $\xi\mapsto \Psi_{\w}(A_{\w,\sigma}\xi)$. Note that $\Psi_{\w}(A_{\w,\sigma}D)$ has kernel $\sigma^{-(n+1)}\F^{-1}(\Psi_{\w})(A_{\w,1/\sigma}y)$ for $y\in\Rn$. Moreover, for each $N\geq0$ there exists a $C_{N}\geq0$, independent of $\w$, $\sigma$ and $y$, such that
\begin{equation}\label{eq:anisdecay}
|\F^{-1}(\Psi_{\w})(A_{\w,1/\sigma}y)|\leq C_{N}(1+\sigma^{-1}|y|_{\w})^{-N}.
\end{equation}
That the constant is independent of $\w$ follows by rotation.

\begin{definition}\label{def:parHardy}
Let $1\leq p<\infty$ and $\w\in S^{n-1}$. Then $H^{p}_{\w}(\Rn)$ consists of those $f\in\Sw'(\Rn)$ such that
\[
\|f\|_{H^{p}_{\w}(\Rn)}:=\Big(\int_{\Rn}\Big(\int_{0}^{\infty}|\Psi_{\w}(A_{\w,\sigma}D)f(x)|^{2}\frac{\ud\sigma}{\sigma}\Big)^{p/2}\ud x\Big)^{1/p}<\infty.
\]
\end{definition}

In fact, this is not how the parabolic Hardy spaces were originally defined in \cite{Calderon-Torchinsky75,Calderon-Torchinsky77}. However, it follows from \cite{Sato18} that, up to norm equivalence, the relevant definitions coincide.

As was the case for $H^{p}(\Rn)$ and $\HT^{p}(\Rn)$, for all $\w\in S^{n-1}$ one has   
\begin{equation}\label{eq:anisLp}
H^{p}_{\w}(\Rn)=L^{p}(\Rn)\text{ if }1<p<\infty,
\end{equation}
with equivalence of norms (see \cite[Theorem 1.2]{Calderon-Torchinsky77}). Moreover, $H^{1}_{\w}(\Rn)\subseteq L^{1}(\Rn)$ continuously, by \cite[Theorem 2.7]{Folland-Stein82}. 

The following proposition relates the parabolic and classical Hardy space norms of a function with frequency support inside a paraboloid.

\begin{proposition}\label{prop:equivpar}
Let $1\leq p<\infty$ and $\ka_{1},\ka_{2}\geq1$. Then there exists a $C>0$ such that the following holds. Let $\w\in S^{n-1}$ and $f\in\Sw'(\Rn)$ satisfy
\[
\supp(\wh{f}\,)\subseteq\{\xi\in\Rn\mid |\xi|\geq\ka_{1}^{-1}, |\hat{\xi}-\w|\leq \ka_{2} |\xi|^{-1/2}\}.
\]
Then $f\in H^{p}(\Rn)$ if and only if $f\in H^{p}_{\w}(\Rn)$, in which case
\[
\frac{1}{C}\|f\|_{H^{p}(\Rn)}\leq \|f\|_{H^{p}_{\w}(\Rn)}\leq C\|f\|_{H^{p}(\Rn)}.
\]
\end{proposition}
\begin{proof}
For $1<p<\infty$, the statement follows directly from \eqref{eq:anisLp}. However, our argument applies for all $1\leq p<\infty$, so there is no need to rely on \eqref{eq:anisLp}.

Fix $\w\in S^{n-1}$ and $f\in\Sw'(\Rn)$ with the prescribed support properties. Since $H^{p}(\Rn)$ and $H^{p}_{\w}(\Rn)$ are both contained in $L^{p}(\Rn)$, we may suppose in the remainder that $f\in L^{p}(\Rn)$, which will allow us to apply Proposition \ref{prop:analytic}. 

For $\tau>0$ and $\xi\in\Rn$, one has $\Psi(\tau\xi)=0$ unless $|\xi|\in[(2\tau)^{-1},2\tau^{-1}]$. Similarly, for $\sigma>0$ one has  $\Psi(A_{\w,\sigma}\xi)=0$ unless $|A_{\w,\sigma}\xi|_{\w}\in[1/2,2]$, i.e.~unless $|\xi|_{\w}\in[(2\sigma)^{-1},2\sigma^{-1}]$. Hence Corollary \ref{cor:equivnorm} implies that
\begin{equation}\label{eq:sigmatau}
\tfrac{1}{8(4+\ka_{1}^{2})}\sigma^{2}\leq \tau\leq 8(1+\ka_{2})^{2}\sigma^{2}
\end{equation}
whenever $\Psi_{\w}(A_{\w,\sigma}D)\Psi(\tau D)f\neq 0$.

Now fix $r\in(0,\min(p,2))$. We claim that, with notation as in \eqref{eq:anismaximal} and \eqref{eq:ismaximal}, the following inequalities hold:
\begin{align}
\label{eq:maxest1}|\Psi_{\w}(A_{\w,\sigma}D)\Psi(\tau D)^{2}f(x)|&\lesssim (M_{\w}(|\Psi(\tau D)^{2}f|^{r})(x))^{1/r},\\
\label{eq:maxest2}|\Psi(\sigma D)\Psi_{\w}(A_{\w,\tau}D)^{2}f(x)|&\lesssim (M(|\Psi_{\tau}(A_{\w,\tau}D)^{2}f|^{r})(x))^{1/r},\\
\label{eq:maxest3}|\Psi(\tau D)^{2}f(x)|&\lesssim (M(|\Psi(\tau D)f|^{r})(x))^{1/r},\\
\label{eq:maxest4}|\Psi_{\w}(A_{\w,\tau}D)^{2}f(x)|&\lesssim (M_{\w}(|\Psi_{\w}(A_{\w,\tau} D)f|^{r})(x))^{1/r},
\end{align}
for implicit constants independent of $\w$, $f$, $\sigma,\tau>0$ and $x\in\Rn$. To see why this is true, first consider \eqref{eq:maxest1}. 

We may suppose that $\sigma$ and $\tau$ satisfy \eqref{eq:sigmatau}. Then, by Corollary \ref{cor:equivnorm},
\[
\supp(\F(\Psi(\tau D)^{2}f))\subseteq\{\xi\in\Rn\mid |\xi|_{\w}\leq 4(1+\ka_{2})(4+\ka_{1}^{2})^{1/2}\sigma^{-1}\}.
\]
Moreover, $f\in L^{p}(\Rn)$, by assumption, and then $\Psi(\tau D)^{2}f\in L^{\infty}(\Rn)$, by a standard Sobolev embedding. Hence \eqref{eq:anisdecay}, Proposition \ref{prop:analytic} and \eqref{eq:equivnorm} yield
\begin{align*}
&|\Psi_{\w}(A_{\w,\sigma}D)\Psi(\tau D)^{2}f(x)| 
\lesssim \int_{\R^n} |\Psi(\tau D)^{2}f(x-y)| \sigma^{-(n+1)}(1+\sigma^{-1}|y|_{\w})^{-N} \ud y \\
&\lesssim (M_{\omega}(|\Psi(\tau D)^{2}f|^r)(x))^{1/r}  \int_{\R^n} (1+\sigma^{-1}|y|_{\w})^{\frac{n+1}{r}} (1+\sigma^{-1}|y|_{\w})^{-N} \sigma^{-(n+1)}\ud y \\
&=(M_{\omega}(|\Psi(\tau D)^{2}f|^r)(x))^{1/r}  \int_{\R^n} (1+|y|_{\w})^{-(N-\frac{n+1}{r})} \ud y\lesssim (M_{\omega}(|\Psi(\tau D)^{2}f|^r)(x))^{1/r},
\end{align*}
if we choose $N>(n+1)(1+\tfrac{1}{r})$. 

The argument for \eqref{eq:maxest2} is analogous, although one relies on \eqref{eq:isanalytic} instead of Proposition \ref{prop:analytic}, which is allowed because of Corollary \ref{cor:equivnorm}. Moreover, \eqref{eq:maxest3} and \eqref{eq:maxest4} are classical estimates which can be obtained in the same manner. This proves the claim.

Now, finally, suppose that $f\in H^{p}(\Rn)$, and set $a:=(8(4+\ka^{2}))^{-1}$ and $b:=8(1+\ka)^{2}$. For $\sigma>0$ and $x\in\Rn$, \eqref{eq:Psi}, H\"{o}lder's inequality, \eqref{eq:sigmatau} and \eqref{eq:maxest1} yield
\begin{align*}
|\Psi_{\w}(A_{\w,\sigma}D)f(x)|^{2}&\leq \Big(\int_{0}^{\infty}|\Psi_{\w}(A_{\w,\sigma}D)\Psi(\tau D)^{2}f(x)|\frac{\ud\tau}{\tau}\Big)^{2}\\
&\lesssim \int_{a\sigma^{2}}^{b\sigma^{2}}(M_{\w}(|\Psi(\tau D)^{2}f|^{r})(x))^{2/r}\frac{\ud\tau}{\tau}.
\end{align*}
Hence
\begin{align*}
\int_{0}^{\infty}|\Psi_{\w}(A_{\w,\sigma}D)f(x)|^{2}\frac{\ud\sigma}{\sigma}&\lesssim \int_{0}^{\infty}\int_{\sqrt{\tau}/\sqrt{b}}^{\sqrt{\tau}/\sqrt{a}}(M_{\w}(|\Psi(\tau D)^{2}f|^{r})(x))^{2/r}\frac{\ud\sigma}{\sigma}\frac{\ud\tau}{\tau}\\
&\lesssim \int_{0}^{\infty}(M_{\w}(|\Psi(\tau D)^{2}f|^{r})(x))^{2/r}\frac{\ud\tau}{\tau}.
\end{align*}
Since $r < \min(p,2)$, Lemma \ref{lem:maximal} implies that $M_{\w}$ is bounded on $L^{p/r}(\Rn;L^{2/r}(\R_{+}))$. In fact, applying Lemma \ref{lem:maximal} with $F(x, \sigma) = |\Psi(\sigma D)^{2} f(x)|^r$, $q= 2/r$ and $p$ replaced by $p/r$, we see that 
\begin{align*}
\|f\|_{H^{p}_{\w}(\Rn)}&= \Big( \int_{\Rn} \Big( \int_0^\infty |\Psi_{\w}(A_{\w, \sigma}D) f(x)|^2 \frac{\ud\sigma}{\sigma} \Big) ^{p/2} \ud x \Big)^{1/p} \\
&\lesssim \Big( \int_{\Rn} \Big( \int_0^\infty  (M_{\w}(|\Psi(\tau D)^{2}f|^{r})(x))^{2/r}\frac{\ud\tau}{\tau} \Big) ^{p/2} \ud x \Big)^{1/p}  \\
&\lesssim \Big( \int_{\Rn} \Big( \int_0^\infty  \big(|\Psi(\tau D)^{2}f(x)|^{r}\big)^{2/r}\frac{\ud\tau}{\tau} \Big) ^{p/2} \ud x \Big)^{1/p} \\
&= \Big( \int_{\Rn} \Big( \int_0^\infty  |\Psi(\tau D)^{2}f(x)|^{2} \frac{\ud\tau}{\tau} \Big) ^{p/2} \ud x \Big)^{1/p}.
\end{align*}
On the other hand, the analogue of Lemma \ref{lem:maximal} also holds for the Hardy--Littlewood maximal function $M$ (see \cite[Section II.1]{Stein93}). Combined with \eqref{eq:maxest3}, this yields
\begin{align*}
\|f\|_{H^{p}_{\w}(\Rn)}&\lesssim \Big(\int_{\Rn} \Big( \int_0^\infty |\Psi(\tau D)^{2} f(x)|^2 \frac{\ud\tau}{\tau} \Big) ^{p/2} \ud x \Big)^{1/p} \\
&\lesssim \Big( \int_{\Rn} \Big( \int_0^\infty  (M(|\Psi(\tau D)f|^{r})(x))^{2/r}\frac{\ud\tau}{\tau} \Big) ^{p/2} \ud x \Big)^{1/p}  \\
&\lesssim \Big( \int_{\Rn} \Big( \int_0^\infty  \big(|\Psi(\tau D)f(x)|^{r}\big)^{2/r}\frac{\ud\tau}{\tau} \Big) ^{p/2} \ud x \Big)^{1/p} \\
&= \Big( \int_{\Rn} \Big( \int_0^\infty  |\Psi(\tau D)f(x)|^{2} \frac{\ud\tau}{\tau} \Big) ^{p/2} \ud x \Big)^{1/p}=\|f\|_{H^{p}(\Rn)},
\end{align*}
where in the final step we used \eqref{eq:LittlePaley}. This shows that $f\in H^{p}_{\w}(\Rn)$, with the required norm bound. The proof of the reverse inequality is analogous, relying on \eqref{eq:maxest2} and \eqref{eq:maxest4}.
\end{proof}

We conclude this subsection with two results, about fractional integration and Fourier multipliers in our present anisotropic setting. 

The first result concerns the anisotropic fractional integration operator $I_{\w,s}$, given for $\w\in S^{n-1}$, $s>0$, $f\in\Sw(\Rn)$ with $0\notin\supp(\wh{f}\,)$, and $x\in\Rn$, by
\begin{equation}\label{eq:fracan}
I_{\w,s}f(x):=\frac{1}{(2\pi)^{n}}\int_{\Rn}e^{ix\cdot\xi}|\xi|_{\w}^{-s}\wh{f}(\xi)\ud\xi.
\end{equation}
By rotation, one obtains from \cite[Theorem 4.1]{Calderon-Torchinsky77} and \eqref{eq:anisLp} 
the following proposition.

\begin{proposition}\label{prop:fracintHardy}
Let $1\leq p<q<\infty$, $\ka>0$ and $s:=(n+1)(\frac{1}{p}-\frac{1}{q})$. Then there exists a $C\geq0$ such that, for all $\w\in S^{n-1}$ and $f\in H^{p}_{\w}(\Rn)$ with $\supp(\wh{f}\,)\subseteq\{\xi\in\Rn\mid |\xi|\geq \ka\}$, one has $I_{\w,s}f\in L^{q}(\Rn)$ and 
\[
\|I_{\w,s}f\|_{L^{q}(\Rn)}\leq C\|f\|_{H^{p}_{\w}(\Rn)}.
\] 
\end{proposition}

On the other hand, the next result gives a criterion for a Fourier multiplier to be bounded on $H^{p}_{\w}(\Rn)$ itself. 

\begin{proposition}\label{prop:multHardy}
Let $p\in[1,\infty)$. Then there exist $M,C_{M},C\geq0$ such that the following holds for all $\w\in S^{n-1}$. Let $m\in C^{\infty}(\Rn)$ be such that
\[
|(\w\cdot \partial_{\xi})^{\beta}\partial_{\xi}^{\alpha} m(\xi)|\leq C_{M}\lb |\xi|_{\w}\rb^{-|\alpha|-2\beta}
\]
for all $\alpha\in\Z_{+}^{n}$ and $\beta\in\Z_{+}$ with $|\alpha|+|\beta|\leq M$, and all $\xi\in\Rn$. Then $m(D)\in \La(H^{p}_{\w}(\Rn))$ and
\[
\|m(D)\|_{\La(H^{p}_{\w}(\Rn))}\leq C.
\]
\end{proposition}
\begin{proof}
By rotating the coordinate system we may assume that $\omega=e_n$, the $n$-th coordinate direction. Then the multiplier satisfies
\[
|\partial_{\xi'}^{\alpha'} \partial_{\xi_n}^{\beta} m(\xi',\xi_n)| \leq C_M (1+|\xi'|+|\xi_n|^{1/2})^{-|\alpha'|-2\beta}.
\]
This allows us to use the results on singular integrals on homogeneous groups, as detailed in, for example, \cite{Folland-Stein82} or \cite[Chapter XIII, Section 5]{Stein93}. More precisely, we consider $\R^n$ as a Lie group with the usual addition structure, and let 
\[
\delta \circ x = (\delta^{a_1} x_1, \dots, \delta^{a_n} x_n)
\]
with $a_1 = \dots = a_{n-1} = 1$ and $a_n = 2$
(see Example 5.2.1 in \cite[Chapter XIII, Section 5]{Stein93}). Define a corresponding norm function $\rho(x) := \max_{1 \leq j \leq n} |x_j|^{1/a_j}$. 
If $K$ is the tempered distribution given by the inverse Fourier transform of $m$, then $K$ satisfies equation (61) in \cite[Chapter XIII, Section 5]{Stein93}, namely
\[
|\partial_{x'}^{\alpha'} \partial_{x_n}^{\beta} K(x)| \lesssim \rho(x)^{-a-|\alpha'|-2 \beta}, \qquad x \ne 0,
\]
where $a = a_1 + \dots + a_n = n+1$ (one can verify this by a simple adaptation of the proof of Proposition 2 in \cite[Chapter IV, Section 4.4]{Stein93}). In the terminology of \cite[p.185]{Folland-Stein82}, this says that $K$ is a kernel of type $(0,r)$ for any positive integer $r$. Thus the case $\alpha = 0$ of \cite[Theorem 6.10]{Folland-Stein82} guarantees that $m(D)$ is bounded on $H^p_{\w}(\Rn)$ for all $0 < p < \infty$; we only need the case $1 \leq p < \infty$ here. (For $1 < p < \infty$, one can also use Theorem 3 of \cite[Chapter I, Section 5]{Stein93} together with duality, as pointed out at the beginning of \cite[Chapter XIII, Section 5.3.1]{Stein93}).
\end{proof}

\subsection{Fourier integral operators}\label{subsec:FIOs}

Here we collect some background on Fourier integral operators. We refer to \cite{Hormander09,Duistermaat11} for the general theory of Fourier integral operators and the associated notions from symplectic geometry. On the other hand, in this article we will mostly work with concrete oscillatory integral representations for such operators, and in this subsection one can find the relevant definitions.

For $m\in\R$ and $\rho,\delta\in[0,1]$, recall that H\"{o}rmander's symbol class $S^{m}_{\rho,\delta}$ consists of those $a\in C^{\infty}(\R^{2n})$ such that, for all $\alpha,\beta\in\Z_{+}^{n}$, there exists a $C_{\alpha,\beta}\geq0$ such that
\begin{equation}\label{eq:seminorms}
|\partial_{x}^{\beta}\partial_{\eta}^{\alpha}a(x,\eta)|\leq C_{\alpha,\beta}\lb\eta\rb^{m-\rho|\alpha|+\delta|\beta|}
\end{equation}
for all $(x,\eta)\in\R^{2n}$. In \cite{HaPoRo20} a slightly different symbol class was considered, the elements of which have additional regularity when differentiated in the radial direction in the fiber variable. This class, denoted by $S^{m}_{\rho,\delta,1}$, consists of those $a\in C^{\infty}(\R^{2n})$ such that, for all $\alpha,\beta\in\Z^{n}_{+}$ and $\gamma\in\Z_{+}$, there exists a $C_{\alpha,\beta,\gamma}\geq0$ such that
\begin{equation}\label{eq:seminorms2}
|(\hat{\eta}\cdot \partial_{\eta})^{\gamma}\partial_{x}^{\beta}\partial_{\eta}^{\alpha}a(x,\eta)|\leq C_{\alpha,\beta,\gamma}\lb\eta\rb^{m-\rho|\alpha|+\delta|\beta|-\gamma}
\end{equation}
for all $(x,\eta)\in\R^{2n}$ with $\eta\neq0$. Note that $S^{m}_{1/2,1/2,1}$ contains $S^{m}_{1,1/2}$ but is strictly contained in $S^{m}_{1/2,1/2}$, which in turn is the critical symbol class for the calculus of Fourier integral operators, in several respects. 
 
In the following definition and throughout the rest of this article, 
\[
o:=\Rn\times\{0\}\subseteq \R^{2n}
\]
denotes the zero section.

\begin{definition}\label{def:operator}
Let $m\in\R$, $a\in S^{m}_{1/2,1/2,1}$ and $\Phi\in C^{\infty}(\R^{2n}\setminus o)$ be such that $\Phi$ is real-valued and positively homogeneous of degree $1$ in the $\eta$ variable, and $\det \partial^2_{x \eta} \Phi (x,\eta)\neq 0$ for $(x,\eta)\in\supp(a)\setminus o$. Set
\[
Tf(x):= (2\pi)^{-n} \int_{\Rn}e^{i\Phi(x,\eta)}a(x,\eta)\wh{f}(\eta)\ud\eta
\]
for $f\in\Sw(\Rn)$ and $x\in \Rn$. Then $T$ is a Fourier integral operator of order $m$ and type $(1/2,1/2,1)$ in \emph{standard form}. If, additionally, the following conditions hold:
\begin{enumerate}
\item\label{it:phase2} $\sup_{(x,\eta)\in \R^{2n}\setminus o}|\partial_{x}^{\beta}\partial_{\eta}^{\alpha}\Phi(x,\hat{\eta})|<\infty$ for all $\alpha,\beta\in\Z_{+}^{n}$ with $|\alpha|+|\beta|\geq 2$;
\item\label{it:phase3} $\inf_{(x,\eta)\in \R^{2n}\setminus o}| \det \partial^2_{x \eta} \Phi (x,\eta)|>0$;
\item\label{it:phase4} $(\partial_{\eta}\Phi(x,\eta),\eta)\mapsto (x,\partial_{x}\Phi(x,\eta))$ is a well-defined bijection on $\R^{2n}\setminus o$,
\end{enumerate}
then we say that $T$ is associated with a \emph{global canonical graph}.
\end{definition}

\begin{remark}\label{rem:oscint} 
If \eqref{it:phase4} holds, then $(\partial_{\eta}\Phi(x,\eta),\eta)\mapsto (x,\partial_{x}\Phi(x,\eta))$ is a homogeneous canonical transformation on $\R^{2n}\setminus o$, and the canonical relation of $T$ is the graph of this transformation.

If \eqref{it:phase2} and \eqref{it:phase3} hold, then \eqref{it:phase4} holds if and only if $\eta\mapsto \partial_{x}\Phi(x,\eta)$ is a bijection on $\R^{n}\setminus \{0\}$ for each $x\in\Rn$, by Hadamard's global inverse function theorem \cite[Theorem 6.2.8]{Krantz-Parks13}. Another application of the global inverse function theorem then shows that condition \eqref{it:phase4} is superfluous for $n\geq3$. Moreover, \eqref{it:phase4} holds if $\Phi(x,\eta)=x\cdot\eta+\phi(\eta)$ for some $\phi\in C^{\infty}(\Rn\setminus\{0\})$ which is positively homogeneous of degree $1$. This case will be considered frequently by us, and it is characterized by the property that $\partial_{x}\Phi(x,\eta)=\eta$ for all $(x,\eta)\in\R^{2n}\setminus o$.
\end{remark}

Recall that a compactly supported Fourier integral operator of order $m$ and type $(1,0)$, associated with a local canonical graph, can, modulo an operator with a Schwartz kernel which is a Schwartz function, be expressed as a finite sum of operators which in appropriate coordinate systems are Fourier integral operators in standard form (see e.g. \cite[Proposition 6.2.4]{Sogge17}), with symbols  in $S^{m}_{1,0}$. In this case, the symbol $a$ has compact support in the $x$-variable, and thus \eqref{it:phase2} and \eqref{it:phase3} hold on the support of $a$, while the map in \eqref{it:phase4} is a locally well-defined homogeneous canonical transformation. By contrast, for operators associated with a global canonical graph as in Definition \ref{def:operator}, the symbols are not required to have compact spatial support, but the conditions on the phase function hold uniformly on all of $\R^{2n}\setminus o$.

\section{Wave packet transforms}\label{sec:transforms}

In this section we introduce the parabolic frequency localizations that appear in the definition of the function spaces for decoupling, and we use them to define associated wave packets and wave packet transforms. We then prove some properties of these transforms, most notably the boundedness of certain operators on phase space associated with them.

\subsection{Wave packets}\label{subsec:packets}

Throughout this article, we fix a family $(\ph_{\w})_{\w\in S^{n-1}}\subseteq C^{\infty}(\Rn)$ of non-negative functions with the following
properties:
\begin{enumerate}
\item\label{it:phiproperties1} For all $\w\in S^{n-1}$ and $\xi\neq0$, one has $\ph_{\w}(\xi)=0$ if $|\xi|<\frac{1}{4}$ or $|\hat{\xi}-\w|>|\xi|^{-1/2}$.
\item\label{it:phiproperties2} For all $\alpha\in\Z_{+}^{n}$ and $\gamma\in\Z_{+}$, there exists a $C_{\alpha,\gamma}\geq0$ such that
\[
|(\w\cdot \partial_{\xi})^{\gamma}\partial^{\alpha}_{\xi}\ph_{\w}(\xi)|\leq C_{\alpha,\gamma}|\xi|^{\frac{n-1}{4}-\frac{|\alpha|}{2}-\gamma}
\]
for all $\w\in S^{n-1}$ and $\xi\neq0$.
\item\label{it:phiproperties3}
The map $(\xi,\w)\mapsto \ph_{\w}(\xi)$ is measurable on $\Rn\times S^{n-1}$, and
\begin{equation}\label{eq:reproduce1}
\int_{S^{n-1}}\ph_{\w}(\xi)^{2}\ud\w=1
\end{equation}
for all $\xi\in\Rn$ with $|\xi|\geq 1$.
\end{enumerate}

The picture below depicts the support of $\varphi_{\omega}$. For $R \geq 1$, the intersection of the annulus $\{\xi\in\Rn\mid |\xi| \eqsim R\}$ with the support of $\varphi_{\omega}$ is a rectangle of dimension 
\[
R\times \underbrace{R^{1/2} \times \dots \times R^{1/2}}_{\text{$n-1$  directions}}
\]
with the long side pointing in the direction of $\omega$. One has $\varphi_{\omega}(\xi)\eqsim R^{(n-1)/4}$ on such a rectangle.
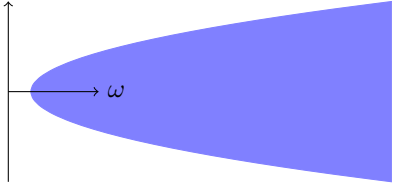
\begin{center}
	\begin{tikzpicture}[scale = 0.3]
		\filldraw[blue!50] plot[smooth,domain=-4:4] ({1+(\x)^2}, {\x});
		\draw[->] (0,0) -- (4,0) node[right] {$\omega$};
		\draw[->] (0,-4) -- (0,4);
	\end{tikzpicture}
\end{center}

\begin{remark}\label{rem:parloc}
For the construction of such a collection, see e.g.~\cite[Section 3.1]{Rozendaal21}. In fact, the functions constructed there have slightly different support properties than in \eqref{it:phiproperties1}. However, the construction in \cite[Section 3.1]{Rozendaal21} also yields functions satisfying \eqref{it:phiproperties1}, if one shrinks the support of the function $\ph$ used there, and slightly modifies the definition of $\ph_{\w}$ itself. Moreover, instead of \eqref{eq:reproduce1}, from \cite[Remark 3.3]{Rozendaal21} one obtains
\[
\int_{S^{n-1}}m(\xi)\ph_{\w}(\xi)^{2}\ud \w=1\quad(|\xi|\geq 1/2)
\]
for a standard symbol $m\in S^{0}(\Rn)$. However, the proof of \cite[Lemma 3.2]{Rozendaal21} also shows that $\sqrt{m}\in S^{0}(\Rn)$, so that one may replace $\ph_{\w}$ by $\sqrt{m}\ph_{\w}$ to arrive at  \eqref{eq:reproduce1}. 
\end{remark}

The exact support properties of the $\ph_{\w}$ are not relevant for this article; all the arguments go through if the conditions in \eqref{it:phiproperties1} are modified up to fixed constants. Note also that, by Corollary \ref{cor:equivnorm},
\begin{equation}\label{eq:phisuppomega}
\supp(\ph_{\w})\subseteq\{\xi\in\Rn\mid \tfrac{1}{4}\leq |\xi|_{\w}, \tfrac{1}{2\sqrt{5}}|\xi|^{1/2}\leq |\xi|_{\w}\leq 2|\xi|^{1/2}\}
\end{equation}
for all $\w\in S^{n-1}$.

Next, we use the collection $(\ph_{\w})_{\w\in S^{n-1}}$ to construct wave packets. For $\w\in S^{n-1}$, $\sigma>0$ and $\xi\in\Rn$, set 
\begin{equation}\label{eq:psiwsigma}
\psi_{\w,\sigma}(\xi):=\Psi(\sigma\xi)\ph_{\w}(\xi),
\end{equation}
where $\Psi\in C^{\infty}_{c}(\Rn)$ is as in \eqref{eq:PsiRn}. Also let
\begin{equation}\label{eq:rhozero}
\rho_{0}(\xi):=\Big(1-\int_{0}^{1}\int_{S^{n-1}}\psi_{\w,\sigma}(\xi)^{2}\ud\w\frac{\ud\sigma}{\sigma}\Big)^{1/2}.
\end{equation}
We may assume that $\rho_{0}\in C^{\infty}_{c}(\Rn)$, as in \cite[Section 4]{HaPoRo20}. 

These functions have the following properties, similar to those in \cite[Lemma 4.1]{HaPoRo20}.

\begin{lemma}\label{lem:packetbounds}
For all $\w\in S^{n-1}$ and $\sigma>0$, one has $\psi_{\w,\sigma}\in C^{\infty}_{c}(\Rn)$. Each $\xi\in\supp(\psi_{\w,\sigma})$ satisfies $\frac{1}{2}\sigma^{-1}\leq |\xi|\leq 2\sigma^{-1}$, $|\hat{\xi}-\w|\leq \sqrt{2\sigma}$, $|\xi|_{\w}\geq 1/4$ and
\[
\tfrac{1}{2\sqrt{10}}\sigma^{-1/2}\leq |\xi|_{\w}\leq \sqrt{8}\sigma^{-1/2}.
\]
Moreover, 
\[
\int_{0}^{\infty}\int_{S^{n-1}}\psi_{\w,\sigma}(\xi)^{2}\ud\w\frac{\ud\sigma}{\sigma}=1
\]
for all $\xi\in\Rn$ with $|\xi|\geq 1$. 
For all $\alpha\in\Z_{+}^{n}$ and $\gamma\in\Z_{+}$, there exists a constant $C_{\alpha,\gamma}\geq0$ such that
\[
|(\w\cdot\partial_{\xi})^{\beta}\partial_{\xi}^{\alpha}\psi_{\w,\sigma}(\xi)|\leq C_{\alpha,\gamma}\sigma^{-\frac{n-1}{4}+\frac{|\alpha|}{2}+\gamma}
\]
for all $\w\in S^{n-1}$, $\sigma>0$ and $\xi\in\Rn$. Also, for each $N\geq0$ there exists a $C_{N}\geq0$ such that, for all $\w\in S^{n-1}$, $\sigma>0$ and $x\in\Rn$,
\begin{equation}\label{eq:boundspsiinverse}
|\F^{-1}(\psi_{\w,\sigma})(x)|\leq C_{N}\sigma^{-\frac{3n+1}{4}}(1+\sigma^{-1}|x|_{\w}^{2})^{-N}.
\end{equation}
In particular, $\{\sigma^{\frac{n-1}{4}}\F^{-1}(\psi_{\w,\sigma})\mid \w\in S^{n-1},\sigma>0\}\subseteq L^{1}(\Rn)$ is uniformly bounded.
Finally, $\rho_{0}\in C^{\infty}_{c}(\Rn)$ satisfies $\rho_{0}(\xi)=0$ for $|\xi|\geq 2$, and $\rho_{0}(\xi)=1$ for $|\xi|\leq 1/2$.
\end{lemma}
Note in particular that $\{\psi_{\w,\sigma}\mid \w\in S^{n-1},\sigma>0\}$ is a uniformly bounded collection in the symbol class $S^{(n-1)/4}_{1/2,1/2,1}$ from \eqref{eq:seminorms2}.

\begin{proof}
The proof is completely analogous to that of \cite[Lemma 4.1]{HaPoRo20}. The first few and the last statement follow from properties \eqref{it:phiproperties1}, \eqref{it:phiproperties2} and \eqref{it:phiproperties3} of $(\ph_{\w})_{\w\in S^{n-1}}$, together with \eqref{eq:phisuppomega} and the properties of $\Psi$. One then obtains \eqref{eq:boundspsiinverse} from integration by parts, using also the equivalent expression for $|\cdot|_{\w}$ from Lemma \ref{lem:equivnorm}.
\end{proof}

\subsection{Wave packet transforms}\label{subsec:transforms}

In this subsection we introduce wave packet transforms that lift functions on $\Rn$ to 
\[
\Spp:=S^{*}\Rn\times\R_{+}=\Rn\times S^{n-1}\times(0,\infty),
\]
endowed with the measure $\ud x\ud\w\frac{\ud\sigma}{\sigma}$. Note that the map $(x,\w,\sigma)\mapsto (x,\sigma^{-1}\w)$ identifies $\Spp$ with $\Rn\times (\Rn\setminus\{0\})$, i.e.~phase space minus the zero section.

To properly define our wave packet transforms, it will be convenient to work with a class of test functions on $\Spp$ and the associated distributions. As in \cite[Section 2.2]{HaPoRo20}, we let $\J(\Spp)$ consist of those $G\in L^{\infty}(\Spp)$ such that 
\[
(x,\w,\sigma)\mapsto (1+|x|+\max(\sigma,\sigma^{-1}))^{N}G(x,\w,\sigma)
\]
is an element of $L^{\infty}(\Spp)$ for all $N\geq0$, endowed with the topology generated by the corresponding weighted $L^{\infty}$ norms. 
Let $\J'(\Spp)$ be the space of continuous linear $F:\J(\Spp)\to \C$, endowed with the topology induced by $\J(\Spp)$. 
We denote the duality between $F\in\J'(\Spp)$ and $G\in\J(\Spp)$ by $\lb F,G\rb_{\Spp}$. 
If $F\in L^{1}_{\loc}(\Spp)$ is such that 
\[
G\mapsto \int_{\Spp}F(x,\w,\sigma)G(x,\w,\sigma)\ud x\ud\w\frac{\ud\sigma}{\sigma}
\]
defines an element of $\J'(\Spp)$, then we simply write $F\in\J'(\Spp)$. 

For $p,q\in[1,\infty)$, we will often work with the function space
\begin{equation}\label{eq:Lqp}
\Lqp:=L^{q}(S^{n-1};L^{p}(\Rn;L^{2}(\R_{+})))
\end{equation}
on $\Spp$, consisting of all $F\in L^{1}_{\loc}(\Spp)$ such that
\[
\|F\|_{\Lqp}:=\Big(\int_{S^{n-1}}\Big(\int_{\Rn}\Big(\int_{0}^{\infty}|F(x,\w,\sigma)|^{2}\frac{\ud\sigma}{\sigma}\Big)^{p/2}\ud x\Big)^{q/p}\ud\w\Big)^{1/q}<\infty.
\]
It is straightforward to check (see also \cite[Lemma 2.10]{HaPoRo20}) that
\begin{equation}\label{eq:inclusionLqp}
\J(\Spp)\subseteq \Lqp\subseteq \J'(\Spp)
\end{equation}
continuously, 
and that the first inclusion is dense.

We can now define our wave packet transform, in a similar manner as in \cite{HaPoRo20}. For $f\in\Sw'(\Rn)$ and $(x,\w,\sigma)\in\Spp$, set
\begin{equation}\label{eq:W}
Wf(x,\w,\sigma):=\begin{cases}
\psi_{\w,\sigma}(D)f(x)&\text{if }0<\sigma<1,\\
\ind_{[1,e]}(\sigma)\rho_{0}(D)f(x)&\text{if }\sigma\geq1.
\end{cases}
\end{equation}
Next, for $F\in \J(\Spp)$ and $x\in\Rn$, set
\begin{equation}\label{eq:V}
\begin{aligned}
VF(x):=&\int_{0}^{1}\int_{S^{n-1}}\psi_{\nu,\tau}(D)F(\cdot,\nu,\tau)(x)\ud\nu\frac{\ud\tau}{\tau}\\
&+\int_{1}^{e}\int_{S^{n-1}}\rho_{0}(D)F(\cdot,\nu,\tau)(x)\ud\nu \frac{\ud\tau}{\tau}.
\end{aligned}
\end{equation}
These transforms have the following properties.

\begin{proposition}\label{prop:bddwavetransform}
The following statements hold:
\begin{enumerate}
\item\label{it:W1} $W:L^{2}(\Rn)\to L^{2}(\Spp)$ is an isometry;
\item\label{it:W2} $W:\Sw(\Rn)\to \J(\Spp)$ and $W:\Sw'(\Rn)\to \J'(\Spp)$ are continuous;
\item\label{it:W3} $V:\J(\Spp)\to \Sw(\Rn)$ is continuous, and $V$ extends uniquely to a continuous map $V:\J'(\Spp)\to \Sw'(\Rn)$;
\item\label{it:W4} $\lb VF,g\rb_{\Rn}=\lb F,Wg\rb_{\Spp}$ for all $F\in \J'(\Spp)$ and $g\in\Sw(\Rn)$, and $\lb Wf,G\rb_{\Spp}=\lb f,VG\rb_{\Rn}$ for all $f\in\Sw'(\Rn)$ and $G\in \J(\Spp)$;
\item\label{it:W5} $VWf=f$ for all $f\in\Sw'(\Rn)$.
\end{enumerate}
\end{proposition}
\begin{proof}
The statement and the proof are essentially contained in \cite[Lemma 4.3]{HaPoRo20}. More precisely, \eqref{it:W1} follows from the definition of $\rho_{0}$ in \eqref{eq:rhozero}
, given that we have assumed that the surface measure $\ud\w$ on $S^{n-1}$ is unit normalized. Moreover, it is proved in \cite[Lemma 4.3]{HaPoRo20} that $W:\Sw(\Rn)\to \J(\Spp)$ and $V:\J(\Spp)\to \Sw(\Rn)$ are continuous. The remaining statements follow upon confirming that $\lb Wf,G\rb_{\Spp}=\lb f,VG\rb_{\Rn}$ for all $f\in\Sw'(\Rn)$ and $G\in\J(\Spp)$.
\end{proof}

\subsection{Operators on phase space}\label{subsec:operatorsphase}

In this subsection we prove a key result about the boundedness of certain operators on the function space $\Lqp$ from \eqref{eq:Lqp}.

For use in the next section, it will be convenient to formulate our result in terms of wave packet transforms associated with (possibly) different parabolic frequency localizations. More precisely, let $(\wt{\ph}_{\w})_{\w\in S^{n-1}}\subseteq C^{\infty}(\Rn)$ be a family of real-valued functions with the same properties \eqref{it:phiproperties1}, \eqref{it:phiproperties2} and \eqref{it:phiproperties3} that $(\ph_{\w})_{\w\in S^{n-1}}$ has, from Section \ref{subsec:packets}. Let $\wt{\Psi}\in C^{\infty}_{c}(\Rn)$ be real-valued, with $\supp(\wt{\Psi})\subseteq\{\xi\in\Rn\mid 1/2\leq |\xi|\leq 2\}$, and such that \eqref{eq:Psi} holds with $\Psi$ replaced by $\wt{\Psi}$. For $\w\in S^{n-1}$, $\sigma>0$ and $\xi\in\Rn$, set $\wt{\psi}_{\w,\sigma}(\xi):=\wt{\Psi}(\sigma\xi)\wt{\ph}_{\w}(\xi)$, and define $\wt{\rho}_{0}\in C^{\infty}_{c}(\Rn)$ by
\begin{equation}\label{eq:wtrho}
\wt{\rho}_{0}(\xi):=\Big(1-\int_{0}^{1}\int_{S^{n-1}}\wt{\psi}_{\w,\sigma}(\xi)^{2}\ud\w\frac{\ud\sigma}{\sigma}\Big)^{1/2}.
\end{equation}
Now set, for $f\in\Sw'(\Rn)$ and $(x,\w,\sigma)\in\Spp$,
\begin{equation}\label{eq:wtW}
\wt{W}f(x,\w,\sigma):=\begin{cases}
\wt{\psi}_{\w,\sigma}(D)f(x)&\text{if }0<\sigma<1,\\
\ind_{[1,e]}(\sigma)\wt{\rho}_{0}(D)f(x)&\text{if }\sigma\geq1.
\end{cases}
\end{equation}
Similarly, for $F\in \J(\Spp)$ and $x\in\Rn$, write
\begin{equation}\label{eq:wtV}
\begin{aligned}
\wt{V}F(x):=&\int_{0}^{1}\int_{S^{n-1}}\wt{\psi}_{\nu,\tau}(D)F(\cdot,\nu,\tau)(x)\ud\nu\frac{\ud\tau}{\tau}\\
&+\int_{1}^{e}\int_{S^{n-1}}\wt{\rho}_{0}(D)F(\cdot,\nu,\tau)(x)\ud\nu\frac{\ud\tau}{\tau}.
\end{aligned}
\end{equation}
These transforms have the same properties as $W$ and $V$, from Proposition \ref{prop:bddwavetransform}.


\begin{theorem}\label{thm:FIOconj}
Let $p,q\in(1,\infty)$. Then 
\[
W\wt{V}:\Lqp\to \Lqp
\]
is bounded.
\end{theorem}
\begin{proof}
By density of $\J(\Spp)$ in $\Lqp$, it suffices to show that 
\[
\|W\wt{V}F\|_{\Lqp}\lesssim \|F\|_{\Lqp}
\]
for all $F\in\J(\Spp)$. Fix $F\in\J(\Spp)$, so that all the quantities which appear below are a priori well defined.

Note that
\[
W\wt{V}F = I + II + III + IV,
\]
where, for $(x,\w,\sigma)\in\Spp$,
\begin{align*}
I(x,\w,\sigma) &:= \ind_{(0,1)}(\sigma)\psi_{\w,\sigma}(D)\Big(\int_{0}^{1}\int_{S^{n-1}}\wt{\psi}_{\nu,\tau}(D)F(\cdot,\nu,\tau)\ud\nu\frac{\ud\tau}{\tau}\Big)(x),\\
II(x,\w,\sigma) &:= \ind_{(0,1)}(\sigma)\psi_{\w,\sigma}(D)\Big(\int_{1}^{e}\int_{S^{n-1}}\wt{\rho}_{0}(D)F(\cdot,\nu,\tau)\ud\nu\frac{\ud\tau}{\tau}\Big)(x),\\
III(x,\w,\sigma) &:= \ind_{[1,e]}(\sigma)\rho_{0}(D)\Big(\int_{0}^{1}\int_{S^{n-1}}\wt{\psi}_{\nu,\tau}(D)F(\cdot,\nu,\tau)\ud\nu\frac{\ud\tau}{\tau}\Big)(x),\\
IV(x,\w,\sigma) &:= \ind_{[1,e]}(\sigma)\rho_{0}(D)\Big(\int_{1}^{e}\int_{S^{n-1}}\wt{\rho}_{0}(D)F(\cdot,\nu,\tau)\ud\nu\frac{\ud\tau}{\tau}\Big)(x).
\end{align*}
We will bound each of these terms individually. 

\subsubsection{Term $I$}

For $(x,\w,\sigma)\in\Spp$, we start by using the support properties of $\psi_{\w,\sigma}$ and $\wt{\psi}_{\nu,\tau}$, from Lemma \ref{lem:packetbounds}, and a change of variables:
\begin{equation}\label{eq:FIOconj1}
\begin{aligned}
| I(x,\w,\sigma) |
& = \Big|\ind_{(0,1)}(\sigma){\psi}_{\w,\sigma}(D)\Big(\int_{1/4}^{4}\int_{S^{n-1}}\wt{\psi}_{\nu,\alpha \sigma}(D)F(\cdot,\nu,\alpha \sigma)\ud\nu\frac{\ud\alpha}{\alpha}\Big)(x)\Big|\\
&\lesssim \int_{1/4}^{4}| \sigma^{\frac{n-1}{2}} {\psi}_{\w,\sigma}(D) F_{I}(\cdot,\w,\alpha\sigma)(x)| \frac{\ud\alpha}{\alpha},
\end{aligned}
\end{equation}
where
\[
F_{I}(y,\w,\alpha\sigma):=\ind_{(0,1)}(\sigma)\fint_{|\nu-\omega| \leq 6\sqrt{\sigma}}  \wt{\psi}_{\nu,\alpha \sigma}(D)F(\cdot,\nu,\alpha \sigma)(y)\ud\nu
\]
for $\alpha\in[1/4,4]$ and $y\in\Rn$. Next, note that the kernel bounds for $\psi_{\w,\sigma}(D)$ from \eqref{eq:boundspsiinverse} yield
\begin{align*}
|\psi_{\w,\sigma}(D)F_{I}(\cdot,\w,\alpha\sigma)(x)|&\leq \int_{\Rn}|\F^{-1}(\psi_{\w,\sigma})(y)F_{I}(x-y,\w,\alpha\sigma)|\ud y\\
&\lesssim \sigma^{-\frac{n-1}{4}}\int_{\Rn}\sigma^{-\frac{n+1}{2}}\frac{|F_{I}(x-y,\w,\alpha\sigma)|}{(1+\sigma^{-1}|y|^{2}_{\w})^{N}}\ud y,
\end{align*}
for $N\geq0$. 

If we choose $N>(n+1)/2$, then \eqref{eq:FIOconj1}, the triangle inequality and Corollary \ref{cor:maximal} show that
\begin{equation}\label{eq:FIOconj2}
\begin{aligned}
&\|I(\cdot,\w,\cdot)\|_{L^{p}(\Rn;L^{2}(\R_{+}))}\\
&\leq \int_{1/4}^{4}\Big(\int_{\Rn}\Big(\int_{0}^{\infty}\big|\sigma^{\frac{n-1}{2}}\psi_{\w,\sigma}(D)F_{I}(\cdot,\w,\alpha\sigma)(x)|^{2}\frac{\ud\sigma}{\sigma}\Big)^{p/2}\ud x\Big)^{1/p}\frac{\ud\alpha}{\alpha}\\
&\lesssim \int_{1/4}^{4}\Big(\int_{\Rn}\Big(\int_{0}^{\infty}\sigma^{\frac{n-1}{2}}|F_{I}(x,\w,\alpha\sigma)|^{2}\frac{\ud\sigma}{\sigma}\Big)^{p/2}\ud x\Big)^{1/p}\frac{\ud\alpha}{\alpha}\\
&\leq \int_{1/4}^{4}\Big(\int_{\Rn}\Big(\int_{0}^{\infty}\Big(\fint_{|\nu-\w|\leq 6\sqrt{\sigma}}|F_{I,\alpha}(x,\nu,\sigma)|\ud\nu\Big)^{2}\frac{\ud\sigma}{\sigma}\Big)^{p/2}\ud x\Big)^{1/p}\frac{\ud\alpha}{\alpha}.
\end{aligned}
\end{equation}
Here
\[
F_{I,\alpha}(x,\nu,\sigma):=\ind_{(0,1)}(\sigma)\sigma^{\frac{n-1}{4}}\wt{\psi}_{\nu,\alpha\sigma}(D)F(\cdot,\nu,\alpha\sigma)(x)
\]
for $\alpha\in[1/4,4]$ and $(x,\nu,\sigma)\in \Spp$. 


Next, set
\[
M_{S^{n-1}}F_{I,\alpha}(x,\w,\sigma):=\sup_{B}\fint_{B}|F_{I,\alpha}(x,\nu,\sigma)|\ud\nu
\]
for $(x,\w,\sigma)\in\Spp$, where the supremum is taken over all balls $B\subseteq S^{n-1}$ containing $\w$. That is, $M_{S^{n-1}}$ is the Hardy--Littlewood maximal function on $S^{n-1}$. Then \eqref{eq:FIOconj2} and the triangle inequality yield
\begin{align*}
&\|I\|_{L^{q}_{\w}L^{p}_{x}L^{2}_{\sigma}(\Spp)}=\Big(\int_{S^{n-1}}\|I(\cdot,\w,\cdot)\|_{L^{p}(\Rn;L^{2}(\R_{+}))}^{q}\ud\w\Big)^{1/q}\\
&\lesssim \Big(\int_{S^{n-1}}\Big(\int_{1/4}^{4}\Big(\int_{\Rn}\Big(\int_{0}^{\infty}|M_{S^{n-1}}F_{I,\alpha}(x,\w,\sigma)|^{2}\frac{\ud\sigma}{\sigma}\Big)^{p/2}\ud x\Big)^{1/p}\frac{\ud\alpha}{\alpha}\Big)^{q}\ud\w\Big)^{1/q}\\
&\leq \int_{1/4}^{4}\|M_{S^{n-1}}F_{I,\alpha}\|_{\Lqp}\frac{\ud\alpha}{\alpha}.
\end{align*}
Since $p\in(1,\infty)$, the space $X:=L^{p}(\Rn;L^{2}(\R_{+}))$ is a Banach lattice with the UMD property (see e.g.~\cite[Section 4.2.c]{HyNeVeWe16}). Given that we have $q\in(1,\infty)$ as well, \cite[Theorem 1.5]{Tozoni04} (see also \cite{RubiodeFrancia86}) then implies that $M_{S^{n-1}}$ is bounded on $\Lqp=L^{q}(S^{n-1};X)$. We now find
\begin{equation}\label{eq:FIOconj3}
\|I\|_{\Lqp}\lesssim \int_{1/4}^{4}\|F_{I,\alpha}\|_{\Lqp}\frac{\ud\alpha}{\alpha},
\end{equation}
for an implicit constant independent of $F$.

Finally, for all $(x,\w,\sigma)\in\Spp$ and $\alpha\in[1/4,4]$, we can use the bounds for $\F^{-1}(\wt{\psi}_{\w,\alpha\sigma})$ from \eqref{eq:boundspsiinverse} to write
\[
|F_{I,\alpha}(x,\w,\sigma)|\lesssim \int_{\Rn}\sigma^{-\frac{n+1}{2}}\frac{|F(x-y,\w,\alpha\sigma)|}{(1+\sigma^{-1}|y|^{2}_{\w})^{N}}\ud y.
\]
So, in the same way as before, Corollary \ref{cor:maximal} can be applied to \eqref{eq:FIOconj3}, yielding
\begin{align*}
\|I\|_{L^{q}_{\w}L^{p}_{x}L^{2}_{\sigma}(\Spp)}&\lesssim \int_{1/4}^{4}\Big(\int_{S^{n-1}}\|F_{I,\alpha}(\cdot,\w,\cdot)\|^{q}_{L^{p}(\Rn;L^{2}(\R_{+}))}\ud\w\Big)^{1/q}\frac{\ud\alpha}{\alpha}\\
&\lesssim \int_{1/4}^{4}\Big(\int_{S^{n-1}}\|F(\cdot,\w,\alpha\cdot)\|^{q}_{L^{p}(\Rn;L^{2}(\R_{+}))}\ud\w\Big)^{1/q}\frac{\ud\alpha}{\alpha}\\
&\lesssim \|F\|_{\Lqp},
\end{align*}
where for the final inequality we used another change of variables. This takes care of $I$, the main term, and we now turn to the remaining terms.

\subsubsection{Term $II$} Write 
\[
F_{II}(x,\sigma):=\ind_{(1/4,1)}(\sigma)\int_{1}^{e}\int_{S^{n-1}}\wt{\rho}_{0}(D)F(\cdot,\nu,\tau)(x)\ud\nu\frac{\ud\tau}{\tau}
\]
for $x\in\Rn$ and $\sigma>0$. Then the support properties of $\psi_{\w,\sigma}$ and $\wt{\rho}_{0}$ from Lemma \ref{lem:packetbounds}, together with the kernel bounds from \eqref{eq:boundspsiinverse}, imply that
\begin{align*}
|II(x,\w,\sigma)|&=|\psi_{\w,\sigma}(D)F_{II}(\cdot,\sigma)(x)|\lesssim \sigma^{-\frac{n-1}{4}}\int_{\Rn}\sigma^{-\frac{n+1}{2}}\frac{|F_{II}(x-y,\sigma)|}{(1+\sigma^{-1}|y|_{\w}^{2})^{N}}\ud y\\
&\eqsim \int_{\Rn}\sigma^{-\frac{n+1}{2}}\frac{|F_{II}(x-y,\sigma)|}{(1+\sigma^{-1}|y|_{\w}^{2})^{N}}\ud y
\end{align*}
for all $\w\in S^{n-1}$. In the last step we used that $F_{II}(\cdot,\sigma)=0$ if $\sigma\notin (1/4,1)$.

Letting $N>(n+1)/2$, Corollary \ref{cor:maximal} now yields
\[
\|II(\cdot,\w,\cdot)\|_{L^{p}(\Rn;L^{2}(\R_{+}))}\lesssim \|F_{II}\|_{L^{p}(\Rn;L^{2}(\R_{+}))}
\]
for every $\w\in S^{n-1}$. Hence
\begin{align*}
&\|II\|_{\Lqp}=\Big(\int_{S^{n-1}}\|II(\cdot,\w,\cdot)\|_{L^{p}(\Rn;L^{2}(\R_{+}))}^{q}\ud\w\Big)^{1/q}\lesssim \|F_{II}\|_{L^{p}(\Rn;L^{2}(\R_{+}))}\\
&=\Big(\int_{\Rn}\Big(\int_{1/4}^{1}\frac{\ud\sigma}{\sigma}\Big)^{p/2}\Big|\int_{1}^{e}\int_{S^{n-1}}\wt{\rho}_{0}(D)F(\cdot,\nu,\tau)(x)\ud\nu\frac{\ud\tau}{\tau}\Big|^{p}\ud x\Big)^{1/p}\\
&\eqsim \|\wt{\rho}_{0}(D)\wt{F}_{II}\|_{L^{p}(\Rn)},
\end{align*}
where
\[
\wt{F}_{II}(y):=\int_{1}^{e}\int_{S^{n-1}}F(y,\nu,\tau)\ud\nu\frac{\ud\tau}{\tau}
\]
for $y\in\Rn$. 

Finally, we can observe that $\wt{\rho}_{0}(D)$ acts boundedly on $L^{p}(\Rn)$, use the triangle inequality, and apply H\"{o}lder's inequality twice, to obtain
\begin{align*}
\|II\|_{\Lqp}&\lesssim \|\wt{\rho}_{0}(D)\wt{F}_{II}\|_{L^{p}(\Rn)}\lesssim \|\wt{F}_{II}\|_{L^{p}(\Rn)}\\
&\leq \int_{S^{n-1}}\Big(\int_{\Rn}\Big(\int_{1}^{e}|F(y,\nu,\tau)|\frac{\ud\tau}{\tau}\Big)^{p}\ud y\Big)^{1/p}\ud\nu\\
&\leq \Big(\int_{S^{n-1}}\Big(\int_{\Rn}\Big(\int_{1}^{e}|F(y,\nu,\tau)|^{2}\frac{\ud\tau}{\tau}\Big)^{p/2}\ud y\Big)^{q/p}\ud\nu\Big)^{1/q}\\\
&\leq \|F\|_{\Lqp}.
\end{align*}
This is the required bound for term $II$.

\subsubsection{Term $III$}

Note that the dependence of $III(x,\w,\sigma)$ on $\w$ and $\sigma$ is essentially trivial. Hence we immediately obtain
\[
\|III\|_{\Lqp}=\Big(\int_{\Rn}\Big|\rho_{0}(D)\Big(\int_{0}^{1}\int_{S^{n-1}}\wt{\psi}_{\nu,\tau}(D)F(\cdot,\nu,\tau)\ud\nu\frac{\ud\tau}{\tau}\Big)(x)\Big|^{p}\ud x\Big)^{1/p}.
\]
Due to the support properties of $\rho_{0}$ and $\wt{\psi}_{\nu,\tau}$, and the fact that $\rho_{0}(D)$ is bounded on $L^{p}(\Rn)$, we obtain from this the inequalities
\begin{align*}
&\|III\|_{\Lqp}\lesssim \Big(\int_{\Rn}\Big|\int_{1/4}^{1}\int_{S^{n-1}}\wt{\psi}_{\nu,\tau}(D)F(\cdot,\nu,\tau)(x)\ud\nu\frac{\ud\tau}{\tau}\Big|^{p}\ud x\Big)^{1/p}\\
&\lesssim \int_{S^{n-1}}\Big(\int_{\Rn}\Big(\int_{1/4}^{1}|\wt{\psi}_{\nu,\tau}(D)F(\cdot,\nu,\tau)(x)|\frac{\ud\tau}{\tau}\Big)^{p}\ud x\Big)^{1/p}\ud\nu\\
&\lesssim \Big(\int_{S^{n-1}}\Big(\int_{\Rn}\Big(\int_{1/4}^{1}|\wt{\psi}_{\nu,\tau}(D)F(\cdot,\nu,\tau)(x)|^{2}\frac{\ud\tau}{\tau}\Big)^{p/2}\ud x\Big)^{q/p}\ud\nu\Big)^{1/q}.
\end{align*}
Here we also used the triangle inequality, and twice H\"{o}lder's inequality.

As before, for each $\nu\in S^{n-1}$ one can combine the kernel bounds for $\wt{\psi}_{\nu,\tau}(D)$ with Corollary \ref{cor:maximal} to write
\[
\Big(\int_{\Rn}\Big(\int_{1/4}^{1}|\wt{\psi}_{\nu,\tau}(D)F(\cdot,\nu,\tau)(x)|^{2}\frac{\ud\tau}{\tau}\Big)^{p/2}\ud x\Big)^{1/p}\lesssim \|F(\cdot,\nu,\cdot)\|_{L^{p}(\Rn;L^{2}(\R_{+}))}.
\]
Due to what we have already shown, this proves the required statement:
\[
\|III\|_{\Lqp}\lesssim \Big(\int_{S^{n-1}}\|F(\cdot,\nu,\cdot)\|_{L^{p}(\Rn;L^{2}(\R_{+}))}^{q}\ud\nu\Big)^{1/q}=\|F\|_{\Lqp}.
\]

\subsubsection{Term $IV$}

Here we can combine arguments used for the previous two terms. Again, the dependence of $IV(x,\w,\sigma)$ on $\omega$ and $\sigma$ is essentially trivial, and $\rho_{0}(D)\wt{\rho}_{0}(D)$ is bounded on $L^{p}(\Rn)$. Together with the triangle inequality and two applications of H\"{o}lder's inequality, these observations show that
\begin{align*}
&\|IV\|_{\Lqp}=\Big(\int_{\Rn}\Big|\rho_{0}(D)\wt{\rho}_{0}(D)\Big(\int_{1}^{e}\int_{S^{n-1}}F(\cdot,\nu,\tau)\ud\nu\frac{\ud\tau}{\tau}\Big)(x)\Big|^{p}\ud x\Big)^{1/p}\\
&\lesssim \Big(\int_{\Rn}\Big|\int_{1}^{e}\int_{S^{n-1}}F(y,\nu,\tau)\ud\nu\frac{\ud\tau}{\tau}\Big|^{p}\ud y\Big)^{1/p}\\
&\leq \int_{S^{n-1}}\Big(\int_{\Rn}\Big(\int_{1}^{e}|F(y,\nu,\tau)|\frac{\ud\tau}{\tau}\Big)^{p}\ud y\Big)^{1/p}\ud\nu\\
&\leq \Big(\int_{S^{n-1}}\Big(\int_{\Rn}\Big(\int_{1}^{e}|F(y,\nu,\tau)|^{2}\frac{\ud\tau}{\tau}\Big)^{p/2}\ud y\Big)^{q/p}\ud\nu\Big)^{1/q}=\|F\|_{\Lqp}.
\end{align*}
This deals with term $IV$ and concludes the proof.
\end{proof}

\section{Function spaces for decoupling}\label{sec:spaces}

In this section we first define our function spaces. We then connect these spaces to the wave packet transforms from Section \ref{sec:transforms}, and we use this connection to derive some of their basic properties, including interpolation and duality theorems.

\subsection{Definition of the spaces}\label{subsec:defspaces}


Recall that $(\ph_{\w})_{\w\in S^{n-1}}\subseteq C^{\infty}(\Rn)$ is a fixed collection of parabolic frequency localizations, introduced in Section \ref{subsec:packets}. Also, as in Definition \ref{def:localHardy}, $\rho\in C^{\infty}_{c}(\Rn)$ satisfies $\rho(\xi)=1$ for $|\xi|\leq 2$.

\begin{definition}\label{def:spaces}
Let $p,q\in[1,\infty)$ and $s\in\R$. Then $\funpqs$ consists of those $f\in\Sw'(\Rn)$ such that $\rho(D)f\in L^{p}(\Rn)$, $\ph_{\w}(D)f\in \HT^{s,p}(\Rn)$ for almost all $\w\in S^{n-1}$, and 
$(\int_{S^{n-1}}\|\ph_{\w}(D)f\|_{\HT^{s,p}(\Rn)}^{q}\ud\w)^{1/q}<\infty$, 
endowed with the norm
\begin{equation}\label{eq:normdef}
\|f\|_{\funpqs}:=\|\rho(D)f\|_{L^{p}(\Rn)}+\Big(\int_{S^{n-1}}\|\ph_{\w}(D)f\|_{\HT^{s,p}(\Rn)}^{q}\ud\w\Big)^{1/q}.
\end{equation}
\end{definition}
The motivation for the notation $\funpqs$ will become clear in the next subsection. It is straightforward to see that $\funpqs=\lb D\rb^{-s}\mathcal{L}_{W,0}^{q,p}(\Rn)$ for all $p,q\in[1,\infty)$ and $s\in\R$. Note also that one may replace $\HT^{s,p}(\Rn)$ by $\lb D\rb^{-s}H^{p}(\Rn)$ in Definition \ref{def:spaces}, since each $\ph_{\w}$ vanishes near zero. Then Proposition \ref{prop:equivpar} shows that in fact
\begin{equation}\label{eq:normdef2}
\|f\|_{\funpqs}\eqsim\|\rho(D)f\|_{L^{p}(\Rn)}+\Big(\int_{S^{n-1}}\|\lb D\rb^{s}\ph_{\w}(D)f\|_{H^{p}_{\w}(\Rn)}^{q}\ud\w\Big)^{1/q}
\end{equation}
for all $p,q\in[1,\infty)$, $s\in\R$ and $f\in \funpqs$.

\begin{remark}\label{rem:HpFIO}
For all $p\in[1,\infty)$ and $s\in\R$, the space $\mathcal{L}_{W,s}^{p,p}(\Rn)$ coincides with $\HT^{s,p}_{FIO}(\Rn)=\lb D\rb^{-s}\Hp$, as defined in \cite{HaPoRo20,Hassell-Rozendaal23}. 
This follows from equivalent characterizations of $\Hp$ proved in \cite{Rozendaal21,FaLiRoSo23}. 
\end{remark}

\subsection{Connection to wave packet transforms}\label{subsec:funwave}

To derive some of the basic properties of our function spaces, it will be convenient to characterize them using the wave packet transform $W$ from \eqref{eq:W}, and the spaces $\Lqp$ from \eqref{eq:Lqp}. The following proposition thus explains why we use the notation $\mathcal{L}_{W,s}^{q,p}(\Rn)$. The $\mathcal{L}_{W,0}^{q,p}(\Rn)$ norm of a function $f$ is equivalent to the norm in the Lebesgue space $\Lqp$ of the lifting of $f$ via the wave packet transform $W$. We then use the operator $\lb D\rb^{-s}$ to add the smoothness parameter $s$.

\begin{proposition}\label{prop:normtransform} 
Let $p,q\in[1,\infty)$. Then there exists a $C>0$ such that the following holds for all $f\in\Sw'(\Rn)$. One has $f\in\funpqzero$ if and only if $Wf\in \Lqp$, in which case
\[
\frac{1}{C}\|f\|_{\funpqzero}\leq \|Wf\|_{\Lqp}\leq C\|f\|_{\funpqzero}.
\]
In particular, $\fun^{2,2}(\Rn)=W^{s,2}(\Rn)$ for all $s\in\R$, with equivalence of norms.
\end{proposition}
\begin{proof}
After dealing with the low-frequency terms, the first statement follows from the Littlewood--Paley characterization of $H^{p}(\Rn)$ in \eqref{eq:LittlePaley}. Then the second statement follows from Proposition \ref{prop:bddwavetransform} \eqref{it:W1}. 
\end{proof}

For the following corollary, recall the definition of the transform $V$ from \eqref{eq:V}.

\begin{corollary}\label{cor:Banach}
Let $p,q\in(1,\infty)$. Then $W:\funpqzero\to \Lqp$ is an isomorphism onto a complemented subspace, $V:\Lqp\to \funpqzero$ is bounded, and
\begin{equation}\label{eq:Visom}
V:\Lqp/\ker(V)\to \funpqzero
\end{equation}
is an isomorphism. In particular, $\funpqs$ is a Banach space for all $s\in\R$.
\end{corollary}
\begin{proof}
The first statement follows from Proposition \ref{prop:normtransform}, Theorem \ref{thm:FIOconj} and Proposition \ref{prop:bddwavetransform} \eqref{it:W5}, since $WV$ is a bounded projection on $\Lqp$ with image $W\funpqzero$, and the kernel of this projection furnishes a complementary subspace.  Next, note that $VF\in \Sw'(\Rn)$ for each $F\in \Lqp$, by \eqref{eq:inclusionLqp} and Proposition \ref{prop:bddwavetransform} \eqref{it:W3}. Hence the second statement also follows from Proposition \ref{prop:normtransform} and Theorem \ref{thm:FIOconj}.  For the same reasons, $V:\Lqp\to \funpqzero$ is surjective, which implies \eqref{eq:Visom}. The final statement in turn follows from \eqref{eq:Visom}, or alternatively because $W:\funpqzero\to W\funpqzero$ is an isomorphism onto a complemented subspace.
\end{proof}

\subsection{Basic properties}\label{subsec:properties}

Although we will mostly deal with the spaces $\funpqs$ for $p,q\in(1,\infty)$, we show here that the final statement of Corollary \ref{cor:Banach} extends to all $p,q\in[1,\infty)$. 

\begin{proposition}\label{prop:Banach2}
Let $p,q\in[1,\infty)$ and $s\in\R$. Then $\funpqs$ is a Banach space.
\end{proposition}
\begin{proof}
We make a few preliminary observations.

Firstly, since $\lb D\rb^{s}:\funpqs\to\funpqzero$ is an isomorphism, we may suppose throughout that $s=0$.

Next, note that $\funpqzero\subseteq\Sw'(\Rn)$ continuously. Indeed, it follows from integration by parts and Lemma \ref{lem:packetbounds} that the operator $\lb D\rb^{-2n}\ph_{\w}(D)$ has an $L^{1}(\Rn)$ kernel, uniformly in $\w\in S^{n-1}$. In turn, one easily obtains from this that
\[
\funpqzero\subseteq \HT^{-2n,p}(\Rn)\subseteq \Sw'(\Rn)
\]
continuously.

We also claim that, if $(f_{k})_{k=0}^{\infty}$ is a uniformly bounded sequence in $\funpqs$, and if $f\in\Sw'(\Rn)$ is such that $f_{k}\to f$ in $\Sw'(\Rn)$ as $k\to\infty$, then $f\in\funpqs$.  To prove this claim, recall from Lemma \ref{lem:packetbounds} that $\psi_{\w,\sigma}\in C^{\infty}_{c}(\Rn)$ for all $\w\in S^{n-1}$ and $\sigma>0$. 
The same is true for the functions $\phi_{\omega,\sigma}$ defined by 
$$\phi_{\omega,\sigma}(\xi):=\Phi(\sigma \xi)\varphi_{\omega}(\xi) := 
\varphi_{\omega}(\xi) \int \limits _{\sigma|\xi|} ^{\infty} \Psi_{0}(\tau)^{2} \frac{\ud \tau}{\tau} \quad \forall \xi \in \R^{n},$$
for $\sigma \geq 0$ and $\omega \in S^{n-1}$. As a result,  
\[
\Phi(\sigma D)\ph_{\w}(D)f_{k}(x)=\F^{-1}(\phi_{\w,\sigma})\ast f_{k}(x)\to \F^{-1}(\phi_{\w,\sigma})\ast f(x)=\Phi(\sigma D)\ph_{\w}(D)f(x)
\]
for all $x\in \Rn$, where we used the assumption and that this convolution is the action of $f_{k}$ on a shifted Schwartz function. Hence
\[
|\Phi(\sigma D)\ph_{\w}(D)f(x)|=\lim_{k\to\infty}|\Phi(\sigma D)\ph_{\w}(D)f_{k}(x)|\leq \liminf_{k\to\infty}\sup_{\tau>0}|\Phi(\tau D)\ph_{\w}(D)f_{k}(x)|,
\]
and Fatou's Lemma gives
\begin{align*}
\int_{\Rn}\sup_{\sigma>0}|\Phi(\sigma D)\ph_{\w}(D)f(x)|^{p}\ud x&\leq \int_{\Rn}\liminf_{k\to\infty}\sup_{\sigma>0}|\Phi(\sigma D)\ph_{\w}(D)f_{k}(x)|^{p}\ud x&\\
&\leq\liminf_{k\to\infty}\int_{\Rn}\sup_{\sigma>0}|\Phi(\sigma D)\ph_{\w}(D)f_{k}(x)|^{p}\ud x.
\end{align*}
Now, noting that $\Phi(0) = 1$ by \eqref{eq:Psi}, we can use the maximal function characterization of $H^{p}(\Rn)$, together with another application of Fatou's Lemma, to obtain
\begin{align*}
&\int_{S^{n-1}}\|\ph_{\w}(D)f\|_{H^{p}(\Rn)}^{q}\ud \w\eqsim \int_{S^{n-1}}\big\|\sup_{\sigma>0}|\Phi(\sigma D)\ph_{\w}(D)f|\big\|_{L^{p}(\Rn)}^{q}\ud \w\\
&\leq \int_{S^{n-1}}\liminf_{k\to\infty}\big\|\sup_{\sigma>0}|\Phi(\sigma D)\ph_{\w}(D)f_{k}|\big\|_{L^{p}(\Rn)}^{q}\ud \w\\
&\leq\liminf_{k\to\infty}\int_{S^{n-1}}\big\|\sup_{\sigma>0}|\Phi(\sigma D)\ph_{\w}(D)f_{k}|\big\|_{L^{p}(\Rn)}^{q}\ud \w\leq \liminf_{k\to\infty}\|f_{k}\|_{\funpqzero}^{q}<\infty,
\end{align*}
by assumption. One can show in a similar way that $\rho(D)f\in L^{p}(\Rn)$, thereby proving that indeed $f\in\funpqzero$.

After this preliminary work, let $(f_{k})_{k=0}^{\infty}$ be a Cauchy sequence in $\funpqzero$. Then, by what we have just shown, there exists an $f\in\funpqzero$ such that $f_{k}\to f$ in the topology of $\Sw'(\Rn)$. To see that in fact $\|f_{k}-f\|_{\funpqzero}\to0$, one can argue as above. More precisely, for each $k\in\Z_{+}$, Fatou's lemma yields
\begin{align*}
&\int_{S^{n-1}}\big\|\sup_{\sigma>0}|\Phi(\sigma D)\ph_{\w}(D)(f_{k}-f)|\big\|_{L^{p}(\Rn)}^{q}\ud \w\\
&\leq \int_{S^{n-1}}\liminf_{l\to\infty}\big\|\sup_{\sigma>0}|\Phi(\sigma D)\ph_{\w}(D)(f_{k}-f_{l})|\big\|_{L^{p}(\Rn)}^{q}\ud \w\\
&\leq\liminf_{l\to\infty}\int_{S^{n-1}}\big\|\sup_{\sigma>0}|\Phi(\sigma D)\ph_{\w}(D)(f_{k}-f_{l})|\big\|_{L^{p}(\Rn)}^{q}\ud \w\\
&=\liminf_{l\to\infty}\|f_{k}-f_{l}\|_{\funpqzero}^{q}\leq \limsup_{l\to\infty}\|f_{k}-f_{l}\|_{\funpqzero}^{q}.
\end{align*}
Combining this with a similar estimate for the low frequencies, one indeed obtains
\[
\limsup_{k\to\infty}\|f_{k}-f\|_{\funpqzero}\leq \limsup_{k\to\infty}\limsup_{l\to\infty}\|f_{k}-f_{l}\|_{\funpqzero}=0.\qedhere
\]
\end{proof}

Next, we show that the definition of $\funpqs$ is independent of the choice of parabolic frequency localizations. More precisely, let $(\wt{\ph}_{\w})_{\w\in S^{n-1}}\subseteq C^{\infty}(\Rn)$ be a family with the same properties \eqref{it:phiproperties1}, \eqref{it:phiproperties2} and \eqref{it:phiproperties3} that $(\ph_{\w})_{\w\in S^{n-1}}$ has, from Section \ref{subsec:packets}, and let $\wt{\rho}_{0}$ be the associated low-frequency cut-off, from \eqref{eq:wtrho}. We will use the operators $\wt{W}$ and $\wt{V}$ from \eqref{eq:wtW} and \eqref{eq:wtV}.

\begin{proposition}\label{prop:independence}
Let $p,q\in(1,\infty)$ and $s\in\R$. Then there exists a $C>0$ such that the following holds for all $f\in \Sw'(\Rn)$. \\One has $f\in\funpqs$ if and only if $\wt{\rho}(D)f\in L^{p}(\Rn)$, $\wt{\ph}_{\w}(D)f\in \HT^{s,p}(\Rn)$ for almost all $\w\in S^{n-1}$, and $(\int_{S^{n-1}}\|\wt{\ph}_{\w}(D)f\|_{\HT^{s,p}(\Rn)}^{q}\ud\w)^{1/q}<\infty$, in which case
\[
\frac{1}{C}\|f\|_{\funpqs}\leq \|f\|_{\mathcal{L}_{\widetilde{W},s}^{q,p}(\Rn)} \leq C\|f\|_{\funpqs}.
\]

\end{proposition}
\begin{proof}
It suffices to consider the case where $s=0$. 

One has $\wt{V}\wt{W}f=f$ for all $f\in\Sw'(\Rn)$, as follows from Proposition \ref{prop:bddwavetransform} \eqref{it:W5} with $W$ and $V$ replaced by $\wt{W}$ and $\wt{V}$. Hence Corollary \ref{cor:Banach} and Theorem \ref{thm:FIOconj} yield
\begin{align*}
\|f\|_{\funpqs}&\eqsim\|Wf\|_{\Lqp}=\|W\wt{V}\wt{W}f\|_{\Lqp}\\
&\lesssim \|\wt{W}f\|_{\Lqp}
\end{align*}
for all $f\in\Sw'(\Rn)$ for which the final quantity is finite. By Corollary \ref{cor:Banach} with $W$ replaced by $\wt{W}$, this proves one of the required inequalities. The other inequality follows by symmetry.
\end{proof}

To conclude this subsection, we consider the Schwartz functions as a subset of $\funpqs$.

\begin{proposition}\label{prop:density}
Let $p,q\in(1,\infty)$ and $s\in\R$. Then 
\begin{equation}\label{eq:density}
\Sw(\Rn)\subseteq \funpqs\subseteq\Sw'(\Rn)
\end{equation}
continuously, and $\Sw(\Rn)$ lies dense in $\funpqs$.
\end{proposition}
\begin{proof}
It suffices to consider the case where $s=0$. 

Let $f\in\Sw(\Rn)$. By Proposition \ref{prop:bddwavetransform} \eqref{it:W2} and \eqref{eq:inclusionLqp}, one then has $Wf\in \J(\Spp)\subseteq \Lqp$ for all $f\in\Sw(\Rn)$. By Proposition \ref{prop:normtransform}, this in turn implies that $f\in\funpqs$. We have thus proved the first embedding in \eqref{eq:density}.

For the second embedding, recall that every $f\in\funpqzero$ satisfies $f=VWf$, by Proposition \ref{prop:bddwavetransform} \eqref{it:W5}. Hence the required statement follows by combining Proposition \ref{prop:normtransform}, \eqref{eq:inclusionLqp} and Proposition \ref{prop:bddwavetransform} \eqref{it:W2}:
\[
\funpqzero\xrightarrow{W}\Lqp\subseteq\J'(\Spp)\xrightarrow{V}\Sw'(\Rn),
\] 
where the maps and the embedding are all continuous. 

Finally, to see that $\Sw(\Rn)$ in fact lies dense in $\funpqzero$, let $f\in\funpqzero$ be given. Then $Wf\in\Lqp$, by Proposition \ref{prop:normtransform}. Since $\J(\Spp)\subseteq\Lqp$ lies dense, there exists a sequence $(F_{j})_{j=0}^{\infty}\subseteq \J(\Spp)$ such that $F_{j}\to Wf$ in $\Lqp$, as $j\to\infty$. Then, by Proposition \ref{prop:bddwavetransform} and Corollary \ref{cor:Banach}, $(VF_{j})_{j=0}^{\infty}\subseteq \Sw(\Rn)$ and
\[
\|VF_{j}-f\|_{\funpqzero}=\|VF_{j}-VWf\|_{\funpqzero}\lesssim \|F_{j}-Wf\|_{\Lqp}\to0
\]
as $j\to\infty$. 
\end{proof}

\subsection{Interpolation and duality}\label{subsec:interdual}

In this subsection we prove interpolation and duality properties of the function spaces for decoupling.

We first prove that our function spaces form a complex interpolation scale.

\begin{theorem}\label{thm:interpM}
Let $p_{0},p_{1},p,q_{1},q_{2},q\in(1,\infty)$, $s_{0},s_{1},s\in \R$ and $\theta\in[0,1]$ be such that $\frac{1}{p}=\frac{1-\theta}{p_{0}}+\frac{\theta}{p_{1}}$, $\frac{1}{q}=\frac{1-\theta}{q_{0}}+\frac{\theta}{q_{1}}$ and $s=(1-\theta)s_{0}+\theta s_{1}$. Then
\[
[\mathcal{L}_{W,s_{0}}^{q_{0},p_{0}}(\R^{n}),\mathcal{L}_{W,s_{1}}^{q_{1},p_{1}}(\R^{n})]_{\theta}=\funpqs,
\]
with equivalent norms.
\end{theorem}
\begin{proof}
It suffices to consider the case where $s_{0}=s_{1}=s=0$. It is a basic fact about interpolation of vector-valued function spaces (see \cite[Theorem 2.2.6]{HyNeVeWe16}) that
\[
[L^{q_{0}}_{\w}L^{p_{0}}_{x}L^{2}_{\sigma}(\Spp),L^{q_{1}}_{\w}L^{p_{1}}_{x}L^{2}_{\sigma}(\Spp)]_{\theta}=\Lqp.
\]
Hence one can combine Theorem \ref{thm:FIOconj} with a result about interpolation of complemented subspaces (see \cite[Theorem 1.17.1.1]{Triebel78}) to conclude that
\[
[WVL^{q_{0}}_{\w}L^{p_{0}}_{x}L^{2}_{\sigma}(\Spp),WVL^{q_{1}}_{\w}L^{p_{1}}_{x}L^{2}_{\sigma}(\Spp)]_{\theta}=WV\Lqp.
\]
Moreover, by Corollary \ref{cor:Banach}, 
\[
W:\mathcal{L}_{W,0}^{\tilde{q},\tilde{p}}(\R^{n})\to WVL^{\tilde{q}}_{\w}L^{\tilde{p}}_{x}L^{2}_{\sigma}(\Spp)
\]
is an isomorphism for all $\tilde{p},\tilde{q}\in(1,\infty)$, which proves the required statement.
\end{proof}

Next, we show that our function spaces have natural duality properties. 

\begin{theorem}\label{thm:duality}
Let $p,q\in(1,\infty)$ and $s\in\R$. Then
\[
(\funpqs)^{*}=\mathcal{L}_{W,-s}^{q',p'}(\R^{n})
\]
with equivalent norms, where the duality pairing is given by 
\[
\int_{\Spp}Wf(x,\w,\sigma)Wg(x,\w,\sigma)\ud x\ud\w\frac{\ud\sigma}{\sigma}
\]
for $f \in\mathcal{L}_{W,-s}^{q',p'}(\R^{n})$ and $g\in\funpqs$, and by $\lb f,g\rb_{\Rn}$ if $g\in\Sw(\Rn)$.
\end{theorem}
\begin{proof}
Since $\lb D\rb^{s}$ commutes with $W$, it suffices to consider the case where $s=0$.

First let $f\in\mathcal{L}_{W,0}^{q',p'}(\R^{n})$ and $g\in\funpqzero$. Then, by Corollary \ref{cor:Banach}, $Wf\in L^{q'}_{\w}L^{p'}_{x}L^{2}_{\sigma}(\Spp)$ and $Wg\in\Lqp$. Hence H\"{o}lder's inequality yields
\begin{align*}
&\Big|\int_{\Spp}Wf(x,\w,\sigma)Wg(x,\w,\sigma)\ud x\ud\w\frac{\ud\sigma}{\sigma}\Big|\\
&\leq \|Wf\|_{L^{q'}_{\w}L^{p'}_{x}L^{2}_{\sigma}(\Spp)}\|Wg\|_{\Lqp}\eqsim \|f\|_{\mathcal{L}_{W,0}^{q',p'}(\R^{n})}\|g\|_{\funpqzero}.
\end{align*}
Moreover, by Proposition \ref{prop:bddwavetransform}, if $g\in\Sw(\Rn)$ then 
\[
\int_{\Spp}Wf(x,\w,\sigma)Wg(x,\w,\sigma)\ud x\ud\w\frac{\ud\sigma}{\sigma}=\lb Wf,Wg\rb_{\Spp}=\lb f,g\rb_{\Rn}.
\]
 In particular, if $\lb Wf,Wg\rb_{\Spp}=0$ for all $g\in\funpqzero$, then $f=0$. So $\mathcal{L}_{W,0}^{q',p'}(\R^{n})\subseteq (\funpqzero)^{*}$.

Conversely, let $l\in (\funpqzero)^{*}$. Then $l\circ V\in (\Lqp)^{*}$, by Corollary \ref{cor:Banach}.  Moreover, by \cite[Theorems 1.3.10 and 1.3.21]{HyNeVeWe16}, 
\[
(\Lqp)^{*}=L^{q'}_{\w}L^{p'}_{x}L^{2}_{\sigma}(\Spp),
\]
with the natural duality pairing. Hence there exists an $F\in L^{q'}_{\w}L^{p'}_{x}L^{2}_{\sigma}(\Spp)$ such that 
\[
l(VG)=\int_{\Spp}F(x,\w,\sigma)G(x,\w,\sigma)\ud x\ud\w\frac{\ud\sigma}{\sigma}
\]
for all $G\in\Lqp$. Set $f:=VF$. Then $f\in\mathcal{L}_{W,0}^{q',p'}(\R^{n})$, by Corollary \ref{cor:Banach}. Also, Proposition \ref{prop:bddwavetransform} implies that
\[
\lb f,g\rb_{\Rn}=\lb VF,g\rb_{\Rn}=\lb F,Wg\rb_{\Spp}=l(VWg)=l(g)
\]
for all $g\in\Sw(\Rn)$. So Proposition \ref{prop:bddwavetransform} and Corollary \ref{cor:Banach} now yield
\begin{align*}
|\lb Wf,G\rb_{\Spp}|&=|\lb f,VG\rb_{\Rn}|=|\lb F,WVG\rb_{\Spp}|=|l(VWVG)|=|l(VG)|\\
&\leq \|l\| \|VG\|_{\funpqzero}\lesssim \|l\| \|G\|_{\Lqp}
\end{align*}
for all $G\in \J(\Spp)$. This concludes the proof, since Proposition \ref{prop:normtransform} implies that
\[
\|f\|_{\mathcal{L}_{W,0}^{q',p'}(\R^{n})}\eqsim \|Wf\|_{L^{q'}_{\w}L^{p'}_{x}L^{2}_{\sigma}(\Spp)}\eqsim \sup |\lb Wf,G\rb_{\Spp}|,
\]
where the supremum is over all $G\in \J(\Spp)$ with $\|G\|_{\Lqp}\leq 1$.
\end{proof}




\section{Boundedness properties of Fourier integral operators}\label{sec:invariance}

In this section we show that the Euclidean half-wave propagators act boundedly on $\funpqs$, but that if $p\neq q$ then general Fourier integral operators do not.

\subsection{Boundedness of the Euclidean half-wave group}\label{subsec:boundedop}

We first prove that our function spaces are invariant under the action of the Euclidean half-wave propagators. The following theorem contains the first statement of Theorem \ref{thm:invarianceintro}. Recall that $\La(X)$ is the collection of bounded operators on a Banach space $X$.

\begin{theorem}\label{thm:FIObdd}
Let $p,q\in[1,\infty)$ and $s\in\R$. Then there exist $C,N\geq 0$ such that, for all $t\in\R$, one has $e^{it\sqrt{-\Delta}}\in\La(\funpqs)$, with
\[
\|e^{it\sqrt{-\Delta}}f\|_{\La(\funpqs)}\leq C(1+|t|)^{N}.
\]
\end{theorem}
\begin{proof}
Throughout, we may suppose that $s=0$. 

Let $\rho'\in C^{\infty}_{c}(\Rn)$ be such that $\rho'\equiv1$ on $\supp(\rho)$, where $\rho$ is as in \eqref{eq:normdef}. It then follows from straightforward kernel estimates (see also \cite[Theorem 1.18]{DosSantosFerreira-Staubach14}) that
\[
\|\rho'(D)e^{it\sqrt{-\Delta}}\|_{\La(L^{p}(\Rn))}\lesssim (1+|t|)^{N}
\]
for all $t\in\R$ and some $N\geq0$. Hence
\[
\|\rho(D)e^{it\sqrt{-\Delta}}f\|_{L^{p}(\Rn)}=\|\rho'(D)e^{it\sqrt{-\Delta}}\rho(D)f\|_{L^{p}(\Rn)}\lesssim (1+|t|)^{N}\|\rho(D)f\|_{L^{p}(\Rn)}
\]
for all $f\in\funpqzero$, and by \eqref{eq:normdef2} it suffices to prove that
\[
\Big(\int_{S^{n-1}}\|\ph_{\w}(D)e^{it\sqrt{-\Delta}}f\|_{H^{p}_{\w}(\Rn)}^{q}\ud\w\Big)^{1/q}\lesssim (1+|t|)^{N}\Big(\int_{S^{n-1}}\|\ph_{\w}(D)f\|_{H^{p}_{\w}(\Rn)}^{q}\ud\w\Big)^{1/q}.
\]
In turn, we may show that 
\[
\|\ph_{\w}(D)e^{it\sqrt{-\Delta}}f\|_{H^{p}_{\w}(\Rn)}\lesssim (1+|t|)^{N}\|\ph_{\w}(D)f\|_{H^{p}_{\w}(\Rn)},
\]
for an implicit constant independent of $\w\in S^{n-1}$ and $t\in\R$. To this end, we will adapt an argument from \cite[Section 9]{Frey-Portal20} to the setting of anisotropic Hardy spaces. 

For $x\in\Rn$, write
\[
e^{it\sqrt{-\Delta}}f(x)=e^{it(\sqrt{-\Delta}-\w\cdot D)}f(x+t\w).
\]
Since translations do not change the $H^{p}_{\w}(\Rn)$ norm, we only need to show that
\begin{equation}\label{eq:toshowwave}
\|\ph_{\w}(D)e^{it(\sqrt{-\Delta}-\w\cdot D)}f\|_{H^{p}_{\w}(\Rn)}\lesssim (1+|t|)^{N}\|\ph_{\w}(D)f\|_{H^{p}_{\w}(\Rn)}.
\end{equation}
Let $\chi_{\w}\in C^{\infty}(\Rn)$ have the following properties, similar to those of $\ph_{\w}$ from Section \ref{subsec:packets}:
\begin{enumerate}
\item\label{it:chiproperties1} For all $\xi\neq0$, one has $\chi_{\w}(\xi)=0$ if $|\xi|<\frac{1}{8}$ or $|\hat{\xi}-\w|>2|\xi|^{-1/2}$.
\item\label{it:chiproperties2} For all $\alpha\in\Z_{+}^{n}$ and $\beta\in\Z_{+}$, there exists a $C_{\alpha,\beta}\geq0$, independent of $\w$, such that
\[
|(\w\cdot \partial_{\xi})^{\beta}\partial^{\alpha}_{\xi}\chi_{\w}(\xi)|\leq C_{\alpha,\beta}|\xi|^{-\frac{|\alpha|}{2}-\beta}
\]
for all $\xi\neq0$.
\item\label{it:chiproperties3} $\chi_{\w}\equiv 1$ on $\supp(\ph_{\w})$.
\end{enumerate}
Such a function can be constructed in the same manner as $\ph_{\w}$.

We may assume, for simplicity of notation, that $\w=e_{1}$. Set
\[
m(\xi):=\chi_{\w}(\xi)e^{it(|\xi|-\w\cdot \xi)}=\chi_{\w}(\xi)e^{it(\hat{\xi}-\w)\cdot \xi}
\]
for $\xi\neq0$, and $m(0):=0$, so that $m\in C^{\infty}(\Rn)$. We claim that, for all $\alpha\in\Z_{+}^{n}$ and $\beta\in\Z_{+}$, one has
\begin{equation}\label{eq:mdecay}
|(w\cdot\partial_{\xi})^{\beta}\partial_{\xi}^{\alpha}m(\xi)|=|\partial_{\xi_{1}}^{\beta}\partial_{\xi}^{\alpha}m(\xi)|\lesssim (1+|t|)^{N}\lb \xi\rb^{-\frac{|\alpha|}{2}-\beta}
\end{equation}
for $\xi\in\Rn$. To check that this is indeed the case, one can invoke properties \eqref{it:chiproperties1} and \eqref{it:chiproperties2} of $\chi_{\w}$. In particular, to deal with the term involving $\partial_{\xi}^{\alpha}$, one should use the condition that $|\hat{\xi}-\w|\leq 2|\xi|^{-1/2}$ for $\xi\in\supp(\chi_{\w})$. To deal with the term involving $\partial_{\xi_{1}}^{\beta}$, observe additionally that $|\xi_{1}-|\xi||\lesssim 1$ if $\xi\in\supp(\chi_{\w})$, as follows by writing
\[
2|\xi|(|\xi|-\xi_{1})=|\xi_{1}-|\xi||^{2}+\xi_{2}^{2}+\ldots+\xi_{n}^{2}=|\xi-|\xi|\w|^{2}=|\xi|^{2}|\hat{\xi}-\w|^{2}\leq 4|\xi|,
\]
where we again used that $|\hat{\xi}-\w|\leq 2|\xi|^{-1/2}$.

Next, due to property \eqref{it:chiproperties1} of $\chi_{\w}$, we can combine \eqref{eq:mdecay} with Corollary \ref{cor:equivnorm} to see that
\[
|(w\cdot\partial_{\xi})^{\beta}\partial_{\xi}^{\alpha}m(\xi)|\lesssim (1+|t|)^{N}\lb |\xi|_{\w}\rb^{-|\alpha|-2\beta}
\]
for all $\xi\in\Rn$. Now Proposition \ref{prop:multHardy} implies that $m(D)\in\La(H^{p}_{\w}(\Rn))$, with an operator norm which is uniformly bounded in $\w$ and polynomially growing in $t$. Finally, we can use property \eqref{it:chiproperties3} of $\chi_{\w}$ to write
\begin{align*}
\|\ph_{\w}(D)e^{it(\sqrt{-\Delta}-\w\cdot D)}f\|_{H^{p}_{\w}(\Rn)}&=\|m(D)\ph_{\w}(D)f\|_{H^{p}_{\w}(\Rn)}\\
&\lesssim (1+|t|)^{N}\|\ph_{\w}(D)f\|_{H^{p}_{\w}(\Rn)}
\end{align*}
thereby yielding \eqref{eq:toshowwave} and concluding the proof.
\end{proof}

\begin{remark}\label{rem:rate}
It follows from \cite[Proposition 3.2]{Rozendaal-Schippa23} that, for all $p,q\in[1,\infty)$ and $s\in\R$, one has $N\geq 2s(p)$ in Theorem \ref{thm:FIObdd}. More precisely, here one can use that, up to constants which depend on $p$, $q$ and $s$, the $\funpqs$ norm of a function $f\in\Sw(\Rn)$ with $\supp(\wh{f}\,)\subseteq \{\xi\in\Rn\mid |\xi|\leq 1\}$ is equivalent to the $L^{p}(\Rn)$ norm of $f$, as can for example be seen from \eqref{eq:normdef}.   

Moreover, \cite[Proposition 3.1]{Rozendaal-Schippa23} is an analogue of Theorem \ref{thm:FIObdd} for the Besov spaces associated with $\funpqs$, and there one has $N=2s(p)$. In future work, the same bound will be established in Theorem \ref{thm:FIObdd}, at least for $p=q$.  
\end{remark}

\subsection{Unboundedness of Fourier integral operators when $p \neq q$}\label{subsec:unboundedop}

We recall from \cite{HaPoRo20} (see also \cite[Corollary 2.4]{LiRoSoYa24} and Corollary \ref{cor:fracFIO} below) that $\fun^{p,p}(\Rn)=\Hps$ is invariant under any  
Fourier integral operator of order zero as in Definition \ref{def:operator}, if the symbol has compact support in the spatial variable.  In this section we show that, by contrast, for all $q\in[1,\infty)\setminus\{p\}$, the space $\funpqzero$ is in general \emph{not} invariant under such Fourier integral operators.

More precisely, we will give an explicit example of a single operator $T$, essentially a simple change of coordinates that can be expressed as in Definition \ref{def:operator}, which is not bounded on $\mathcal{L}_{W,0}^{q,p}(\R^{2})$ for all $p,q\in(1,\infty)$ with $p\neq q$. Then, for $s\in\R$, the operator $\lb D\rb^{s}T\lb D\rb^{-s}$ has similar properties as $T$, and it is not bounded on $\mathcal{L}_{W,s} ^{q,p}(\R^{2})$. In fact, it will follow from the construction that a similar procedure can be used to show, for any $n\geq2$, that $\funpqs$ is not invariant under general Fourier integral operators of order zero. Moreover, cf.~Remark \ref{rem:wavenotbdd}, the solution operators to variable-coefficient wave equations will also typically not leave $\funpqs$ invariant, unless $p=q$.

For our example, let $\psi: \RR^2 \to \RR^2$ be the diffeomorphism given by
\[
\psi(x_1, x_2):=\Big(\frac{x_2}{1+x_1^2}, x_1 \Big)
\]
for $(x_{1},x_{2})\in\R^{2}$, with inverse 
\begin{equation}\label{eq:phi}
\phi(y_1, y_2):=(y_2, (1+y_2^2) y_1)
\end{equation}
for $(y_{1},y_{2})\in\Rtwo$. Now, we let $T$ be the operation of pullback by $\psi$, followed by multiplication by a smooth cutoff function $k\in C^{\infty}_{c}(\R^{2})$ with $k(x)=1$ for $|x| \leq 4$, and $k(x)=0$ for $|x| \geq 8$. That is, 
$$
Tf(x) := k(x)f(\psi(x)) 
$$
for $f\in\Sw(\Rtwo)$ and $x\in\Rtwo$.
Thus, $T$ is a properly supported FIO of order zero, with Schwartz kernel $K(x,y)=k(x)\delta_{\psi(x)}(y)$, associated to the canonical graph
$$
\{ (x, \xi, y, \eta)\in\R^{8} \mid y = \psi(x), \ \xi = d\psi^t \eta \},
$$
where 
$d\psi^t$ is the transpose of $d\psi$, mapping $T_{x}^* \Rtwo$ to $T_{\psi(x)}^* \Rtwo$. In fact, $T$ is an operator as in Definition \ref{def:operator}, with symbol $k$ and phase function $\Phi(x,\eta)=\psi(x)\cdot\eta$. As such, it is bounded on $\fun^{p,p}(\Rtwo)=\HT^{s,p}_{FIO}(\Rtwo)$ for all $p\in[1,\infty]$ and $s\in\R$.

The following theorem contains the second statement in Theorem \ref{thm:invarianceintro} for $n=2$, with the result for general $n$ following from a similar argument.

\begin{theorem}\label{thm:FIOnotbdd}
 Let $p,q\in[1,\infty)$ be such that $p\neq q$. Then $T$ is not bounded on $\mathcal{L}_{W,0}^{q,p}(\Rtwo)$. 
\end{theorem}



Before getting into the technical details of the proof, we describe the idea. 

Any Fourier integral operator 
maps a wave packet, ``located" at $(x, \omega) \in S^* \RR^n$, to another (slightly distorted) wave packet, located at $\chi(x, \omega)$, where $\chi$ is the map on $S^* \RR^n$ induced by the (homogeneous) canonical relation of the Fourier integral operator. Below, we will choose a function $u$ which is a sum of many wave packets all pointing in the same direction $\omega_0\in S^{1}$. 
We shall see that the $\mathcal{L}_{W,0}^{q,p}(\R^{2})$ norm of $u$ is then an $\ell^{p}$ sum of the $L^p(\R^{2})$ norms of the individual wave packets. 
Thus the growth of the norm in the number of wave packets $N$ will be approximately $N^{1/p}$. On the other hand, we have chosen $T$ such that the image of these wave packets will all have \emph{different} directions. 
As a result, the $\mathcal{L}_{W,0}^{q,p}(\R^{2})$ norm of $Tu$ is an $\ell^{q}$ sum of the $L^{p}(\R^{2})$ norms of each of the translated wave packets, and as such the growth of the norm in $N$ of $Tu$ will be approximately $N^{1/q}$. Letting $N\to\infty$, this suffices for $q < p$.

For $q > p$, 
we merely have to consider the opposite situation. That is, we choose a $u$ which is a sum of many wave packets in different directions and with the property that $Tu$ is a sum of wave packets all pointing in the same direction. For this one can use the inverse of $\psi$ from \eqref{eq:phi}. 

\begin{proof}[Proof of Theorem \ref{thm:FIOnotbdd}]
As already indicated, we will consider the case where $q<p$, with the other case being analogous.

We choose a frequency scale $R>0$, which will be large and, eventually, taken to infinity. Let $\omega_0:=(1,0)\in S^{1}\subseteq\Rtwo$, and let $u_0$ be the wave packet given by
\[
u_0(x) := \frac{1}{(2\pi)^{2}} \int_{\Rtwo} e^{ix \cdot \xi} \rho_{1}\big( R^\gamma |\hat \xi - \omega_0| \big) \Psi_0\big( \tfrac{|\xi|}{R} \big) \, \ud\xi
\]
for $x\in\Rtwo$. Here $\Psi_0$ is as in \eqref{eq:Psi0},  $\rho_{1}\in C^{\infty}_{c}(\R)$ is non-negative, supported in $[0, 2]$ and identically $1$ on $[0,1]$, and $\gamma \in (0, 1/2)$. 
Then $\wh{u_0}\in S^{0}_{1-\gamma,0}$ satisfies symbol estimates of order zero and type $1 - \gamma$, with constants in \eqref{eq:seminorms} independent of $R$, and
\[
\supp(\wh{u_{0}})\subseteq\{\xi\in\R^{2}\mid R\leq |\xi|\leq 2R, |\hat{\xi}-\w_{0}|\lesssim R^{-\gamma}\}.
\]
Note that the angular spread here is wider than for a dyadic-parabolic region. 
This choice is made in order to have a stationary phase expansion in powers of $R$ in \eqref{eq:statphase} below. 


Now fix $\alpha\in(0,\gamma)$, and let $N\in\N$ be such that $\tfrac{1}{2}R^\alpha\leq N\leq 2R^{\alpha}$. We consider a function $u$ that consists of a sum of $N$ translates of $u_0$, along the line where $x_1 = 1$. More precisely, we set
\[
u := \sum_{j=1}^N u_{j,N},
\] 
where
\[
y_{j, N} := \big(1, \tfrac{j}{N} \big) \in \Rtwo\quad\text{and}\quad u_{j, N}(x) := u_0(x - y_{j, N})
\]
for $j\in\{1,\ldots,N\}$ and $x\in\R^{2}$.

\subsubsection{Estimating the $\funpqzero$ norm of $u$}

First note that 
\[
\ph_{\w}(D)u(x)=\frac{1}{(2\pi)^{2}}\sum_{j=1}^N  \int_{\Rtwo} e^{i(x - y_{j,N}) \cdot \xi} \varphi_{\omega}(\xi) \rho_{1}\big( R^\gamma |\hat \xi - \omega_0| \big) \Psi_0\big( \tfrac{|\xi|}{R} \big)  \, \ud\xi 
\]
for $\w\in S^{1}$ and $x\in \R^{2}$. As a function of $\xi$, the product  $\varphi_{\omega}(\xi) \rho_{1}(R^{\gamma}|\hat \xi - \omega_0|)$  is an element of the symbol class $S^{1/4}_{1/2,0,1}$ from \eqref{eq:seminorms2}, with symbol seminorms independent of $R$, 
and with support contained 
where $|\hat \xi - \omega| \lesssim R^{-1/2}$ and $|\omega - \omega_0| \lesssim R^{-\gamma}$. Since $\gamma < 1/2$, there exists a $C_{1}\geq0$ such that this product vanishes if $|\omega - \omega_0| \geq C_{1}R^{-\gamma}$. On the other hand, for $|\w-\omega_{0}|\leq C_{1} R^{-\gamma}$ 
a similar estimate holds as in \eqref{eq:boundspsiinverse}, with $\sigma = R^{-1}$. That is, for each $K\geq0$ one has
\begin{equation}\label{eq:chiphiomega}
|\ph_{\w}(D)u_{j, N}(x)| \lesssim 
R^{7/4} ( 1 + R |x - y_{j, N}|^2_\omega )^{-K} 
\end{equation}
for all $x\in\R^{2}$ and $\w\in S^{n-1}$ with $|\w-\w_{0}|\leq C_{1}R^{-\gamma}$. Just as in the proof of \eqref{eq:boundspsiinverse}, this can be shown through integration by parts, 
with the 
estimate being insensitive to the presence of the factor $\rho_{1}$, because $|\omega - \omega_0| \leq C_{1} R^{-\gamma}$. 

Now, the $\varphi_\omega(D) u_{j,N}$ all decay rapidly in $R^{1/2}|x - y_{j, N}|_{\w}$, due to \eqref{eq:chiphiomega}. Since $\alpha < 1/2$, the $L^p(\Rtwo)$ norm of $\varphi_\omega(D)u_{j,N}$ is thus equal to the $L^p$ norm of $\varphi_\omega(D)u_{j,N}$ restricted to the anisotropic ball $B_{j,N}:=B^{\w}_{cR^{-\alpha}}(y_{j,N})$ from \eqref{eq:anisball}, up to an  error of order $O(R^{-\infty})$. Here $c>0$ is small enough such that all these balls are disjoint, which in turn is possible because each $B_{j,N}$ is contained in the isotropic ball of radius $cR^{-\alpha}$ around $y_{j,N}$, by \eqref{eq:equivnorm2}, and because the spacing between the $y_{j,N}$ is $1/N\eqsim R^{-\alpha}$.

It follows that
\begin{align*}
\| \varphi_\omega(D) u \|_{L^p(\R^2)}^{p}&=\Big\| \sum_{j=1}^N  \varphi_\omega(D) u_{j, N} \Big\|_{L^p(\Rtwo)}^p = \sum_{j=1}^N \big\| \varphi_\omega(D) u_{j, N} \big\|_{L^p(B_{j,N})}^p + \, O(R^{-\infty})\\
&\leq \sum_{j=1}^N \big\| \varphi_\omega(D) u_{j, N} \big\|_{L^p(\Rtwo)}^p + \, O(R^{-\infty}). 
\end{align*}
By \eqref{eq:chiphiomega} and \eqref{eq:volumeball}, this gives a bound
\[
\| \varphi_\omega(D) u \|_{L^p(\Rtwo)} \lesssim \begin{cases} R^{7/4} R^{-3/(2p)} N^{1/p},& \text{if }|\omega - \omega_0| \leq C_{1} R^{-\gamma},\\
0,&\text{if }|\omega - \omega_0| \geq C_{1} R^{-\gamma}. 
\end{cases}
\]
Note that, for $p=1$, the same bound holds with $L^{1}(\Rtwo)$ replaced by $H^{1}(\Rtwo)$, given that $u$ has frequency support in a dyadic annulus. 

We then take the $L^q(S^{1})$ norm in $\omega$ to obtain an upper bound for the 
$\mathcal{L}_{W,0} ^{q,p}(\R^{2})$ norm of $u$, after choosing $R$ large enough so that the low-frequency term $\rho(D)u$ in \eqref{eq:normdef} vanishes:
\begin{equation}\label{eq:upperbd}
\| u \|_{\mathcal{L}_{W,0}^{q,p}(\Rtwo)}\lesssim R^{7/4}  R^{ - 3/(2p)} N^{1/p} R^{-\gamma/q},
\end{equation}
for an implicit constant independent of $N$ and $R$. 

\subsubsection{Stationary phase for $Tu$}

We now consider $v := Tu$, given by 
\[
v(x) = k(x) u(\psi(x)) = \frac{1}{(2\pi)^{2}} k(x) \sum_{j=1}^N \int_{\Rtwo} e^{i(\psi(x) - y_{j,N}) \cdot \xi} \rho_{1}\big( R^\gamma |\hat \xi - \omega_0| \big) \Psi_0\big( \tfrac{|\xi|}{R} \big) \, \ud\xi
\]
for $x\in\Rtwo$.

We want to apply the Fourier multiplier $\varphi_\omega(D)$ to $v$, so we take the Fourier transform of $v$. This is
\[
\F v(\xi) =   \frac{1}{(2\pi)^{2}}  \int_{\Rtwo} e^{-i x \cdot \xi} k(x) \sum_{j=1}^N \int_{\Rtwo} e^{i(\psi(x) - y_{j,N}) \cdot \eta} \rho_{1}\big( R^\gamma |\hat \eta - \omega_0| \big) \Psi_0\big( \tfrac{|\eta|}{R} \big) \ud\eta\ud x
\]
for $\xi\in\Rtwo$. Let $v_{j, N}$ be the $j$-th term in this sum. We scale the frequencies $\xi$ and $\eta$ by a factor of $R$ and write
\begin{equation}\label{eq:statphase}
\F v_{j,N}(R\xi)\!= \!\frac{R^2}{(2\pi)^{2}}\! \int_{\R^{4}}\!e^{iR(- x \cdot \xi + (\psi(x) - y_{j,N}) \cdot \eta)} k(x)  \rho_{1}\big( R^\gamma |\hat \eta - \omega_0| \big) \Psi_0( |\eta| ) \ud\eta\ud x.
\end{equation}
We shall apply the stationary phase expansion, see e.g. Theorem 7.7.5 of \cite{Hormander03}. 

 First we observe that, for each fixed $\xi$, the phase function $f(x, \eta) := - x \cdot \xi + (\psi(x) - y_{j,N}) \cdot \eta$ has a unique critical point at $x = x_{j,N} := \psi^{-1} (y_{j,N})$ and $\xi = d_x \psi(x_{j,N}) \cdot \eta$, which uniquely determines $\eta$ since $\psi$ is a diffeomorphism and therefore $d_x \psi$ is an invertible linear map. Next, the Hessian of $f$ at the critical point is given in block form by 
\[
\begin{pmatrix} d^2_{xx} \psi(x_{j,N}) \cdot \eta & d_x \psi(x_{j,N}) \\ d_x \psi(x_{j,N}) & 0 \end{pmatrix}.
\]
This is clearly nondegenerate, since $d_x \psi(x_{j,N})$ is a nonsingular matrix. 

In applying the stationary phase expansion to \eqref{eq:statphase}, we need to check that we have a valid expansion as $R \to \infty$, i.e.\ that the powers of $R$ tend to $-\infty$, since the amplitude in \eqref{eq:statphase} also depends on $R$. The error estimate in \cite[Theorem 7.7.5]{Hormander03}  after $m$ terms is a multiple of
$$
R^{-m} \sum_{|\alpha| \leq 2m} \sup_{x,\eta\in\Rtwo} \Big| D^\alpha_{x, \eta} \Big( k(x) \rho_{1}\big( R^\gamma |\hat \eta - \omega_0| \big)\Psi_0( |\eta| ) \Big) \Big|,
$$
and this in turn is bounded above by a multiple of $R^{-m(1 - 2\gamma)}$, given that a derivative of the amplitude costs at worst $R^\gamma$. Now, since $\gamma$ was chosen to be strictly less than $1/2$, the power of $R$ in the error estimate tends to $-\infty$, so this is a valid expansion as $R \to \infty$. For exactly the same reason, each term in the expansion, except for the zeroth one, has an overall power of $R$ that is strictly decreasing. 

We aim to prove a lower bound on $\| v \|_{\mathcal{L}_{W,0}^{q,p}(\Rtwo)}$. We claim that it suffices to prove a lower bound after replacing $v$ by its leading term in the stationary phase expansion of \eqref{eq:statphase}. 
In fact, due to symbolic derivative estimates on $\rho_{1}$ and $\Psi_0$, successive terms in this expansion have a similar form, but with decreasing powers of $R$. Each such term can be bounded \emph{above} in the $\mathcal{L}_{W,0}^{q,p}(\Rtwo)$ norm in a similar way as for $u$ above, and contributes less, by a factor of a negative power of $R$, than the leading term. 
On the other hand, the estimate on the remainder after taking $m$ terms in the stationary phase expansion of \eqref{eq:statphase} is bounded by a multiple of $ R^{-m(1 - 2\gamma)}$ for $\xi$ in a compact set, using \cite[equation (7.7.12)]{Hormander03}, or by a multiple of $R^{2-m(1 - \gamma)} \ang{\xi}^{-m}$ when $\ang{\xi}$ is large, using the ``non-stationary phase" estimate of \cite[equation (7.7.1${}^\prime$)]{Hormander03} (noting that $|d_x f| \eqsim \ang{\xi}$ for large $\ang{\xi}$). These pointwise estimates imply a bound on the  $\mathcal{L}_{W,0}^{q,p}(\Rtwo)$ norm of the remainder that decreases faster than any given power of $R$, if $m$ is taken sufficiently large. 
It follows that indeed, if we want to obtain a \emph{lower} bound on $\| v \|_{\mathcal{L}_{W,0}^{q,p}(\Rtwo)}$, valid for sufficiently large $R$, then  (up to a factor of, say, $1/2$) we only need to consider the leading term of the expansion.

\subsubsection{Estimating the $\mathcal{L}_{W,0} ^{q,p}(\R^{2})$ norm of $v$}

We will denote the leading term in the expansion for $v$ by $v_{\ldg}$, and write $v_{\ldg} = \sum_{j=1}^N v_{j, N, \ldg}$ correspondingly. With this notation, and writing $A_{j, N} := (d\psi^{t}(x_{j,N}))^{-1}$ for $j=1,\ldots,N$, we have
\[
\F v_{j, N, \ldg}(\xi) = e^{-i x_{j, N} \cdot \xi}\frac{ k(x_{j,N})  \rho_{1}\big( R^\gamma |\widehat{ A_{j, N} \xi} - \omega_0| \big)    \Psi_0\big( |A_{j, N} \xi | / R\big)}{(\det(f''(x_{j,N},\eta_{j,N})/2\pi i))^{1/2}}
\]
for $\xi\in\Rtwo$. The Fourier transform of $\ph_\omega(D) v_{j, N, \ldg}$  is therefore
\begin{equation}
e^{-i x_{j, N} \cdot \xi} \frac{\ph_\omega(\xi) k(x_{j,N})  \rho_{1}\big( R^\gamma |\widehat{ A_{j, N}  \xi} - \omega_0| \big)    \Psi_0\big( |A_{j, N}  \xi | / R\big)}{(\det(f''(x_{j,N},\eta_{j,N})/i))^{1/2}} .
\label{eq:leadingterm2}\end{equation}
If we write
\[
\w_{j,N}:=\frac{d\psi^{t}(x_{j,N}) \omega_0}{|d\psi^{t}(x_{j,N}) \omega_0|},
\]
then the support of this function is contained in 
\begin{equation}\label{eq:suppvjN}
\big\{ \big| \hat \xi - \omega \big| \lesssim R^{-1/2} \big\} \cap  \big\{ \big| \hat \xi - \w_{j,N} \big| \lesssim R^{-\gamma}  \big\},
\end{equation}
for $R$ sufficiently large. 

We now consider this function of $\xi$, for $j = 1, \dots, N$, assuming that $\omega$ lies in an $R^{-\gamma}$ neighbourhood of one of the directions $\omega_{j,N}$. It is easy to compute that 
$$
d\psi^{t}(x_{j,N}) \omega_0 = d\psi^{t}(x_{j,N}) \begin{pmatrix} 1 \\ 0 \end{pmatrix}  = \frac1{1 + (j/N)^2} \begin{pmatrix} -2j/N \\ 1 \end{pmatrix}
$$
and therefore
$$
\omega_{j,N} = \frac{d\psi(x_{j,N})^{t} \omega_0}{|d\psi(x_{j,N})^{t} \omega_0|} = \frac1{\sqrt{1 + (2j/N)^2}} \begin{pmatrix} -2j/N \\ 1 \end{pmatrix}.
$$
As $j$ varies, these directions $\omega_{j,N}$ are separated by an amount at least $c N^{-1} \eqsim R^{-\alpha} \gg  R^{-\gamma}$, as $\alpha < \gamma$. 
Let $S_{j,N}$ be a neighbourhood of $\omega_{j,N}$ in $S^1$ of size approximately $R^{-\gamma}$. Then for $\omega\in S_{j,N}$, using \eqref{eq:suppvjN} we have $\varphi_\omega(D)  v_{\ldg} = \varphi_\omega(D)  v_{j, N, \ldg}$. 
We seek a lower bound on the $L^p(\Rtwo)$ norm of this function. 
Due to \eqref{eq:leadingterm2}, we can express $\varphi_\omega(D)  v_{j, N, \ldg}(x)$ as $(\det(f''(x_{j,N},\eta_{j,N}))/i)^{-1/2}$ times
$$
 \int_{\Rtwo} e^{i(x-x_{j, N}) \cdot \xi} \ph_\omega(\xi) k(x_{j,N})  \rho_{1}\big( R^\gamma |\widehat{ A_{j, N} \xi} - \omega_0| \big)    \Psi_0\big( |A_{j, N} \xi | / R\big)\ud\xi,
$$
for all $x\in\Rtwo$. 
We notice that the integrand is real and nonnegative, and it is bounded from below by a multiple of $R^{1/4}$  (since $\varphi_\omega$ is a symbol of order $1/4$, cf.~property \eqref{it:phiproperties2} in Section \ref{subsec:packets}) in a rectangle of size approximately $R \times R^{1/2}$ with long axis in direction $\omega$. 
It follows that there is a rectangle in $x$-space centered at $x_{j, N}$, of size approximately $ R^{-1} \times R^{-1/2}$, where 
$|\ph_{\w}(D)v_{j, N, \ldg}(x)|$ is bounded from below by a multiple of $R^{7/4}$. 

In turn, it follows from this that, for each $j = 1, \dots, N$ and $\w\in S_{j,N}$, one has 
$$
\|\ph_\omega(D) v_{\ldg}\|_{L^{p}(\Rtwo)}\gtrsim R^{7/4} R^{-3/(2p)}.
$$
For $p=1$, because $H^{1}(\Rtwo)\subseteq L^{1}(\Rtwo)$, the same holds with $L^{1}(\Rtwo)$ replaced by $H^{1}(\Rtwo)$. We now integrate over $\omega$ to obtain
\begin{equation}\label{eq:lowerbd}
\| v_{\ldg} \|_{\mathcal{L}_{W,0}^{q,p}(\Rtwo)}\gtrsim R^{7/4} R^{-3/(2p)} (N R^{-\gamma})^{1/q},
\end{equation}
where we also used that the measure of $\cup_j S_{j, N}$ is approximately $N R^{-\gamma}$. By the arguments above, the right-hand side of \eqref{eq:lowerbd} is also a lower bound for the $\mathcal{L}_{W,0}^{q,p}(\Rtwo)$ norm of $v$ itself.

\subsubsection{Conclusion}

Comparing \eqref{eq:lowerbd} to \eqref{eq:upperbd}, we see that the powers of $R$ match, but the power of $N$ in \eqref{eq:upperbd} is $1/p$, while the power of $N$ in \eqref{eq:lowerbd} is $1/q$. As $R \to \infty$, we have $N^{1/q} \gg N^{1/p}$ (recall that $N \eqsim R^{\alpha}$).  Given that we assumed that $q<p$, the map $T$ is therefore not bounded on $\fun^{s,p,q}(\Rtwo)$. 
\end{proof}

\begin{remark}\label{rem:wavenotbdd}
The proof above can be tweaked to show that the same phenomenon will occur for any Fourier integral operator 
with a canonical relation which maps points $(x_j, \omega_0)$ in phase space, with varying base point $x_j$ and fixed direction $\omega_0$, to points with varying direction. For a concrete example, consider the Fourier integral operator $T = e^{it\sqrt{L}}$, where $L$ is the self-adjoint operator
$$
L := D_{x_1}^2 + (1+x_1^2) D_{x_2}^2
$$
in two space dimensions. The bicharacteristic flow associated with $\sqrt{L}$ is 
\begin{alignat*}{3}
\dot x_1 &= \frac{\xi_1}{\sqrt{\xi_1^2 + (1 + x_1^2)\xi_2^2}},\qquad \dot x_2 &&= \frac{(1+x_1^2) \xi_2}{\sqrt{\xi_1^2 + (1 + x_1^2)\xi_2^2}},\\
\dot \xi_1 &= \frac{-x_1\xi_2^2}{\sqrt{\xi_1^2 + (1 + x_1^2)\xi_2^2}}, \qquad\dot\xi_{2}&&= 0.
\end{alignat*}
Letting $\omega = \xi_1/\xi_2$, we find that $x_1$ and $\omega$ satisfy the autonomous equations
\[
\dot x_1 = \frac{\omega}{\sqrt{1 + x_1^2 + \omega^2}},\qquad \dot \omega = \frac{-x_1}{\sqrt{1 + x_1^2 + \omega^2}},
\]
showing that $x_1$ and $\omega$ move along circles $\{ x_1^2 + \omega^2 = r \}$ at a speed depending only on the radius $r$ of the circle. It is thus easy to write down the exact solution and verify that for a fixed small positive time $t$, if $\omega(0) = 0$ and $x_1(0) = c$ then $\omega(t) = \omega(t, c)$ varies with $c$. Therefore, we can implement a similar example as above, to conclude that the Fourier integral operator $e^{it\sqrt{L}}$ is not bounded on $\mathcal{L}_{W,0} ^{q,p}(\Rtwo)$  for $t\neq 0$ and $q \neq p$. 

On the other hand, the positive result of Theorem~\ref{thm:FIObdd} extends to FIOs $T$ that are Fourier multipliers. More precisely, if the phase function $\Phi$ is of the form 
\begin{equation}\label{eq:FIO multiplier}
\Phi(x, \eta) = \la x \cdot \eta + \phi(\eta)
\end{equation}
for some $\la\in\R\setminus\{0\}$, a $\phi:\Rn\setminus\{0\}\to\R$ positively homogeneous of degree $1$, and all $(x,\eta)\in \R^{2n}\setminus o$,  then $T$ maps wave packets in a fixed direction $\omega_0$ to wave packets in a (new) fixed direction $\omega_1$, and the conclusion of Theorem \ref{thm:FIObdd} still holds. This is analogous to \cite[Theorem 5.2]{CoNiRo10} which contains a similar property for the modulation spaces $M^{p,q}$. The close relation between the two results can be seen by observing that the conditions on the phase function in \cite[Theorem 5.2]{CoNiRo10}, namely 
$$
\sup_{x,x',\eta} \big| \nabla_x \Phi(x, \eta) - \nabla_x \Phi(x', \eta) \big| <\infty,
$$
together with our assumption that the phase function is homogeneous of degree $1$ in $\eta$, imply that $\Phi$ is an affine function (linear plus constant) in $x$, similar to \eqref{eq:FIO multiplier}. (On the other hand, the conditions imposed on the phase function in \cite{CoNiRo10} are different from our conditions, and the authors of \cite{CoNiRo10} do not require homogeneity of degree $1$ in $\eta$; indeed, as can be perceived from the other assumptions on the phase in \cite[Theorem 5.2]{CoNiRo10}, their primary interest is the case of quadratic growth in $\eta$). 

The same article, \cite{CoNiRo10}, also shows that a phenomenon analogous to Theorem \ref{thm:FIOnotbdd} takes place on the modulation spaces $M^{p,q}$, in the sense that \cite[Proposition 7.1]{CoNiRo10} gives an example of an oscillatory operator which is bounded on $M^{p,p}$ but not on $M^{p,q}$ for any $p\neq q$. Their example is a multiplication operator with a symbol which has a quadratic (in $x$) oscillatory phase, and which is therefore not a Fourier integral operator in the sense of \cite{Hormander71b,Hormander09} and of this article. 
The differences in the conditions on the phase function $\Phi$ between our article and \cite{CoNiRo10} reflect the fact that their spaces and operators are designed for Schr\"odinger-like operators, whereas ours are designed for (half)-wave-like operators. 
In particular, the modulation spaces, apart from being more like Besov spaces than Hardy spaces, involve frequency localization to regions of unit size (reflecting the fact that the multiplier $|\eta|^2$ can only be linearized accurately on unit-size regions of frequency space), whereas the spaces $\funpqs$ localize at frequency level $R$ to dyadic-parabolic regions of dimension $R^{1/2}\times\ldots\times R^{1/2}\times R$ (reflecting the fact that the multiplier $|\eta|$ can be linearized accurately on such regions).  This makes the analysis of the spaces, and the proofs of Theorem \ref{thm:FIOnotbdd} and \cite[Proposition 7.1]{CoNiRo10}, quite different. 
\end{remark}

\section{Embeddings}\label{sec:embeddings}

In this section we consider various embedding properties of $\funpqs$. To do so, we first prove some norm inequalities involving the parabolic frequency localizations from Section \ref{subsec:packets}. We then consider inclusions involving our function spaces and the classical function spaces $\HT^{s,p}(\Rn)$, and finally we prove fractional integration theorems for $\funpqs$.

\subsection{Preliminary results}\label{subsec:preliminaryresults}

Some of the embedding properties involving our function spaces are obvious. For example, by H\"{o}lder's inequality, it is clear that as the parameters $q$ and $s$ grow, the space $\funpqs$ gets smaller. In other words, for all $p,q,v\in[1,\infty)$ and $s,r\in\R$, one has
\begin{equation}\label{eq:Sobolev11}
\funpqs\subseteq 
\mathcal{L}_{W,r} ^{v,p}(\Rn)
\end{equation}
if $s\geq r$ and $q\geq v$.

On the other hand, some of the other embeddings are more subtle, and to prove those we will rely in part on the following norm inequalities involving the functions $(\ph_{\w})_{\w\in S^{n-1}}$ from Section \ref{subsec:packets}.

\begin{proposition}\label{prop:mappingphiw}
Let $1\leq p\leq r<\infty$, and set $s:=\frac{n+1}{2}(\frac{1}{p}-\frac{1}{u})$. Then there exists a $C\geq0$ such that the following statements hold for all $\w\in S^{n-1}$ and $f\in\Sw'(\Rn)$.
\begin{enumerate}
\item\label{it:mappingphiw1} If $\ph_{\w}(D)f\in \HT^{p}(\Rn)$, then 
$\lb D\rb^{-s}\ph_{\w}(D)f\in L^{u}(\Rn)$ and
\[
\|\lb D\rb^{-s}\ph_{\w}(D)f\|_{L^{u}(\Rn)}\leq C\|\ph_{\w}(D)f\|_{\HT^{p}(\Rn)}.
\]
\item\label{it:mappingphiw2} If $f\in H^{p}_{\w}(\Rn)$, then $\lb D\rb^{-\frac{n-1}{4}}\ph_{\w}(D)f\in H^{p}(\Rn)$ and 
\[
\|\lb D\rb^{-\frac{n-1}{4}}\ph_{\w}(D)f\|_{\HT^{p}(\Rn)}\leq C\|f\|_{H^{p}_{\w}(\Rn)}.
\]
\end{enumerate} 
\end{proposition}

\begin{remark} Notice the different numerology compared to the standard fractional integral result from $L^p(\R^n)$ to $L^u(\Rn)$, which would require $s \geq n(1/p - 1/u)$. This can be explained heuristically as follows: because of the parabolic cutoff, the result is ``really'' about anisotropic spaces, in which case $\ang{D}^{-s}$ is anisotropically smoothing of order $2s$, since $|\xi| \eqsim |\xi|_\omega^2$ on the support of $\ph_\omega$ (see \eqref{eq:phisuppomega}). On the other hand, the parabolic norm has dimension $n+1$ instead of $n$ in terms of the volume of balls (see \eqref{eq:volumeball}). A combination of these factors results in the condition on $s$ being as stated in the proposition.  

We also remind the reader that $\ph_\omega$ is a symbol of order $(n-1)/4$, so the multiplier $\lb D\rb^{-\frac{n-1}{4}} \ph_{\w}(D)$ in \eqref{it:mappingphiw2} is a symbol of order zero. 
\end{remark}

\begin{proof}
For \eqref{it:mappingphiw1}, we write
\[
\lb D \rb^{-s} \varphi_{\omega}(D) =  I_{\w,2s} m_{\w}(D) \varphi_{\omega}(D).
\] 
Here $I_{\w,2s}$ is the anisotropic fractional integration operator from \eqref{eq:fracan}, and
\[
m_{\w}(\xi) := |\xi|_{\w}^{2s}(1+|\xi|^2)^{-s/2}\chi_{\w}(\xi) 
\]
for $\xi\in\Rn$ and some $\chi_{\w}\in C^{\infty}(\Rn)$ with similar support properties and behavior under differentiation as $\ph_{\w}$, but such that $\chi_{\w}\equiv 1$ on $\supp(\ph_{\w})$. Then 
\[
\|\lb D\rb^{-s}\ph_{\w}(D)f\|_{L^{u}(\Rn)}\lesssim \|m_{\w}(D)\ph_{\w}(D)f\|_{H^{p}_{\w}(\Rn)},
\]
by Proposition \ref{prop:fracintHardy}.

On the other hand, for all $\alpha\in\Z_{+}^{n}$ and $\beta\in\Z_{+}$ one has
\[
\sup_{\xi\in\Rn} \lb |\xi|_{\w}\rb^{|\alpha|+2\beta}|(\w\cdot \partial_{\xi})^{\beta}\partial_{\xi}^{\alpha} m_{\w}(\xi)|  <\infty,
\]
by \cite[Lemma 1.5]{Calderon-Torchinsky75} and \eqref{eq:phisuppomega} with $\ph_{\w}$ replaced by $\chi_{\w}$. Hence Propositions \ref{prop:multHardy} and \ref{prop:equivpar} yield
\[
\|m_{\w}(D)\ph_{\w}(D)f\|_{H^{p}_{\w}(\Rn)}\lesssim \|\ph_{\w}(D)f\|_{H^{p}_{\w}(\Rn)}\eqsim \|\ph_{\w}(D)f\|_{H^{p}(\Rn)},
\]
as required for \eqref{it:mappingphiw1}.

To prove (\ref{it:mappingphiw2}), one uses Propositions \ref{prop:equivpar} and \ref{prop:multHardy} to observe that
\begin{align*}
\|\lb D \rb^{-\frac{n-1}{4}} \varphi_{\omega}(D) f \|_{\HT^p(\Rn)} 
&\eqsim \|\lb D \rb^{-\frac{n-1}{4}} \varphi_{\omega}(D) f \|_{H^p_{\omega}(\Rn)} \\
&\lesssim \|f \|_{H^p_{\omega}(\Rn)}.
\end{align*}
Here the application of Proposition \ref{prop:multHardy} is allowed since
\[
\sup_{\xi\in\Rn} \lb |\xi|_{\w}\rb^{|\alpha|+2\beta}\big|(\w\cdot \partial_{\xi})^{\beta}\partial_{\xi}^{\alpha} \big(\lb \xi\rb^{-\frac{n-1}{4}} \varphi_{\omega}(\xi)\big)\big| <\infty
\]
for all $\alpha\in\Z_{+}^{n}$ and $\beta\in\Z_{+}$, where we again used \eqref{eq:phisuppomega}.
\end{proof}

\subsection{Sobolev embeddings}\label{subsec:Sobolev}

Recall from Remark \ref{rem:HpFIO} that $\fun^{p,p}(\Rn) = \Hps$ for all $p\in[1,\infty)$ and $s\in\R$. Hence the Sobolev embeddings for $\Hp$ from \eqref{eq:Sobolevintro} immediately yield the following embeddings for $\fun^{p,p}(\Rn)$, for all $p\in[1,\infty)$ and $s\in\R$:
\begin{equation}\label{eq:Sobolevfunpps}
\HT^{s+s(p),p}(\Rn)\subseteq \fun^{p,p}(\Rn)\subseteq \HT^{s-s(p),p}(\Rn).
\end{equation}
The following theorem extends these embeddings to $\funpqs$ for more general values of $q$. It contains the first statement of Theorem \ref{thm:Sobolevintro} as a special case. 

\begin{theorem}\label{thm:Sobolev2}
Let $p\in[1,\infty)$, $q\in(1,\infty)$ and $s\in\R$. Then the following assertions hold.
\begin{enumerate}
\item \label{it:Sobolev21} If $q \leq \max(p,2)$, then
\[
\HT^{s+s(p),p}(\Rn) \subseteq \funpqs.
\]
\item \label{it:Sobolev22}  If $q \geq \min(p,2)$, then
\[
\funpqs \subseteq \HT^{s-s(p),p}(\Rn).
\]
\end{enumerate}
\end{theorem}
\begin{proof}
First note that \eqref{it:Sobolev22} follows by duality from \eqref{it:Sobolev21} if $p\in(1,\infty)$, cf.~Theorem \ref{thm:duality}, and from \eqref{eq:Sobolev11} and \eqref{eq:Sobolevfunpps} if $p=1$. Moreover, in light of \eqref{eq:Sobolev11}, we only need to prove \eqref{it:Sobolev21} when $q = \max(p,2)$. This means we need to establish \eqref{it:Sobolev21} when $2 \leq p < \infty$ and $q = p$, which is already contained in \eqref{eq:Sobolevfunpps}, and to establish \eqref{it:Sobolev21} when $1 < p \leq 2$ and $q = 2$, which we will do in the following. 

Below we prove \eqref{it:Sobolev21} when $p = 1$ and $q = 2$, and then we appeal to interpolation.

\subsubsection{Preliminary work}

Assume without loss of generality that $s = - s(1) = -\frac{n-1}{4}$. Our goal is to show that
\[
\|\rho(D)f\|_{L^1(\Rn)} + \Big( \int_{S^{n-1}} \|\ph_{\w}(D)f\|_{\HT^{-\frac{n-1}{4},1}(\Rn)}^2 \ud\omega \Big)^{1/2} \lesssim \|f\|_{\HT^1(\R^n)}
\]
for every $f\in \HT^{1}(\Rn)$. 
Since $\|\rho(D)f\|_{L^1(\Rn)} \lesssim \|f\|_{L^1(\Rn)} \lesssim \|f\|_{\HT^1(\Rn)}$, we only need to bound  the second term on the left-hand side above.

For $j \in\Z$, let $\Delta_j$ be the Littlewood--Paley projection onto $\{\xi\in\Rn\mid 2^{j-1}\leq |\xi| \leq 2^{j+1}\}$, and for $\omega \in S^{n-1}$, let
\[
T_{\omega,j} := \varphi_{\omega}(D) |D|^{-\frac{n-1}{4}} \Delta_j.
\]
Due to the support properties of $\ph_{\w}$, one has
\begin{align*}
&\|\ph_{\w}(D)f\|_{\HT^{-\frac{n-1}{4},1}(\Rn)}\eqsim \||D|^{-\frac{n-1}{4}}\ph_{\w}(D)f\|_{H^{1}(\Rn)}\\
&\lesssim \sum_{j=-3}^{0} \||D|^{-\frac{n-1}{4}}\ph_{\w}(D) \Delta_j f\|_{L^{1}(\Rn)}+\Big\| \Big( \sum_{j=1}^{\infty} |T_{\omega,j} f|^2 \Big)^{1/2} \Big\|_{L^1(\R^n)}\\
&\lesssim \|f\|_{\HT^{1}(\Rn)} + \Big\| \Big( \sum_{j=1}^{\infty} |T_{\omega,j} f|^2 \Big)^{1/2} \Big\|_{L^1(\R^n)}.
\end{align*}
Hence it suffices to show that
\begin{equation} \label{eq:main_embed}
\Big( \int_{S^{n-1}} \Big\| \Big( \sum_{j=1}^{\infty} |T_{\omega,j} f|^2 \Big)^{1/2} \Big\|_{L^1(\R^n)}^2 \ud\omega \Big)^{1/2} \lesssim \|f\|_{\HT^1(\R^n)}.
\end{equation}

\subsubsection{Proof of \eqref{eq:main_embed}}

Let $f$ be a local Hardy $\HT^1(\Rn)$ atom, $L^{\infty}$ normalized and adapted to a (Euclidean) ball of radius $r$. By translation invariance we may assume that the ball is centered at $0$. First assume $0 < r < 1$, which is the main case. Then $\|f\|_{L^{\infty}(\Rn)} \leq r^{-n}$, and $\int_{\Rn} f(x)\ud x = 0$.  We will show that 
\begin{equation} \label{1}
\Big( \int_{S^{n-1}} \Big\| \Big( \sum_{j \colon 2^j \geq r^{-1}} |T_{\omega,j} f|^2 \Big)^{1/2} \Big\|_{L^1(B_{10\sqrt{r}}^{\w}(0))}^2 \ud\omega \Big)^{1/2} \lesssim 1
\end{equation}
and
\begin{equation} \label{2}
\Big\| \Big( \sum_{j \colon 2^j \geq r^{-1}} |T_{\omega,j} f|^2 \Big)^{1/2} \Big\|_{L^1(\R^n \setminus B_{10 \sqrt{r}}^{\w}(0))} \lesssim 1 \quad \text{uniformly for $\omega \in S^{n-1}$}.
\end{equation}
Here $B_{10\sqrt{r}}^{\w}(0)$ is the anisotropic ball from \eqref{eq:anisball}. 
In addition, we will show that if $1 < 2^j < r^{-1}$, then 
\begin{equation} \label{3}
\|T_{\omega,j} f\|_{L^1(\R^n)} \lesssim 2^j r \quad \text{uniformly for $\omega \in S^{n-1}$},
\end{equation}
which implies
\begin{equation} \label{4}
\Big\| \Big( \sum_{j \colon 1 < 2^j < r^{-1}} |T_{\omega,j} f|^2 \Big)^{1/2} \Big\|_{L^1(\R^n)} \lesssim 1 \quad \text{uniformly for $\omega \in S^{n-1}$}
\end{equation}
(because we can bound the $\ell^2$ norm by the $\ell^1$ norm and apply Fubini). Estimates 
\eqref{1}, \eqref{2} and \eqref{4} will imply our desired estimate \eqref{eq:main_embed}, once we have also obtained the corresponding result for $r\geq 1$. So we now proceed to prove \eqref{1}, \eqref{2} and \eqref{3}.

\subsubsection{Proof of \eqref{1}}

Here we use $L^2$ theory. Cauchy--Schwarz, \eqref{eq:volumeball} and Plancherel give
\begin{align*}
&\Big\| \Big( \sum_{j \colon 2^j \geq r^{-1}} |T_{\omega,j} f|^2 \Big)^{1/2} \Big\|_{L^1(B_{10 \sqrt{r}}^{\w}(0))}\\
& \leq |B_{10 \sqrt{r}}^{\w}(0)|^{1/2} \Big\| \Big( \sum_{j \colon 2^j \geq r^{-1}} |T_{\omega,j} f|^2 \Big)^{1/2} \Big\|_{L^2(\R^n)} \\
& \lesssim r^{\frac{n+1}{4}} \Big( \sum_{j \colon 2^j \geq r^{-1}}  \|\ind_{R_{\omega,j}} \widehat{f}\,\|_{L^2(\R^n)}^2 \Big)^{1/2},
\end{align*}
where $R_{\omega,j}$ is the support of the multiplier for $T_{\omega,j}$; it is a dyadic-parabolic region of size $2^j \times (2^{j/2})^{n-1}$.
Taking the $L^2$ norm over $\omega \in S^{n-1}$, we have
\[
\begin{split}
& \quad \Big( \int_{S^{n-1}} \Big\| \Big( \sum_{j \colon 2^j \geq r^{-1}} |T_{\omega,j} f|^2 \Big)^{1/2} \Big\|_{L^1(B_{10\sqrt{r}}(0,\omega))}^2 \ud\omega \Big)^{1/2} \\
& \lesssim r^{\frac{n+1}{4}} \Big( \sum_{j \colon 2^j \geq r^{-1}}  \|\ind_{R_{\omega,j}} \widehat{f}\|_{L^2(\R^n \times S^{n-1})}^2 \Big)^{1/2} \\
& \lesssim r^{\frac{n+1}{4}} \Big( \sum_{j \colon 2^j \geq r^{-1}}  \|\ind_{|\xi| \eqsim 2^j} \widehat{f}\|_{L^2(\R^n)}^2 2^{-\frac{j}{2}(n-1)} \Big)^{1/2}
\end{split}
\]
where in the last inequality we have used the fact that 
\[
\int_{S^{n-1}} [\ind_{R_{\omega,j}}(\xi)]^2 \ud\omega \lesssim 2^{-\frac{j}{2}(n-1)}.
\]
We may now bound $2^{-\frac{j}{2}(n-1)}$ trivially by $r^{\frac{n-1}{2}}$, and pull it outside the sum in $j$. Thus
\[
\begin{split}
& \Big( \int_{S^{n-1}} \Big\| \Big( \sum_{j \colon 2^j \geq r^{-1}} |T_{\omega,j} f|^2 \Big)^{1/2} \Big\|_{L^1(B_{10\sqrt{r}}^{\w}(0))}^2 \ud\omega \Big)^{1/2} \\
\lesssim & r^{\frac{n+1}{4}} r^{\frac{n-1}{4}} \|f\|_{L^2(\R^n)} \lesssim r^{\frac{n+1}{4}} r^{\frac{n-1}{4}} r^{-\frac{n}{2}} \eqsim 1,
\end{split}
\]
establishing \eqref{1}.

\subsubsection{Proofs of \eqref{2} and \eqref{3}}

Here we use kernel estimates. By rotation symmetry, we will assume $\omega = e_1$ and decompose $\xi = (\xi_1, \xi')$ correspondingly. For $j \geq 1$, let $m_j$ be the multiplier of $T_{\omega,j}$ and $K_j$ be the inverse Fourier transform of $m_j$. Then $m_j$ is supported in $\{|\xi_1| \eqsim 2^j, |\xi'| \lesssim 2^{j/2}\}$, and
\[
\|\partial_{\xi_1}^{\beta} \partial_{\xi'}^{\alpha} m_j\|_{L^{\infty}(\Rn)} \lesssim 2^{-\frac{j}{2} (|\alpha|+2 \beta)}
\]
for all $\alpha\in\Z_{+}^{n-1}$ and $\beta\in\Z_{+}$. It follows that, similarly to \eqref{eq:boundspsiinverse}, 
\begin{equation} \label{5}
|K_j(x)| \lesssim 2^{\frac{j}{2}(n+1)} (1 + 2^{\frac{j}{2}} |x|_\omega)^{-N}
\end{equation}
for all $x\in\Rn$ and $N\geq0$; furthermore,
\begin{equation} \label{6}
|\partial_{x_1}^{\beta} \partial_{x'}^{\alpha} K_j(x)| \lesssim 2^{\frac{j}{2}(n+1+|\alpha|+2\beta)} (1 + 2^{\frac{j}{2}} |x|_\omega)^{-N}
\end{equation}
for all $N\geq0$. We will use \eqref{5} in the proof of \eqref{2}, and use \eqref{6} (with $|\alpha| + \beta = 1$) in the proof of \eqref{3}.

To continue with the proof of \eqref{2}, from \eqref{5}, we have
\[
\Big(\sum_{j \colon 2^j \geq r^{-1}} |K_j(x)|^2 \Big)^{1/2} \lesssim \Big( \sum_{j \colon 2^j \geq r^{-1}} 2^{j(n+1)} (2^{\frac{j}{2}} |x|_\omega)^{-2N} \Big)^{1/2} \lesssim r^{-\frac{n+1}{2}} (r^{-\frac{1}{2}} |x|_\omega)^{-N} 
\]
for all $x \neq0$ if $N > n+1$ (fix one such $N$ from now on). This is useful for $x \notin B_{10 \sqrt{r}}^{\w}(0)$: in that case $|x|_\omega \geq 10\sqrt{r}$ and
\[
\Big( \sum_{j \colon 2^j \geq r^{-1}} |T_{\omega,j} f(x)|^2 \Big)^{1/2}
\leq \Big( \sum_{j \colon 2^j \geq r^{-1}} |f|*|K_j|^2(x) \Big)^{1/2}
\leq r^{-\frac{n+1}{2}} (r^{-\frac{1}{2}} |x|_\omega)^{-N}.
\]
More precisely, the first inequality is Cauchy--Schwarz, using that $\int_{\Rn} |f(x)|\ud x \leq 1$, and the second inequality holds if $|x|_\omega \geq 10\sqrt{r}$, because then \eqref{eq:equivnorm} shows that
\begin{align*}
|f|*\sum_{j \colon 2^j \geq r^{-1}} |K_j|^2(x) &\lesssim r^{-n} \int_{|y| \leq r} r^{-(n+1)} (r^{-\frac{1}{2}} |x-y|_\omega)^{-2N} \ud y \\
&\lesssim r^{-(n+1)} (r^{-\frac{1}{2}} |x|_\omega)^{-2N},
\end{align*}
which upon taking the square root gives the second inequality. 

As a result, the left hand side of \eqref{2} is bounded by a multiple of
\[
\int_{\Rn\setminus B_{10\sqrt{r}}^{\w}(0)} r^{-\frac{n+1}{2}} (r^{-\frac{1}{2}} |x|_\omega)^{-N}  \ud x \lesssim 1,
\]
as desired.

Finally, the proof of \eqref{3} makes additional use of the cancellation condition $\int_{\Rn} f(y)\ud y = 0$. To this end, suppose $1 < 2^j < r^{-1}$. Then 
\[
T_{\omega,j}f(x) = \int_{\R^n} f(y) [K_j(x-y)-K_j(x)] \ud y
\]
for all $x\in\Rn$, because $\int_{\Rn} f(y)\ud y = 0$. As a result,
\[
\|T_{\omega,j}f \|_{L^1(\R^n)} \leq \sup_{|y| \leq r} \int_{\R^n} |K_j(x-y)-K_j(x)| \ud x.
\]
But for $|y| \leq r$, 
\[
\int_{\R^n} |K_j(x-y)-K_j(x)| \ud x \leq \int_0^1 \int_{\R^n} |y \cdot \nabla K_j(x-ty)| \ud x\ud t \leq r \int_{\R^n} |\nabla K_j(x)| \ud x.
\]
From \eqref{6}, we have
\[
\int_{\R^n} |\partial_{x_1} K_j(x)| \ud x \lesssim 2^j,
\]
while
\[
\int_{\R^n} |\partial_{x'} K_j(x)| \ud x \lesssim 2^{\frac{j}{2}} \leq 2^j.
\]
This completes our proof of \eqref{3} when $0<r<1$.

\subsubsection{Proof of \eqref{eq:main_embed} for large balls}

It remains to consider \eqref{eq:main_embed} when $f$ is a local Hardy atom adapted to $B_r(0)$ with $r \geq 1$, which is easier. In that case, we will show that
\begin{equation} \label{7}
\Big\| \Big( \sum_{j=1}^{\infty} |T_{\omega,j} f|^2 \Big)^{1/2} \Big\|_{L^1(\R^n)} \lesssim 1
\end{equation}
uniformly in $\omega \in S^{n-1}$, from which \eqref{eq:main_embed} then follows. To prove \eqref{7}, we split the integral into an integral over $B_{10 r}(0)$ and one over its complement. Then $L^2$ theory gives
\begin{align*}
\Big\| \Big( \sum_{j=1}^{\infty} |T_{\omega,j} f|^2 \Big)^{1/2} \Big\|_{L^1(B_{10r}(0))} &\leq |B_{10r}(0)|^{1/2} \Big\| \Big( \sum_{j=1}^{\infty} |T_{\omega,j} f|^2 \Big)^{1/2} \Big\|_{L^2(\R^n)} \\
&\lesssim r^{\frac{n}{2}} \|f\|_{L^2(\R^n)} \lesssim 1,
\end{align*}
and kernel estimates give
\[
|K_j(x)| \lesssim 2^{\frac{j}{2} (n+1)} (2^j |x|)^{-N}.
\]
More precisely, the last inequality follows from \eqref{5} with $N$ replaced by $2N$ there, because \eqref{eq:equivnorm} implies that $\sqrt{8}|x|_\omega \geq |x|^{1/2}$ if $|x| \geq 2$. It thus follows that, for $x \notin B_{10 r}(0)$ and $N$ large,
\[
\Big( \sum_{j=1}^{\infty} |K_j(x)|^2 \Big)^{1/2} \lesssim |x|^{-N}
\]
and 
\[
\Big( \sum_{j=1}^{\infty} |T_{\omega,j} f|^2 \Big)^{1/2}
\leq \Big( \sum_{j=1}^{\infty} |f|*|K_j|^2 \Big)^{1/2}
\lesssim |x|^{-N}.
\]
This shows that
\[
\Big\| \Big( \sum_{j=1}^{\infty} |T_{\omega,j} f|^2 \Big)^{1/2} \Big\|_{L^1(\R^n \setminus B_{10r}(0))} \lesssim 1
\]
as well, giving finally \eqref{7}. Note we did not need any cancellation condition on $f$ in the above argument when $r\geq 1$.

\subsubsection{Interpolation}

The above argument can be repeated with a complex $z$ with $\Real z = -\frac{n-1}{4}$. The conclusion is that the operator
\[
T_z f(x,\w) := (I-\Delta)^{z/2} \varphi_{\omega}(D)f(x)
\]
is bounded from $H^1(\Rn)$ into $L^2(S^{n-1};H^1(\Rn))$, the bound growing at most polynomially in $\Imag z$. This is because growth in $\Imag z$ arises from differentiating the kernel $K_j$, and the argument only requires a finite number of such derivatives. 
Furthermore, $T_z$ is bounded from $L^2(\Rn)$ to $L^2(S^{n-1};L^2(\Rn))$ when $\textrm{Re}\,z = 0$, by Plancherel, and this bound is independent of the imaginary part of $z$. Also note that if $\frac{1}{p} = \frac{1-\theta}{1} + \frac{\theta}{2}$, for $\theta\in(0,1)$ and $1<p<2$, then we have \[
(H^1(\Rn),L^2(\Rn))_{\theta} = L^p(\Rn)
\]
 (see e.g. \cite[Theorem 2.4.7]{Triebel10}) and 
\[
(L^2(S^{n-1};L^1(\Rn)), L^2(S^{n-1};L^2(\Rn)))_{\theta} = L^2(S^{n-1};L^p(\Rn)),
\]
by e.g. \cite[Theorems 5.1.1 and 5.1.2]{Bergh-Lofstrom76}). Thus by Stein's complex interpolation theorem, $T_{-s(p)}$ is bounded from $L^p(\Rn)$ to $L^2(S^{n-1};L^p(\Rn))$ for $1 < p < 2$. In other words,
\[
\left( \int_{S^{n-1}} \|\varphi_{\omega}(D)f\|_{L^p(\Rn)}^2 \ud\omega \right)^{1/2} \lesssim \|f\|_{W^{s(p),p}(\Rn)}
\]
for $1 < p < 2$. Together with the estimate $\|\rho(D)f\|_{L^p(\Rn)} \lesssim \|f\|_{W^{s(p),p}(\Rn)}$ we obtain the assertion of Theorem \ref{thm:Sobolev2} (\ref{it:Sobolev21}) for $1 < p < 2$.
\end{proof}


\subsection{Fractional integration}\label{subsec:fractint}

In this subsection we consider embeddings between the spaces $\mathcal{L}_{W,s_{1}} ^{q,p_{1}}(\Rn)$ and $\mathcal{L}_{W,s_{2}} ^{q,p_{2}}(\Rn)$.  

More precisely, we consider fractional integration theorems, where one trades in the differentiability parameter $s$ to increase the integrability parameter $p$. In this regard, the classical result says that 
\begin{equation}\label{eq:fracintclas}
\HT^{n(\frac{1}{p}-\frac{1}{r}),p}(\Rn)  \subseteq \HT^{r}(\Rn)
\end{equation}
for all $1\leq p\leq r\leq \infty$ and $s\in\R$ with $p\leq r$. The following is an analog of this embedding involving our function spaces and containing part of Theorem \ref{thm:Sobolevintro}.

\begin{theorem}\label{thm:Sobolev1}
Let $p,q,r\in[1,\infty)$ and $s \in\R$ be such that $p\leq  r$. Then
\begin{equation}\label{eq:fracintdec}
\mathcal{L}_{W,s+\frac{n+1}{2}(\frac{1}{p}-\frac{1}{r})} ^{q,p}(\Rn)
 \subseteq \mathcal{L}_{W,s} ^{q,r}(\Rn).
\end{equation}
\end{theorem}
\begin{proof}[Proof of Theorem \ref{thm:Sobolev1}]
We may suppose that $s=0$. Recall that
\[
\|f\|_{\mathcal{L}_{W,0} ^{q,r}(\Rn)} =\|\rho(D) f\|_{L^{r}(\Rn)} + \left( \int_{S^{n-1}} \|\varphi_{\omega}(D) f\|_{\HT^{r}(\Rn)}^q \ud\omega \right)^{1/q}.
\]
But 
\[
\|\varphi_{\omega}(D) f\|_{\HT^{r}(\Rn)} \lesssim \|\varphi_{\omega}(D) f\|_{\HT^{\frac{n+1}{2}(\frac{1}{p}-\frac{1}{r}),p}(\Rn)}
\]
by Proposition \ref{prop:mappingphiw} \eqref{it:mappingphiw1}, which implies the required statement when one applies \eqref{eq:fracintclas} to $\rho(D)f$. 
\end{proof}

\begin{remark}\label{rem:fracintimproved}
If $p\leq 2\leq r<\infty$, then one can recover \eqref{eq:fracintclas} from Theorem \ref{thm:Sobolev1}, by combining \eqref{eq:fracintdec} with Theorem \ref{thm:Sobolev2}:
\[
\HT^{n(\frac{1}{p}-\frac{1}{r}),p}(\Rn)  \subseteq
\mathcal{L}_{W,n(\frac{1}{p}-\frac{1}{r})-s(p)} ^{2,p}(\Rn)
\subseteq \mathcal{L}_{W,\frac{n-1}{2}(\frac{1}{p}-\frac{1}{r})-s(p)} ^{2,r}(\Rn) \subseteq \HT^{r}(\Rn).
\]
Here we used that $s(p)+s(r)=\frac{n-1}{2}(\frac{1}{p}-\frac{1}{r})$, because $p\leq 2\leq r$. In fact, by the sharpness of the embeddings in Theorem \ref{thm:Sobolev2}, in this case \eqref{eq:fracintdec} is a strict improvement of \eqref{eq:fracintclas}, at least if $(p,r)\neq (2,2)$. On the other hand, for $p,r\leq 2$ or $p,r\geq 2$, again due to the sharpness of the embeddings in Theorem \ref{thm:Sobolev2}, \eqref{eq:fracintclas} neither follows from \eqref{eq:fracintdec}, nor the other way around.
\end{remark}

Although Theorem \ref{thm:Sobolev1} only yields a strict improvement of \eqref{eq:fracintclas} for $p\leq 2\leq r$, for general $p\leq r$ we can nonetheless improve a classical result about the mapping properties of Fourier integral operators. Recall the definition of Fourier integral operators in standard form, from Definition \ref{def:operator}.

\begin{corollary}\label{cor:fracFIO}
Let $m\in\R$, and let $T$ be one of the following:
\begin{enumerate}
\item\label{it:fracFIO1} a Fourier integral operator of order $m$ and type $(1/2,1/2,1)$ in standard form, the symbol of which has compact support in the spatial variable; 
\item\label{it:fracFIO2} a Fourier integral operator of order $m$ and type $(1/2,1/2,1)$ in standard form, associated with a global canonical graph;
\item\label{it:fracFIO3} a compactly supported Fourier integral operator of order $m$ and type $(1,0)$, associated with a local canonical graph.
\end{enumerate}
Let $1\leq p\leq q<\infty$ and $s \in\R$. Then 
\begin{equation}\label{eq:fracFIO}
T:
\mathcal{L}_{W,s+m+\frac{n+1}{2}(\frac{1}{p}-\frac{1}{q})} ^{q,p}(\Rn) 
\to\funpqs
\end{equation}
is bounded. In particular, suppose that one of the following conditions holds:
\begin{enumerate}[(a)]
\item\label{it:fracFIO4} $1 \leq p \leq q \leq 2$ and $m=\frac{1}{q}-\frac{1}{p}-2s(p)$;
\item\label{it:fracFIO5} $2 \leq p \leq q < \infty$ and $m=\frac{1}{q}-\frac{1}{p}-2s(q)$;
\item\label{it:fracFIO6} $1\leq p\leq 2 \leq q < \infty$ and $m=n(\frac{1}{q}-\frac{1}{p})$.
\end{enumerate}
Then 
\begin{equation}\label{eq:fracFIO2}
T:\HT^p(\Rn)\to L^q(\Rn)
\end{equation}
is bounded.
\end{corollary}
\begin{proof}
One obtains \eqref{eq:fracFIO} by combining Theorem \ref{thm:Sobolev1} with the boundedness of $T$ from $\mathcal{L}_{W,s+m} ^{q,q}(\Rn)=\HT^{s+m,q}_{FIO}(\Rn)$ to $\mathcal{L}_{W,s} ^{q,q}(\Rn)=\HT^{s,q}_{FIO}(\Rn)$. For $m=s=0$, this invariance is contained in \cite[Theorem 6.10]{HaPoRo20} in cases \eqref{it:fracFIO2} and \eqref{it:fracFIO3}, and the techniques used there also allow one to treat operators as in \eqref{it:fracFIO1}. For general $m,s\in\R$, the mapping property $T:\HT^{s+m,q}_{FIO}(\Rn)\to\HT^{s,q}_{FIO}(\Rn)$ can be found in \cite[Proposition 3.3 and Corollary 3.4]{LiRoSoYa24} in cases \eqref{it:fracFIO1} and \eqref{it:fracFIO3}. In case \eqref{it:fracFIO2} it follows from \cite[Proposition 3.3]{Rozendaal22b}, after precomposing with the operator $\lb D\rb^{-m}$.

In turn, \eqref{eq:fracFIO2} follows from \eqref{eq:fracFIO}, together with Sobolev embeddings. Indeed, in cases \eqref{it:fracFIO4} and \eqref{it:fracFIO5}, Theorem \ref{thm:Sobolev2} implies that 
\[
\HT^{p}(\Rn)\subseteq \mathcal{L}_{W,-s(p)} ^{q,p}(\Rn)=
 \mathcal{L}_{W,s(q)+m+\frac{n+1}{2}(\frac{1}{p}-\frac{1}{q})} ^{q,p}(\Rn).
\]
Hence \eqref{eq:fracFIO} and the embedding $\mathcal{L}_{W,s(q)} ^{q,p}(\Rn)\subseteq L^{q}(\Rn)$, again from Theorem \ref{thm:Sobolev2}, yield \eqref{eq:fracFIO2}. On the other hand, for \eqref{it:fracFIO6} one can combine Theorems \ref{thm:Sobolev2} and \ref{thm:Sobolev1} to obtain
\[
\HT^{p}(\Rn)\subseteq \mathcal{L}_{W,-s(p)} ^{2,p}(\Rn)\subseteq \mathcal{L}_{W,-n(\frac{1}{p}-\frac{1}{2})} ^{2,2}(\Rn)
\]
and
\[
\mathcal{L}_{W,-m-n(\frac{1}{p}-\frac{1}{2})} ^{2,2}(\Rn)\subseteq \mathcal{L}_{W,-m-n(\frac{1}{p}-\frac{1}{2})-\frac{n+1}{2}(\frac{1}{2}-\frac{1}{q})} ^{2,q}(\Rn)\subseteq L^{q}(\Rn).
\]
An application of \eqref{eq:fracFIO} in the middle, which in fact boils down to $L^{2}$ theory, thus concludes the proof.
\end{proof}

\begin{remark}\label{rem:fracFIO}
In the case of an operator as in \eqref{it:fracFIO1} with symbol in $S^{m}_{1,0}$, the mapping property in \eqref{eq:fracFIO2} is well known (see e.g.~\cite[Section IX.6.15]{Stein93}), and from this one can derive the same statement for operators as in \eqref{it:fracFIO3}. Corollary \ref{cor:fracFIO} improves upon this result in several ways, since it allows for the larger class of $S^{m}_{1/2,1/2,1}$ symbols, removes the assumption that the symbol has compact spatial support, cf.~\eqref{it:fracFIO2}, and yields stronger estimates through \eqref{eq:fracFIO}, due to the sharpness of the embeddings in Theorem \ref{thm:Sobolev2}. In particular, as follows from the proof of \eqref{eq:fracFIO2} in case \eqref{it:fracFIO6} , Corollary \ref{cor:fracFIO} extends the improved fractional integration result mentioned in Remark \ref{rem:fracintimproved}, from the identity operator to general Fourier integral operators.
\end{remark}

\begin{remark}\label{rem:fracFIO2}
Corollary \ref{cor:fracFIO} is sharp, in the sense that the exponent $q$ cannot be enlarged (or the parameter $m$ decreased)  for general Fourier integral operators as in \eqref{it:fracFIO1}, \eqref{it:fracFIO2} or \eqref{it:fracFIO3}. In fact, \eqref{eq:fracFIO} implies \eqref{eq:fracFIO2}, and the latter is already sharp, as is noted in \cite[Section IX.6.16]{Stein93}.

Corollary~\ref{cor:fracFIO} also implies that the operator from Section \ref{subsec:unboundedop}, despite not being bounded on $\funpqs$, is bounded from $\mathcal{L}_{W,s+\frac{n+1}{2}(\frac{1}{p}-\frac{1}{q})} ^{q,p}(\Rn)$ to $\funpqs$ if $q>p$. 
\end{remark}

\section{Decoupling}\label{sec:decouple}

In this section we show that the decoupling inequalities for the sphere and the light cone are equivalent to norm bounds for certain functions in $\funpqs$ and $\funpqsone$, respectively. To do so, we first give an equivalent description of the $\funpqs$ norm of functions with frequency support in a dyadic annulus, which in turn shows that the $\funpqs$ norm itself behaves well under decoupling.

\subsection{A discrete description of the $\funpqs$ norm}\label{subsec:discrete}

Throughout this section, for each $R\geq 2$, fix a maximal collection $V_{R}\subseteq S^{n-1}$ of unit vectors such that $|\nu-\nu'|\geq R^{-1/2}$ for all $\nu,\nu'\in V_{R}$. Then $V_{R}$ has approximately $R^{(n-1)/2}$ elements. 
Let $(\chi_{\nu})_{\nu\in V_{R}}\subseteq C^{\infty}(\Rn\setminus\{0\})$ be an associated partition of unity. More precisely, each $\chi_{\nu}$ is positively homogeneous of degree $0$ and satisfies $0\leq \chi_{\nu}\leq 1$ and 
\[
\supp(\chi_{\nu})\subseteq\{\xi\in\Rn\setminus\{0\}\mid |\hat{\xi}-\nu|\leq 2R^{-1/2}\}.
\]
Moreover, $\sum_{\nu\in V_{R}}\chi_{\nu}(\xi)=1$ for all $\xi\neq0$, and for all $\alpha\in\Z_{+}^{n}$ and $\beta\in\Z_{+}$ there exists a $C_{\alpha,\beta}\geq0$ independent of $k$ such that, if $R/2\leq |\xi|\leq 2R$, then
\[
|(\hat{\xi}\cdot\partial_{\xi})^{\beta}\partial_{\xi}^{\alpha}\chi_{\nu}(\xi)|\leq C_{\alpha,\beta}R^{-(|\alpha|/2+\beta)}
\]
for all $\nu\in V_{R}$. Such a collection is straightforward to construct (see e.g.~\cite[Section IX.4]{Stein93}), and the collection $\{\F^{-1}(\chi_{\nu})\mid R\geq 2, \nu\in V_{R}\}$ is uniformly bounded in $L^{1}(\Rn)$. 

The following proposition, which gives an equivalent description of the $\funpqs$ norm of a function with frequency support in a dyadic annulus, is the key tool to relate $\funpqs$ to decoupling inequalities. It contains part of Theorem \ref{thm:decoupleintro}.

\begin{proposition}\label{prop:discrete}
Let $p,q\in[1,\infty)$ and $s \in \R$. Then there exists a $C>0$ such that the following holds. Let $f\in \funpqs$ be such that $\supp(\wh{f}\,)\subseteq\{\xi\in\Rn\mid R/2\leq |\xi|\leq 2R\}$ for some $R\geq 2$. Then
\begin{equation}\label{eq:discretenorm}
\frac{1}{C}\|f\|_{\funpqs}\leq R^{s+\frac{n-1}{2}(\frac{1}{2}-\frac{1}{q})}\Big(\!\sum_{\nu\in V_{R}}\!\|\chi_{\nu}(D)f\|_{L^{p}(\Rn)}^{q}\Big)^{1/q}\!\leq C\|f\|_{\funpqs}
\end{equation}
and
\begin{equation}\label{eq:discretenormdecoupled}
\frac{1}{C}\|f\|_{\funpqs} \leq  \Big(\sum_{\nu\in V_{R}}\|\chi_{\nu}(D)f\|_{\funpqs}^{q}\Big)^{1/q}\leq C\|f\|_{\funpqs}.
\end{equation}
\end{proposition}
\begin{proof}
The proof of \eqref{eq:discretenorm} where $p=q$ is covered by \cite[Proposition 4.1]{Rozendaal22b}, at least for $R=2^{k}$ with $k\in\N$. For general $p$ and $q$, the statement follows from \cite[Proposition 2.4]{Rozendaal-Schippa23}, since $\|f\|_{L^{p}(\Rn)}$ and the Besov space norm $\|f\|_{B^{0}_{p,q}(\Rn)}$ are comparable, which in turn follows from the assumption that $f$ has frequency support in a dyadic annulus. 

To prove \eqref{eq:discretenormdecoupled}, let $g\in \funpqs$ be such that 
\begin{equation}\label{eq:fparloc}
\supp(\wh{g})\subseteq\{\xi\in\Rn\mid R/2\leq |\xi|\leq 2R, |\hat{\xi}-\nu|\leq 2R^{-1/2}\}
\end{equation}
for some $R\geq 2$ and $\nu\in S^{n-1}$. Then $\chi_\w(D) g$ is only nonzero for a fixed finite number of $\w \in V_R$, independent of $R$. Hence \eqref{eq:discretenorm}, the fact that the $\chi_{\w}(D)$ have kernels that are uniformly in $L^{1}(\Rn)$, and H\"{o}lder's inequality yield
\begin{equation}\label{eq:fparloc2}
\begin{aligned}
\|g\|_{\funpqs}&\eqsim R^{s+\frac{n-1}{2}(\frac{1}{2}-\frac{1}{q})}\Big(\sum_{\w\in V_{R}}\|\chi_{\w}(D)g\|_{L^{p}(\Rn)}^{q}\Big)^{1/q}\\
&\lesssim  R^{s+\frac{n-1}{2}(\frac{1}{2}-\frac{1}{q})}\|g\|_{L^{p}(\Rn)}\\
&=R^{s+\frac{n-1}{2}(\frac{1}{2}-\frac{1}{q})}\Big\|\sum_{\w\in V_{R}}\chi_{\w}(D)g\Big\|_{L^{p}(\Rn)}\\
&\lesssim R^{s+\frac{n-1}{2}(\frac{1}{2}-\frac{1}{q})}\Big(\sum_{\w\in V_{R}}\|\chi_{\w}(D)g\|_{L^{p}(\Rn)}^{q}\Big)^{1/q}\eqsim \|g\|_{\funpqs}.
\end{aligned}
\end{equation}
Applying this with $g=\chi_{\nu}(D)f$ and substituting back into \eqref{eq:discretenorm} yields \eqref{eq:discretenormdecoupled}. 
\end{proof}

\subsection{Decoupling for the sphere}\label{subsec:decoupleshere}



For $\ka\geq 1$ and $R\geq 2\ka$, let $S_{R,\ka}$ be a $\ka R^{-1}$-neighborhood of the sphere $S^{n-1}$ in $\Rn$, and let $V_{R}$ be as in the previous subsection. Then $\theta_{\nu}:=\supp(\chi_{\nu})\cap S_{R,\ka}$, for each $\nu\in V_{R}$, is a curved rectangle of dimensions approximately $R^{-1/2}\times \dots\times R^{-1/2}\times R^{-1}$ pointing in the direction of $\nu$. Of course, the exact size of the region depends on $\ka$, but this parameter will be fixed below. Moreover, $S_{R,\ka}=\cup_{\nu\in V_{R}}\theta_{\nu}$, and the $\theta_{\nu}$ have finite overlap.

This observation allows us to formulate the decoupling inequality for the sphere. More precisely, for $p,q\in[2,\infty)$, set
\begin{equation}\label{eq:dpq}
d(p,q) := \begin{cases}
s(q)+s(p)-\frac{1}{p} &\quad \text{if }p \geq \frac{2(n+1)}{n-1},\\
s(q) &\quad \text{if }2 \leq p \leq \frac{2(n+1)}{n-1}.
\end{cases}
\end{equation}
Then the $\ell^{q}$ decoupling inequality for the sphere, from \cite{Bourgain-Demeter15}, is as follows.

\begin{theorem}\label{thm:decouplesphere}
Let $p,q\in[2,\infty)$, $\ka\geq1$ and $\veps>0$. Then there exists a $C\geq0$ such that 
\[
\|f\|_{L^{p}(\Rn)}\leq CR^{d(p,q)+\veps}\Big(\sum_{\nu\in V_{R}}\|\chi_{\nu}(D)f\|_{L^{p}(\Rn)}^{q}\Big)^{1/q}
\]
for all $R\geq 2\ka$ and $f\in \Sw(\Rn)$ with $\supp(\wh{f}\,)\subseteq S_{R,\ka}$.
\end{theorem}

We can reinterpret this decoupling inequality as a bound for the $\funpqs$ norm of functions with frequency support which is highly localized in a radial sense. For $1<p<\infty$, set
\begin{equation}\label{eq:alphap}
\alpha(p) := \begin{cases}
s(p)-\frac{1}{p} &\quad \text{if }p \geq \frac{2(n+1)}{n-1},\\
0 &\quad \text{if }\frac{2(n+1)}{n+3} \leq p \leq \frac{2(n+1)}{n-1},\\
s(p)-\frac{1}{p'} &\quad \text{if }p\leq \frac{2(n+1)}{n+3}.
\end{cases}
\end{equation}
Note that $\alpha(p)=\alpha(p')=d(p,q)-s(q)$ for $p\geq 2$, and that we already encountered $\alpha(p)$ for $p>2$ in \eqref{eq:alphapintro}. The following corollary contains Theorem \ref{thm:decoupleintro} \eqref{it:decoupleintro2}.

\begin{corollary}\label{cor:decouplesphere}
Let $p\in(1,\infty)$, $q\in[2,\infty)$, $s\in\R$, $\ka\geq1$ and $\veps>0$. Then there exists a $C\geq0$ such that the following holds for all $f\in \Sw'(\Rn)$ satisfying $\supp(\wh{f}\,)\subseteq\{\xi\in\Rn\mid R-\ka\leq |\xi|\leq R+\ka\}$ for some $R\geq 2\ka$.
\begin{enumerate}
\item\label{it:decouplesphere1} If $p\geq 2$ and $f\in\funpqs$, then $f\in W^{s-\alpha(p)-\veps,p}(\Rn)$ and 
\begin{equation}\label{eq:decouplesphere}
\|f\|_{W^{s-\alpha(p)-\veps,p}(\Rn)}\leq C\|f\|_{\funpqs}.
\end{equation}
\item\label{it:decouplesphere2} If $p\leq 2$ and $f\in W^{s+\alpha(p)+\veps,p}(\Rn)$, then $f\in \mathcal{L}_{W,s} ^{2,p}(\Rn)$ and 
\[
\|f\|_{\mathcal{L}_{W,s} ^{2,p}(\Rn)}\leq C\|f\|_{W^{s+\alpha(p)+\veps,p}(\Rn)}.
\]
\end{enumerate}
\end{corollary}
\begin{proof}
For \eqref{it:decouplesphere1}, we may assume without loss of generality that $s=\alpha(p)+\veps$. Moreover, by slightly enlarging $\ka$ and by applying Proposition \ref{prop:density}, we may suppose that $f\in\Sw(\Rn)$. 

Now set $f_{R}(y):=R^{-n}f(y/R)$ for $y\in\Rn$, and recall that the $\chi_{\nu}$ are positively homogeneous of degree zero. Then 
two changes of variables, Theorem \ref{thm:decouplesphere} and Proposition \ref{prop:discrete} combine to yield
\begin{align*}
\|f\|_{L^{p}(\Rn)}&=R^{\frac{n}{p'}}\|f_{R}\|_{L^{p}(\Rn)}\lesssim R^{\frac{n}{p'}+d(p,q)+\veps}\Big(\sum_{\nu \in  V_{R}}\|\chi_{\nu}(D)f_{R}\|_{L^{p}(\Rn)}^{q}\Big)^{1/q}\\
&=R^{d(p,q)+\veps}\Big(\sum_{\nu \in  V_{R}}\|\chi_{\nu}(D)f\|_{L^{p}(\Rn)}^{q}\Big)^{1/q}\eqsim \|f\|_{\mathcal{L}_{W,\alpha(p)+\veps} ^{q,p}(\Rn)}.
\end{align*}
This proves \eqref{it:decouplecone1}.

By duality, \eqref{it:decouplesphere2} follows from \eqref{it:decouplesphere1}. More precisely, Theorem \ref{thm:duality} and Proposition \ref{prop:density} yield $\|f\|_{\fun^{2,p}(\Rn)}\eqsim\sup|\lb f,g\rb_{\Rn}|$, where the supremum is taken over all $g\in\Sw(\Rn)$ with $\supp(\wh{g})\subseteq\{\xi\in\Rn\mid R-2\ka\leq |\xi|\leq R+2\ka\}$ and $\|g\|_{
\mathcal{L}_{W,-s} ^{2,p'}(\Rn)}\leq 1$. Hence \eqref{it:decouplesphere1} yields
\begin{align*}
\|f\|_{\fun^{2,p}(\Rn)}&\eqsim\sup|\lb f,g\rb_{\Rn}|\lesssim \sup \|f\|_{W^{s+\alpha(p)+\veps,p}(\Rn)}\|g\|_{W^{-s-\alpha(p')-\veps,p'}(\Rn)}\\
&\lesssim \sup \|f\|_{W^{s+\alpha(p)+\veps,p}(\Rn)}\|g\|_{\mathcal{L}_{W,-s} ^{2,p'}(\Rn)}= \|f\|_{W^{s+\alpha(p)+\veps,p}(\Rn)}.\qedhere
\end{align*}
\end{proof}

\begin{remark}\label{rem:decouplingcompare}
For $p\neq 2$, the estimates in \eqref{it:decouplesphere1} and \eqref{it:decouplesphere2} of Corollary \ref{cor:decouplesphere} improve upon \eqref{it:Sobolev21} and \eqref{it:Sobolev22} of Theorem \ref{thm:Sobolev2}, which hold without any assumptions on the frequency support of $f$
. Indeed, for all $2<p<\infty$ one has 
\[
s(p')-\alpha(p')=s(p)-\alpha(p)=
\begin{cases}
\frac{1}{p} &\quad \text{if }p \geq \frac{2(n+1)}{n-1},\\
s(p) &\quad \text{if }2<p \leq \frac{2(n+1)}{n-1}.
\end{cases}
\]
On the other hand, Corollary \ref{cor:decouplesphere} yields less norm control than in the parabolically localized setting of \eqref{eq:fparloc} and \eqref{eq:fparloc2}.

Also note that Corollary \ref{cor:decouplesphere} is equivalent to Theorem \ref{thm:decouplesphere}, in the sense that the same rescaling argument as above allows one to deduce Theorem \ref{thm:decouplesphere} from \eqref{eq:decouplesphere}.
\end{remark}

\begin{remark}\label{rem:decouplinggeneral}
The general decoupling inequality in \cite{Bourgain-Demeter15}, for compact $C^{2}$ hypersurfaces with positive definite second fundamental form, follows from Theorem \ref{thm:decouplesphere}, by decomposing the surface into a finite number of small pieces. Hence Corollary \ref{cor:decouplesphere} can also be used to reformulate this more general decoupling inequality.
\end{remark}

\subsection{Decoupling for the cone}\label{subsec:decouplecone}

For $\kappa\geq1$ and $R\geq2\ka$, let $W_{R}$ and $(\wt{\chi}_{\w})_{\w\in W_{R}}$ be collections with the same properties as $V_{R}$ and $(\chi_{\nu})_{\nu\in V_{R}}$, from Section \ref{subsec:discrete}, but with $n$ replaced by $n+1$. More precisely, $W_{R}\subseteq S^{n}$ is a maximal collection of unit vectors in $\R^{n+1}$ such that $|\w-\w'|\geq R^{-1/2}$ for all $\w,\w'\in W_{R}$, and $(\wt{\chi}_{\w})_{\w\in W_{R}}\subseteq C^{\infty}(\R^{n+1}\setminus\{0\})$ is an associated partition of unity of functions that are positively homogeneous of degree zero. Note that the Fourier multiplier $\wt{\chi}_{\w}(D)$ with symbol $\wt{\chi}_{\w}$ now acts on functions on $\Rnone$.

Let $\Gamma_{R,\ka}\subseteq \R^{n+1}$ be a $\ka R^{-1}$-neighborhood of the truncated light cone
\[
\{\zeta=(\xi,\tau)\in \R^{n}\times\R\mid 1/2\leq |\xi|=\tau\leq 2\}.
\] 
Then we can formulate the $\ell^{q}$ decoupling inequality for the cone as follows, using the exponent $d(p,q)$ from \eqref{eq:dpq}.

\begin{theorem}\label{thm:decouplecone}
Let $p,q\in[2,\infty)$ and $\ka,\veps>0$. Then there exists a $C\geq0$ such that 
\[
\|g\|_{L^{p}(\R^{n+1})}\leq CR^{d(p,q)+\veps}\Big(\sum_{\w\in W_{R}}\|\wt{\chi}_{\w}(D)g\|_{L^{p}(\R^{n+1})}^{q}\Big)^{1/q} 
\]
for all $R\geq 2\ka$ and $g\in \Sw(\R^{n+1})$ with $\supp(\wh{g})\subseteq \Gamma_{R,\ka}$.
\end{theorem}
\begin{proof}
For $\w\in W_{R}$, set $\theta_{\w}:=\supp(\wt{\chi}_{\w})\cap \Gamma_{R,\ka}$. Then $\Gamma_{R,\ka}=\cup_{\w\in W_{R}}\theta_{\w}$, and the $\theta_{\w}$ have finite overlap. Fix $\w=(\w_{1},\ldots,\w_{n+1})\in W_{R}$ such that $\theta_{\w}\neq \emptyset$, and suppose that only the first and last coordinates of $\w$ are nonzero. We claim that $\theta_{\w}$ is contained in a slab of dimensions approximately $R^{-1}\times R^{-1/2}\times\dots\times R^{-1/2}\times 1$, pointing along the light cone. 
By rotation, the required statement is then equivalent to the standard formulation of the $\ell^{q}$ decoupling inequality for the cone, from \cite{Bourgain-Demeter15}. 

To prove the claim, let $\zeta=(\xi,\tau)\in \theta_{\w}$, with $\xi=(\xi_{1},\ldots,\xi_{n})\in \Rn$. Then, by assumption on $\w$ and because $\zeta\in \supp(\chi_{w})$, one has
\begin{equation}\label{eq:slabbound}
\Big|\frac{\xi_{1}}{|\zeta|}-\w_{1}\Big|^{2}+\Big|\frac{\xi_{2}}{|\zeta|}\Big|^{2}+\ldots+\Big|\frac{\xi_{n}}{|\zeta|}\Big|^{2}\leq |\hat{\zeta}-\w|^{2}\leq \frac{4}{R}.
\end{equation}
Moreover, $|\zeta|\eqsim 1$, since $\zeta\in \Gamma_{R,\ka}$. Hence 
\[
|\xi_{2}|^{2}+\ldots+|\xi_{n}|^{2}\leq  4R^{-1}|\zeta|^{2}\eqsim R^{-1}.
\]
It also follows from \eqref{eq:slabbound} that $\w$ points in the direction of the light cone. By combining all this, one sees that $\theta_{\w}$ is as claimed.
\end{proof}

We will reinterpret this decoupling inequality as a bound for the $\funpqsone$ norm of functions with frequency support near the light cone. More precisely, let
\[
\Gamma:=\{\zeta=(\xi, \tau)\in\R^{n}\times\R\mid |\tau-|\xi||\leq 1\}\subseteq \R^{n+1}
\]
be a thickened version of the full light cone. For $1<p<\infty$, \vanish{let 
\begin{equation}\label{eq:gammapq}
\gamma(p,q) := \begin{cases}
\frac{n-1}{2}(\frac{1}{2}-\frac{1}{p})-\frac{1}{p}+\frac{1}{2q}-\frac{1}{4} &\quad \text{if }p \geq \frac{2(n+1)}{n-1},\\
\frac{1}{2}-\frac{1}{q}&\quad \text{if }\frac{2(n+1)}{n+3}\leq p \leq 2,\\
\frac{1}{q}-\frac{1}{2}&\quad \text{if }2\leq p \leq \frac{2(n+1)}{n-1},\\\frac{n-1}{2}(\frac{1}{p}-\frac{1}{2})+\frac{1}{p}-\frac{1}{q}-\frac{1}{2} &\quad \text{if }p\leq \frac{2(n+1)}{n+3},
\end{cases}
\end{equation}
and note that $\gamma(p,q)=\gamma(p',q')=$}
recall the definition of $\alpha(p)$ from \eqref{eq:alphap}.

\begin{corollary}\label{cor:decouplecone}
Let $p\in(1,\infty)$, $q\in[2,\infty)$, $s\in\R$, $\ka\geq 1$ and $\veps>0$. Then there exists a $C\geq0$ such that the following holds for all $g\in \Sw'(\R^{n+1})$ satisfying
\begin{equation}\label{eq:decouplecone0}
\supp(\wh{g})\subseteq \Gamma\cap \{\zeta=(\xi,\tau)\in\R^{n+1}\mid R/2\leq |\xi|\leq 2R\}
\end{equation}
for some $R\geq 2\ka$.
\begin{enumerate}
\item\label{it:decouplecone1} If $p\geq 2$ and $g\in \funpqsone$, then $g\in W^{s-\alpha(p)+\frac{1}{4}-\frac{1}{2q}-\veps,p}(\R^{n+1})$ and
\begin{equation}\label{eq:decouplecone}
\|g\|_{W^{s-\alpha(p)+\frac{1}{4}-\frac{1}{2q}-\veps,p}(\R^{n+1})}\leq C\|g\|_{\funpqsone}.
\end{equation}
\item\label{it:decouplecone2} If $p\leq 2$ and $g\in W^{s+\alpha(p)+\veps,p}(\R^{n+1})$, then $g\in\fun^{2,p}(\R^{n+1})$ and
\[
\|g\|_{\mathcal{L}_{W,s} ^{2,p}(\Rnone)}\leq C\|g\|_{W^{s+\alpha(p)+\veps,p}(\R^{n+1})}.
\]
\end{enumerate}
\end{corollary}
Note that, as in Remark \ref{rem:fracintimproved} and Corollary \ref{cor:decouplesphere}, the case $q=2$ is special, in the sense that one may use duality to reverse the inequality in \eqref{it:decouplesphere1} and obtain \eqref{it:decouplesphere2}.
\begin{proof}
The proof is analogous to that of Corollary \ref{cor:decouplesphere}. The difference in the exponent in \eqref{eq:decouplecone} compared to \eqref{eq:decouplesphere} arises from applying Proposition \ref{prop:discrete} with $n$ replaced by $n+1$.
\end{proof}

\section{Regularity for wave equations}\label{sec:wave}

In this section we obtain new regularity results for the Euclidean wave equation. We first connect our function spaces to local smoothing estimates, and then we indicate how these estimates can be applied to nonlinear wave equations with initial data outside of $L^{2}$-based Sobolev spaces.

\subsection{Local smoothing}\label{subsec:localsmooth}


The following result connects our function spaces to local smoothing for the wave equation.

\begin{theorem}\label{thm:localsmoothmain}
Let $p\in(2,\infty)$, $q\in[2,\infty)$, $s\in\R$ and $\veps>0$. Then there exists a $C\geq 0$ such that 
\begin{equation}\label{eq:localsmoothmain}
\Big(\int_{0}^{1}\|e^{it\sqrt{-\Delta}}f\|_{W^{s,p}(\Rn)}^{p}\ud t\Big)^{1/p}\leq C\|f\|_{\mathcal{L}_{W,s+\alpha(p)+\veps}^{q,p}}
\end{equation}
for all $f\in \mathcal{L}_{W,s+\alpha(p)+\veps}^{q,p}(\R^{n})$.
\end{theorem}
\begin{proof}
For $p=q$, the statement is contained in \cite[Theorem 4.4]{Rozendaal22b}. For $q=2$, it is a consequence of \cite[Theorem 1.1]{Rozendaal-Schippa23} and the standard embedding $W^{s+\alpha(p)+\veps,p}(\Rn)\subseteq B^{s+\alpha(p)+\veps/2}_{p,p}(\Rn)$. For general $q\geq2$, it follows by combining the latter result with the embedding $\funpqs\subseteq
\mathcal{L}_{W,s} ^{2,p}(\Rn)$ from \eqref{eq:Sobolev11}. 
\end{proof}

\begin{remark}\label{rem:localsmoothcon}
The local smoothing conjecture for the Euclidean wave equation posits that for each $\veps>0$ one has
\begin{equation}\label{eq:localsmoothcon}
\Big(\int_{0}^{1}\|e^{it\sqrt{-\Delta}}f\|_{W^{s,p}(\Rn)}^{p}\ud t\Big)^{1/p}\lesssim \|f\|_{W^{s+\sigma(p)+\veps,p}(\Rn)},
\end{equation}
for an implicit constant independent of $f\in W^{s+\sigma(p)+\veps,p}(\Rn)$. Here $\sigma(p):=0$ for $2<p\leq 2n/(n-1)$, and $\sigma(p):=2s(p)-1/p$ for $p\geq 2n/(n-1)$. By the Sobolev embeddings for $\funpqs$ in Theorem \ref{thm:Sobolev2} \eqref{it:Sobolev21}, \eqref{eq:localsmoothmain} improves upon \eqref{eq:localsmoothcon} for $p\geq 2(n+1)/(n-1)$ and $2\leq q\leq p$. Moreover, by the sharpness of the Sobolev embedding in Theorem \ref{thm:Sobolev2} \eqref{it:Sobolev22}, \eqref{eq:localsmoothmain} in fact yields a strict improvement of \eqref{eq:localsmoothcon} for such $p$ and $q$. On the other hand, for $2<p<2(n+1)/(n-1)$ and $2\leq q\leq p$, \eqref{eq:localsmoothmain} neither follows from \eqref{eq:localsmoothcon}, nor vice versa. In particular, due to the sharpness of the Sobolev embedding in Theorem \ref{thm:Sobolev2} \eqref{it:Sobolev22}, \eqref{eq:localsmoothmain} yields sharper estimates than \eqref{eq:localsmoothcon} for certain initial data. Here it is relevant to note that the exponents $\alpha(p)$ in \eqref{eq:localsmoothmain} and $\sigma(p)$ in \eqref{eq:localsmoothcon} are sharp, for all $2<p<\infty$ and $2\leq q\leq p$. In the case of \eqref{eq:localsmoothmain}, this follows from \eqref{eq:Sobolev11} and \cite[Theorem 5.3]{Rozendaal22b}. 

We also note that, at least when restricted to dyadic frequency annuli, $\funpqs$ is the largest space of initial data for which one can obtain local smoothing when applying the $\ell^{q}$ decoupling inequality in the manner in which it is typically used (for more on this see \cite{Rozendaal22b,Rozendaal-Schippa23}).
\end{remark}

\begin{remark}\label{rem:Strichartz}
By Theorem \ref{thm:Sobolev1}, one has $W^{1/2,2}(\Rn)\subsetneq
\mathcal{L}_{W,0} ^{2,p}(\Rn)$ for $p=2(n+1)/(n-1)$. Given that $\alpha(p)=0$, Theorem \ref{thm:localsmoothmain} therefore almost yields a strict improvement of the classical Strichartz estimate
\begin{equation}\label{eq:clasStrichartz}
\Big(\int_{0}^{1}\|e^{it\sqrt{-\Delta}}f\|_{L^{p}(\Rn)}^{p}\ud t\Big)^{1/p}\leq C\|f\|_{W^{1/2,2}(\Rn)}.
\end{equation}
In fact, by the sharpness of the Sobolev embeddings in Theorem \ref{thm:Sobolev1}, Theorem \ref{thm:localsmoothmain} already complements \eqref{eq:clasStrichartz}. More precisely, for this specific $p$, Theorem \ref{thm:localsmoothmain} yields sharper estimates than \eqref{eq:clasStrichartz} for a large class of initial data, while also allowing for initial data in $L^{p}$-based Sobolev spaces.  
\end{remark}

\begin{remark}\label{rem:spherical}
As in \cite{GhLiRoSo24}, Theorem \ref{thm:localsmoothmain} can be used to obtain improved maximal function estimates and results about pointwise almost everywhere convergence for the Euclidean wave equation, as well as improved bounds for the local spherical maximal function.
\end{remark}

\subsection{Nonlinear wave equations}\label{subsec:nonlinear}

In this subsection we indicate how our results can be applied to nonlinear wave equations with initial data in $\funpqs$ for $p>2$. Such initial data might be referred to as ``slowly decaying", due to the fact that such a function may decay slower at infinity than an $L^{2}(\Rn)$ function does. On the other hand, even for compactly supported initial data, the results presented here show that, assuming additional integrability beyond that of an $L^{2}(\Rn)$ function, one can obtain well-posedness statements for rougher initial data than one obtains from Strichartz estimates (for more on this see \cite[Section 1.3]{LiRoSoYa24}). 

Our results and proofs for nonlinear wave equations are analogous to those in \cite{Rozendaal-Schippa23} (see also \cite{Schippa22,LiRoSoYa24}). In particular, here we only will consider the cubic nonlinear wave equation in dimension $n=2$, although a similar approach can be used in other dimensions and for different nonlinearities.

Consider the Cauchy problem for the cubic nonlinear wave equation on $\R^{2}\times\R$:
\begin{equation}\label{eq:nonlinear}
\begin{cases}(\partial_{t}^{2}-\Delta_{g})u(x,t)=\pm |u(x,t)|^{2}u(t,x),
\\u(x,0)=f(x), \ \partial_{t}u(x,0)=g(x).
\end{cases}
\end{equation}
Our main result concerning \eqref{eq:nonlinear} is as follows.

\begin{theorem}\label{thm:nonlinear}
Let $q\in[2,\infty]$ and $\veps,T>0$. Then \eqref{eq:nonlinear} is quantitatively well posed with initial data space
\begin{equation}\label{eq:initial}
X=(\mathcal{L}_{W,\veps} ^{q,6}(\R^{2})+W^{1/2,2}(\R^{2}))\times(\mathcal{L}_{W,\veps-1} ^{q,6}(\R^{2})\times W^{-1/2,2}(\R^{2}))
\end{equation}
and solution space
\begin{equation}\label{eq:solutionspace}
S_{T}=L^{4}\big([0,T];L^{6}(\R^{2})\big) \cap C\big([0,T];\mathcal{L}_{W,\veps} ^{q,6}(\R^{2}) + W^{1/2,2}(\R^{2})\big).
\end{equation}
Moreover, \eqref{eq:nonlinear} is also quantitatively well posed with initial data space
\begin{equation}\label{eq:initial2}
X=(\mathcal{L}_{W,\veps} ^{q,4}(\R^{2})+W^{3/8,2}(\R^{2}))\times(\mathcal{L}_{W,\veps-1} ^{q,4}(\R^{2})\times W^{-5/8,2}(\R^{2}))
\end{equation}
and solution space
\begin{equation}\label{eq:solutionspace2}
S_{t_{0}}=L^{24/7}\big([0,T];L^{4}(\R^{2})\big) \cap C\big([0,T];\mathcal{L}_{W,\veps} ^{q,4}(\R^{2}) + W^{3/8,2}(\R^{2})\big).
\end{equation}
\end{theorem}

Our notion of quantitative well-posedness is taken from \cite{Bejenaru-Tao06}, and the definition is recalled in the proof below. Via a fixed-point argument, it implies that there exists a $\delta=\delta(t_{0})>0$ such that, if $\|(u_{0},u_{1})\|_{X}<\delta$, then \eqref{eq:nonlinear} has a unique solution $u\in S_{t_{0}}$, and this solution depends analytically on the initial data. Moreover, in Theorem \ref{thm:nonlinear}, for all $(u_{0},u_{1})\in X$ there exists a $t_{0}>0$ such that there is a unique solution $u\in S_{t_{0}}$ to \eqref{eq:nonlinear}.

\begin{proof}
The proof is almost completely analogous to that of \cite[Theorem 1.4]{Rozendaal-Schippa23}, although we rely on Theorems \ref{thm:localsmoothmain} and \ref{thm:FIObdd} instead of results about adapted Besov spaces. We briefly sketch the idea (see also \cite[Section 6.2]{LiRoSoYa24}).

Write \eqref{eq:nonlinear} as
\[
u = L(f,g) + N(u,u,u),
\]
with 
\[
\begin{split}
L(f,g)(t)&:=\cos(t\sqrt{-\Delta_{g}})f+\frac{\sin(t\sqrt{-\Delta_{g}})}{\sqrt{-\Delta_{g}}}g, \\
N(u_{1},u_{2},u_{3})(t)&:=\pm \int_{0}^{t}\frac{\sin((t-s)\sqrt{-\Delta_{g}})}{\sqrt{-\Delta_{g}}}u_{1}(s)\overline{u_{2}(s)}u_{3}(s)\ud s,
\end{split}
\]
for $u_{1},u_{2},u_{3}\in S_{T}$ and $t\in[0,T]$. We then say that \eqref{eq:nonlinear} is quantitatively well posed if
\begin{align}
\label{eq:linearabstract}
\| L(f,g) \|_{S_{T}} &\leq C \| (f,g) \|_{X }, \\
\label{eq:nonlinearabstract}
\| N(u_1,u_2,u_3) \|_{S_{T}} &\leq C \prod_{j=1}^3 \| u_j \|_{S_{T}},
\end{align}
for some $C\geq 0$ independent of $f,g\in X$ and $u_{1},u_{2},u_{3}\in S_{T}$. 

Now, to prove \eqref{eq:linearabstract} one can rely on Theorems \ref{thm:localsmoothmain} and \ref{thm:FIObdd} for initial data in $\mathcal{L}_{W,s} ^{q,p}(\R^{2})$, and homogeneous Strichartz estimates (cf.~\cite{Keel-Tao98}) for initial data in $W^{s,2}(\R^{2})$. On the other hand, the proof of \eqref{eq:nonlinearabstract} relies on H\"{o}lder's inequality and inhomogeneous Strichartz estimates. We refer to \cite{Rozendaal-Schippa23} for details. 
\end{proof}

One can equally well extend other results from \cite{Rozendaal-Schippa23} to initial data in $\funpqs$ spaces. In particular, if $\veps>1/2$ in Theorem \ref{thm:nonlinear}, then one obtains global existence for the defocusing equation with the initial data space in \eqref{eq:initial} and the solution space in \eqref{eq:solutionspace}.

\begin{remark}\label{rem:Besovnonlinear}
The main difference between Theorem \ref{thm:nonlinear} and the results in \cite{Rozendaal-Schippa23} concerns the second inclusion in \eqref{eq:solutionspace} and \eqref{eq:solutionspace2}. Although standard embeddings between Besov and Sobolev spaces, combined with the $\veps$ loss in \eqref{eq:initial} and \eqref{eq:initial2}, imply that the spaces of initial data in \eqref{eq:initial} and \eqref{eq:initial2} are not fundamentally larger than those in \cite{Rozendaal-Schippa23}, or vice versa, $\funpqs$ satisfies sharp embeddings into the $L^{p}$-based Sobolev scale, cf.~Theorem \ref{thm:Sobolev2}. By contrast, the adapted Besov spaces from \cite{Rozendaal-Schippa23} embed in a sharp manner into the standard Besov scale. Hence Theorems \ref{thm:nonlinear} and \ref{thm:Sobolev2} show that $L^{p}$ regularity of the initial data is pointwise preserved in an optimal sense, whereas \cite[Theorem 1.4]{Rozendaal-Schippa23} yields the corresponding statement for Besov regularity.
\end{remark}

\section*{Acknowledgments}

The authors would like to thank the referee for carefully reading the manuscript, and for several useful remarks and suggestions.
Rozendaal would like to thank Yiyu Tang for pointing out various typos and minor mistakes. Yung would like to thank Andreas Seeger for a helpful discussion that inspired the proof of Theorem \ref{thm:Sobolev2}. Yung is partially supported by a Future Fellowship FT200100399 from the Australian Research Council.

\bibliographystyle{plain}
\bibliography{Bibliography}

\end{document}